\documentclass[10pt]{amsart}
\usepackage{amsbsy}
\usepackage{amssymb,amsthm, amsmath, latexsym}
\usepackage{mathrsfs}
\usepackage{color}
\usepackage{graphicx,epsfig,subfigure,psfrag}
\usepackage{color}
\usepackage{enumerate}
 \usepackage[usenames,dvipsnames]{pstricks}

\newtheorem{theorem}{Theorem}[section]
\newtheorem{proposition}{Proposition}[section]

\newtheorem{lemma}{Lemma}[section]

\newtheorem{definition}{Definition}[section]
\newtheorem{remark}{Remark}[section]

\numberwithin{equation}{section} \numberwithin{theorem}{section}
\numberwithin{proposition}{section} \numberwithin{lemma}{section}
\numberwithin{corollary}{section}
\numberwithin{definition}{section} \numberwithin{remark}{section}

\newcommand{\R}{\mathbb{R}}
\newcommand{\C}{{\mathcal C}}
\newcommand{\Ccap}{{\rm Cap}}
\newcommand{\diam}{{\rm diam}}
\newcommand{\D}{{\mathcal D}}
\newcommand{\ol}{\overline}
\newcommand{\Ms}{{\mathbb M}^{2{\times}2}_{\rm sym}}

\author[J.-F. Babadjian]{Jean-Fran\c cois Babadjian}
\author[F. Iurlano]{Flaviana Iurlano}
\author[A. Lemenant]{Antoine Lemenant}

\address[J.-F. Babadjian]{Laboratoire de Math\'ematiques d'Orsay, Univ. Paris-Sud, CNRS, Universit\'e Paris-Saclay, 91405 Orsay, France.}
\email{jean-francois.babadjian@math.u-psud.fr}

\address[F. Iurlano]{Sorbonne Universit\'e, CNRS, Universit\'e de Paris, Laboratoire Jacques-Louis Lions, F-75005 Paris, France}
\email{iurlano@ljll.math.upmc.fr}

\address[A. Lemenant]{Universit\'e Paris Diderot -- Paris 7, CNRS, UMR 7598 Laboratoire Jacques-Louis Lions, Paris, F-75005, France}
\email{lemenant@ljll.univ-paris-diderot.fr}

\date{\today}

\title{Partial regularity for the crack set minimizing the two-dimensional Griffith energy}

\begin{document}

\begin{abstract} In this paper we prove a  $\mathcal C^{1,\alpha}$ regularity result for minimizers of the planar Griffith functional arising from a variational model of brittle fracture. We prove that  any isolated connected component of the crack, the singular set of a minimizer, is locally a  $\mathcal C^{1,\alpha}$ curve outside a set of zero Hausdorff measure.
\end{abstract}

\maketitle

\tableofcontents

 \section{Introduction}

Following the original Griffith theory of brittle fracture \cite{Griffith}, the variational approach introduced in \cite{FM} rests on the competition between a bulk energy, the elastic energy stored in the material, and a dissipation energy which is propositional to the area (the length in 2D) of the crack. In a planar elasticity setting, the Griffith energy  is defined by
$$\mathcal{G}(u,K):=\int_{\Omega\setminus K}\mathbf A e(u):e(u)\, dx + \mathcal{H}^{1}(K),$$
where $\Omega\subset \R^2$, which is bounded and open, stands for the reference configuration of a linearized elastic body, and  $\mathbf A$ is a suitable elasticity tensor. Here, $e(u)=(\nabla u+\nabla u^T)/2$ is the elastic strain, the symmetric gradient of the displacement $u:\Omega\setminus K \to \R^2$ which is defined outside the crack $K \subset \overline \Omega$. This energy functional falls within the framework of free discontinuity problems, and it is defined on pairs function/set
$$(u,K)\in \mathcal{A}(\Omega):=\{K\subset {  \overline \Omega} \text{ is closed and } u \in LD(\Omega' \setminus K)\},$$
where $\Omega' \supset \overline \Omega$ is a bounded open set (see  \eqref{eq:LD} for a precise definition of the space $LD$ of functions of Lebesgue deformation).

Minimizers of the Griffith energy have attracted a lot of attention in the last years. Although very close to its scalar analogue, which is known as the Mumford-Shah functional, the existence of a global minimizer $(u,K)\in \mathcal{A}(\Omega)$ (with a prescribed Dirichlet boundary condition) was proved only very recently in \cite{CC2,CC,CCI,CFI,FS} (see Section~\ref {section_statement} for details). It was also established in the meantime that the crack set $K$ is $\mathcal H^1$-rectifiable and Ahlfors regular.

\medskip

The main result of this paper is the following partial regularity property for the crack $K$. 

\begin{theorem}\label{mainTH} 
Let $\Omega \subset \R^2$ be a bounded and simply connected open set with $\mathcal C^1$ boundary, let $\psi \in W^{1,\infty}(\R^2;\R^2)$ be a boundary data, and let $\mathbf A$ be a fourth order elasticity tensor of the form
$$\mathbf A \xi= \lambda ({\rm tr} \xi) I +2\mu \xi \quad \text{ for all }\xi \in \mathbb M^{2 \times 2}_{\rm sym},$$
where $\mu>0$ and $\lambda+\mu>0$. Let $(u,K)\in \mathcal{A}(\Omega)$ be a solution to the minimization problem
$$\inf\left\{\mathcal G(v,K'): \quad (v,K') \in \mathcal A(\Omega), \, v=\psi \text{ a.e. in }\Omega' \setminus\overline \Omega\right\}.$$
Then for every isolated connected component $\Gamma$ of ${K\cap \Omega}$ there exist $\alpha\in (0,1)$ 
and an exceptional relatively closed set $Z \subset \Gamma$ such that 
 $\mathcal{H}^1(Z)=0$ and $\Gamma \setminus Z$ is locally a $\mathcal C^{1,\alpha}$ curve.
\end{theorem}

\subsection*{Comments about the main result}

The strategy of our approach is inspired by the regularity theory for minimizers of the classical Mumford-Shah functional. However, the presence of the symmetric gradient in the bulk energy term prevents the standard theory from being applied directly. We will explain later the main differences with the classical theory, and how we overcome some of the difficulties in this paper. Before that, let us first list few remarks about the main result.
 
Firstly, it would be desirable to obtain the analogue of Bonnet's result \cite{b} for the Griffith energy, i.e. to prove that each isolated connected components of $K$ is a finite union of curves and to classify the blow-up limits of minimizers. However, this seems difficult since the proof of \cite{b} relies on the monotonicity formula for the Dirichlet energy, which is not known in the case of the elastic energy.
 
Secondly, we emphasize that our proof strongly uses the two-dimensional setting and that it cannot be easily generalized in higher dimensions.  We will describe below the main ideas of the proof, highlighting where the 2D assumption is crucial.

Thirdly, the $\mathcal C^{1,\alpha}$ regularity can be used as a first step in order to get higher regularity of both $u$ and the crack $K$. Indeed, once we know that $K$ is locally the graph of a $\mathcal C^{1,\alpha}$ function, one can write the Euler equation (which is a priori not well justified without any regularity of $K$, even in a week sense), and then apply the result in \cite{KLM}. Assuming that $u$ is moreover bounded, we obtain that $K$ is analytic (see \cite[Corollary 4.11]{KLM} together with the comment just after \cite[Remark 4.12]{KLM}). 
  
Fourthly,  since any connected component of $K$ is automatically uniformly rectifiable (because it is compact, connected, and Ahlfors regular \cite[Theorem 31.5]{d}), it is tempting to think that the exceptional negligible set $Z$ of Theorem \ref{mainTH} could be taken such that ${\rm dim}_{\mathcal H}(Z)<1$. For the classical Mumford-Shah problem this is true and it can be proved using  the uniform rectifiability of $K$. Indeed, this property permits to apply the so-called $\varepsilon$-regularity Theorem in many balls, and not only almost everywhere, as it comes using Carleson measure estimates (see for instance \cite{Ri}). For the Griffith energy, we establish an analogous $\varepsilon$-regularity theorem on any isolated connected piece, which requires a separating property on the initialized ball. Up to our knowledge, such as separating property is not quantitatively controlled by the uniform rectifiability, as it is the case for the flatness and the normalized energy (see \cite[Section 3.2.3]{L3}).

Yet, let us stress that it is not known in general how to control the  connected components of the singular set of a minimizer. Even in the scalar case, this question  is a big issue related to the Mumford-Shah conjecture.  Of course the number of connected components with positive $\mathcal{H}^1$-measure has to be at most countable, but   it seems difficult to exclude the possibility of uncountably many negligible connected components that accumulate to form a set with positive $\mathcal{H}^1$-measure. We could also imagine many small connected components of positive measure that accumulate near a given bigger component. The assumption to consider an isolated connected component in our main theorem rules out these pathological situations. The precise role of this hypothesis will be explained later.

Finally, our main result is stated on an isolated connected component of a general minimizer.  An alternative could be to minimize the Griffith energy under a connectedness constraint, or under a uniform bound on the number of the connected components. Existence and Ahlfors regularity of a minimizer in this class are much easier to obtain, due to Blashke and Go\l\c{a}b Theorems, and our result in this case would imply that the singular set is $\mathcal{C}^{1,\alpha}$ regular $\mathcal{H}^1$-almost  everywhere. Indeed, a careful inspection to our proof reveals that all the competitors that we use preserve the topology of $K$, thus they can still be used under connectedness constraints on the singular set, leading to the same estimates. Moreover, it is quite probable that most of  the results  contained in this paper could be applied to almost minimizers instead of  minimizers (i.e. pairs that minimize the Griffith energy  in all balls of radius $r$ with its own boundary datum, up to an error excess  controlled by some $Cr^{1+\alpha}$ term). For sake of simplicity we decided to treat  in this paper minimizers of the global functional only.

\subsection*{Comments about the proof}
The Griffith energy is  similar to the classical Mumford-Shah energy for some aspects, but it actually necessitates the introduction of new ideas and new techniques. For the classical Mumford-Shah problem, there are two main approaches. The first one, in dimension 2, see \cite{b} (or \cite{dMum}, written a bit differently in the monograph \cite{d}), is of pure variational nature. It was extended in higher dimensions in \cite{l2} with a more complicated geometrical stopping time argument. Alternatively, there is a PDE approach \cite{afp1, afp2} (see also \cite{afp}), valid in any dimension, which consists in working on the Euler-Lagrange equation. However, none of   the aforementioned approaches can be directly applied to the Griffith energy. 

More precisely, while trying to perform the regularity theory for the Griffith energy, one has to face the following main obstacles:

\emph{(i) No Korn inequality.} The well-known Korn inequality in elasticity theory enables one to control the full gradient $\nabla u$ by the symmetric part of the gradient $e(u)$. Unfortunately, it is not valid in the cracked domain $\Omega \setminus K$, due to the possible lack of regularity of $K$ (see \cite{CFI3,F}). Therefore, one has to keep working with the symmetric gradient in all the estimates.

\emph{(ii) No Euler-Lagrange equation.}  A consequence of the failure of the Korn inequality is the lack of the Euler-Lagrange equation. Indeed, while computing the derivative of the Griffith energy with respect to inner variations, i.e.\ by a perturbation of $u$ of the type $u\circ \Phi_t(x)$ where $\Phi_t={\rm id}+t \Phi$, some mixtures of derivatives of $u$ appear  and these are not controlled by the symmetric gradient $e(u)$. Therefore, the so-called ``tilt-estimate'',  which is one of the key ingredients of the method in \cite{afp1,afp2} cannot be used.

\emph{(iii) No coarea formula.} A fondamental tool in calculus of variations and in geometric measure theory is the so-called coarea formula, which enables one to reconstruct the total variation of a scalar function by integrating the perimeter of its level sets. In our setting, on the one hand the displacement $u$ is a vector field, and on the other hand even for each coordinate of $u$ there would be no analogue of this formula with $e(u)$ replacing $\nabla u$.  In the approach of \cite{d} or \cite{l2}, the coarea formula is a crucial ingredient which ensures that, provided the energy of $u$ is very small in some ball, one can use  a suitable level set of $u$ to ``fill the holes" of $K$, with very small length. It permits to reduce to the case where the crack $K$ ``separates'' the ball in two connected components. This is essentially the reason why our regularity result only holds on (isolated) connected components of $K$.

\emph{(iv) No monotonicity  formula for the elastic energy.}  One of the main ingredients to control the energy in \cite{b} and \cite{d} (in dimension 2), is the so-called monotonicity formula, which essentially says that a suitable renormalization of the bulk energy localized in a ball of radius $r$ is a nondecreasing function of $r$. This is not known for the elastic energy, i.e.\ while replacing $\nabla u$ by $e(u)$. 

\emph{(v) No good extension techniques.} To prove any kind of regularity result, one has to create convenient competitors, and the main competitor in dimension 2 is obtained by replacing $K$ in some ball $B$ where it is sufficiently flat, by a segment $S$ which is nearly a diameter. While doing so, and in order to use the minimality of $(u,K)$, one has to define a new function $v$ which coincides with $u$ outside $B$, which belongs to $LD(B\setminus S)$, and whose elastic energy is controlled by that of $u$. Denoting by $C^\pm$ both connected components of $\partial B \setminus S$, the way this is achieved in the standard Mumford-Shah theory (see \cite{d} or \cite{MorelSolimini}) consists in introducing the harmonic extensions of $u|_{C^\pm}$ to $B$ using the Poisson kernel. This provides two new functions $u^\pm \in H^1(B)$, whose Dirichlet energies in the ball $B$ are controlled by that of $u$ on the boundary $\partial B \setminus K$. For the Griffith energy, the same argument cannot be used since there is no natural ``boundary'' elastic energy on $\partial B \setminus K$.

\medskip

Let us now explain the novelty of the paper and how we  obtain a regularity result, in spite of the aforementioned problems. We do not have any hope to solve directly the general problem (i), which would probably be a way to solve all the other ones.  We follow mainly the two-dimensional approach of \cite{d}, for which one has to face the main obstacles (iii)--(v) described above.

Due to the absence of the coarea formula, we cannot control the size of the holes in $K$ at small scales, when $K$ is very flat, as done in \cite{d}. This is a first reason why our theorem restricts to a connected component of $K$ only. There is a second reason related to the decay of the normalized energy by use of a compactness argument (in the spirit of \cite{l3} or \cite{afp1})
in absence of a monotonicity formula for the energy. In this argument, one of our  main tool is the so-called Airy function $w$ associated to a minimizer $u$, which can be constructed only in the two-dimensional case. This function has been already used in \cite{BCL} to prove compactness and  $\Gamma$-convergence results related to the elastic energy, and it is defined through the harmonic conjugate (see Proposition~\ref {prop:airy}). The main property of $w$ is that it is a scalar biharmonic function in $\Omega \setminus K$, which satisfies $|D^2w|=|\mathbf Ae(u)|$. What is important is the fact that $D^2w$ is a full gradient, while $e(u)$ is only a symmetric gradient.  The other interesting fact in terms of boundary conditions, at least under connected assumptions, is the transformation of a Neumann type problem on the displacement $u$ into a Dirichlet problem on the Airy function $w$, which is usually easier to handle.

 We then obtain that, provided $K$ is sufficiently flat in some ball $B(x_0,r)$ and the normalized energy 
 $$\omega(x_0,r):=\frac{1}{r}\int_{B(x_0,r)}\mathbf Ae(u):e(u) \,dx$$ is sufficiently small, we can control the decay of the energy $r \mapsto \omega(x_0,r)$ as $r \to 0$ (see Proposition~\ref {JFprop}). This first decay estimate is proved by contradiction, using a compactness and $\Gamma$-convergence argument on the elastic energy. In this argument, it is crucial the starting point $x_0$ to belong to an isolated connected component of $K$.

The second part of the proof is a decay estimate on the flatness, namely the quantity
$$\beta(x_0,r):=\frac{1}{r}\inf_{L}  \max\left\{ \sup_{x\in K\cap \overline{B}(x_0,r)}{\rm dist }(x,L), \sup_{x\in L\cap \overline{B}(x_0,r)}{\rm dist }(x,K)\right\},$$
where the infimum is taken over all affine lines $L$ passing through $x$,  measuring how far is $K$ from a reference line in $B(x_0,r)$.  This quantity is particularly useful since a decay estimate of the type $\beta(x_0,r)\leq Cr^{\alpha}$ leads to a $\mathcal C^{1,\alpha}$ regularity result on $K$ (see Lemma~\ref {c1estimates}). 

The excess of density, namely
$$ \frac{\mathcal{H}^1(K\cap B(x_0,r))-1}{2r},$$
controls the quantity $\beta(x_0,r)^2$, as a consequence of the Pythagoras inequality (see Lemma ~\ref {pythag}). In order to estimate the excess of density, the standard technique consists in comparing $K$, where $K$ is already known to be very flat in a ball $B(x_0,r)$ (i.e. $\beta(x_0,r)\leq \varepsilon$), with the competitor given by the replacement  of $K$ by a segment $S$ in $B(x_0,r)$. While doing this, one has to define a suitable admissible function $v$ in $B(x_0,r)$ associated to the competitor $S$, that coincides with $u$ outside $B(x_0,r)$ and has an elastic energy controlled by that of $u$. This is where we have to face the problem (v) mentioned earlier. The way we overcome this difficulty is a technical  extension result (see Lemma~\ref {extLem}). Whenever $\beta(x_0,r)+\omega(x_0,r)\leq \varepsilon$ for $\varepsilon$ sufficiently small (depending only on the Ahlfors regularity constant $\theta_0$), one can find a rectangle $U$, such that  
$$\overline B(x_0,r/5)\subset U \subset B(x_0,r),$$
and  a ``wall set'' $\Sigma \subset \partial U$, such that:
$$K \cap \partial U \subset \Sigma \quad \quad \text{ and } \quad \quad \mathcal{H}^1(\Sigma)\leq \eta  r,$$
where $\eta$ is small.
Moreover, if $K'$ is a competitor for $K$ in $U$  (which ``separates''), then there exists a function $v \in LD(U \setminus K')$ such that 
$$u=v \text{ on } \partial U \setminus \Sigma,$$
and
$$\int_{U \setminus K'}  {\mathbf A }e(v):e(v)\, dx \leq \frac{C}{\eta^6} \int_{B(x_0,r)\setminus K}  {\mathbf A }e(u):e(u) \, dx.$$
The main point being that the set $\Sigma \subset U$ where the values of $u$ and $v$ do not match, has very small length, essentially of order $\eta>0$, that can be taken arbitrarily small. The price to pay is a diverging factor as $\eta \to 0$ in the right-hand side of the previous inequality. A similar statement with  $\mathcal{H}^1(\Sigma)\leq r\beta(x_0,r)$ is much easier to prove, and is actually used before as a preliminary construction (see Lemma~\ref {extension}).  We believe Lemma~\ref {extLem} to be one of the most original part of the proof of Theorem~\ref {mainTH}.

With this extension result at hand,  estimating the flatness through the excess of density as described before, and choosing $\eta$ of order $\omega(x_0,r)^{1/6}$, enables one to obtain a decay estimate for the flatness of the type (see Proposition~\ref {main3}),
$$\beta\left(x_0,\frac{r}{50}\right)\leq C\omega(x_0,r)^{1/14}.$$
The previous decay estimate together with the decay of the renormalized energy constitute the main ingredients which lead to the $\mathcal C^{1,\alpha}$ regularity result. 

\subsection*{Organization of the paper} 
The paper is organized as follows. In Section~\ref {section_statement}, we introduce the main notation used throughout the paper, and we precisely define the variational problem of fracture mechanics we are interested in. In Section~\ref {sec3}, we prove our main result, Theorem~\ref {mainTH}, concerning the partial $\mathcal C^{1,\alpha}$-regularity of the isolated connected components of the crack. The proof relies on two fundamental results. The first one, Proposition~\ref {main3} is a flatness estimate in terms of the renormalized bulk energy which is established in Section~\ref {sec5}. The second one, Proposition~\ref {JFprop}, is a bulk energy decay which is proved in Section~\ref {sec6}. Eventually, we gather in the Appendix of Section~\ref {app} several technical results.

\section{Statement of the problem}
\label{section_statement}

\subsection{Notation}

The Lebesgue measure in $\R^n$ is denoted by $\mathcal L^n$, and the $k$-dimensional Hausdorff measure by $\mathcal H^k$. If $E$ is a measurable set, we will sometimes write $|E|$ instead of $\mathcal L^n(E)$. If $a$ and $b \in \R^n$, we write $a \cdot b=\sum_{i=1}^n a_i b_i$ for the Euclidean scalar product, and we denote the norm by $|a|=\sqrt{a \cdot a}$. The open (resp. closed) ball of center $x$ and radius $r$ is denoted by $B(x,r)$ (resp. $\overline B(x,r)$).

\medskip

We write $\mathbb M^{n \times n}$ for the set of real $n \times n$ matrices, and $\mathbb M^{n \times n}_{\rm sym}$ for that of all real symmetric $n \times n$ matrices. Given a matrix $A \in \mathbb M^{n \times n}$, we let $|A|:=\sqrt{{\rm tr}(A A^T)}$ ($A^T$ is the transpose of $A$, and ${\rm tr }A$ is its trace) which defines the usual Frobenius norm over $\mathbb M^{n \times n}$. 

\medskip

Given an open subset $U$ of $\R^n$, we denote by $\mathcal M(U)$ the space of all real valued Radon measures with finite total variation. We use standard notation for Lebesgue spaces $L^p(U)$ and Sobolev spaces $W^{k,p}(U)$ or $H^k(U):=W^{k,2}(U)$. If $K$ is a closed subset of $\R^n$, we denote by $H^k_{0,K}(U)$ the closure of $\C^\infty_c(\overline U \setminus K)$ in $H^k(U)$. In particular, if $K=\partial U$, then $H^k_{0,\partial U}(U)=H^k_0(U)$. 

\medskip

\noindent {\bf Functions of Lebesgue deformation.} Given a vector field (distribution) $u : U \to \R^n$, the symmetrized gradient of $u$ is denoted by 
$$e(u):=\frac{\nabla u + \nabla u^T}{2}.$$
In linearized elasticity, $u$ stands for the displacement, while $e(u)$ is the elastic strain. The elastic energy of a body is given by a quadratic form of $e(u)$, so that it is natural to consider displacements such that $e(u) \in L^2(U;\mathbb M^{n \times n}_{\rm sym})$. If $U$ has Lipschitz boundary, it is well known that $u$ actually belongs to $H^1(U;\R^n)$ as a consequence of the Korn  inequality. However, when $U$ is not smooth,  we can only assert that  $u \in L^2_{\rm loc}(U;\R^n)$. This motivates the following definition of the space of Lebesgue deformation:
\begin{equation}\label{eq:LD}
LD(U):=\{ u \in L^2_{\rm loc}(U;\R^n) : \; e(u) \in L^2(U;\mathbb M^{n \times n}_{\rm sym})\}.
\end{equation}
If $U$ is connected and $u$ is a distribution with $e(u)=0$, then necessarily it is a rigid movement, {\it i.e.}  $u(x)=Ax+b$ for all $x \in U$, for some skew-symmetric matrix $A \in \mathbb M^{n \times n}$ and some vector $b \in \R^n$. If, in addition, $U$ has Lipschitz boundary, the following Poincar\'e-Korn inequality holds: there exists a constant $c_U>0$ and a rigid movement $r_U$ such that
\begin{equation}\label{poincare-korn}
\|u-r_U\|_{L^2(U)}\leq c_U \|e(u)\|_{L^2(U)} \quad \text{for all }u \in LD(U).
\end{equation}
According to \cite[Theorem 5.2, Example 5.3]{AMR}, it is possible to make $r_U$ more explicit in the following way: consider a measurable subset $E$ of $U$ with $|E|>0$, then one can take 
$$r_U(x):=\frac{1}{|E|}\int_E u(y)\, dy + \left( \frac{1}{|E|}\int_E\frac{\nabla u(y) - \nabla u(y)^T}{2}\, dy \right)\left(x-\frac{1}{|E|}\int_E y\, dy \right),$$
provided the constant $c_U$ in \eqref{poincare-korn} also depends on $E$.

\medskip 

\noindent{\bf Hausdorff convergence of compact sets.} Let $K_1$ and $K_2$ be compact subsets of a common compact set $K \subset \R^n$. The Hausdorff distance between $K_1$ and $K_2$ is given by
$$
d_{\mathcal H}(K_1,K_2):=\max\left\{ \sup_{x \in K_1}{\rm dist}(x,K_2), \sup_{y \in K_2}{\rm dist}(y,K_1)\right\}.
$$
We say that a sequence $(K_n)$ of compact subsets of $K$ converges in the Hausdorff distance to the compact set $K_\infty$ if 
$d_{\mathcal H}(K_n,K_\infty) \to 0$. 
Finally let us recall Blaschke's selection principle which asserts that 
from any sequence $(K_n)_{n \in \mathbb N}$ of compact subsets of $K$, one can extract a subsequence converging in the Hausdorff distance.

\medskip

\noindent {\bf Capacities.} In the sequel, we will use the notion of capacity for which we refer to \cite{AH,HP}. We just recall the definition and several facts. The $(k,2)$-capacity of a compact set $K \subset \R^n$ is defined by
$$\Ccap_{k,2}(K):=\inf \left\{ \|\varphi\|_{H^k(\R^n)} : \varphi \in \C^\infty_c(\R^n),\; \varphi \geq 1 \text{ on } K\right\}.$$
This definition is then extended to open sets $A \subset \R^2$ by
$$\Ccap_{k,2}(A):= \sup\big\{ \Ccap_{k,2}(K) : K \subset A, \; K \text{ compact}\big\},$$
and to arbitrary sets $E \subset \R^n$ by
$$\Ccap_{k,2}(E):= \inf\big\{ \Ccap_{k,2}(A) : E \subset A, \; A \text{ open}\big\}.$$
One of the interests of capacity is that it enables one to give an accurate sense to the pointwise value of Sobolev functions. More precisely, every $u \in H^k(\R^n)$ has a $(k,2)$-quasicontinuous representative $\tilde u$, which means that $\tilde u=u$ a.e.\ and that, for each $\varepsilon>0$, there exists a closed set $A_\varepsilon \subset \R^n$ such that $\Ccap_{k,2}(\R^n \setminus A_\varepsilon)<\varepsilon$ and $\tilde u|_{A_\varepsilon}$ is continuous on $A_\varepsilon$ (see \cite[Section 6.1]{AH}). The $(k,2)$-quasicontinuous representative is unique, in the sense that two $(k,2)$-quasicontinuous representatives of the same function $u \in H^k(\R^n)$ coincide $\Ccap_{k,2}$-quasieverywhere. In addition, if $U$ is an open subset of $\R^n$, then $u{{}}\in H^k_0(U)$ if and only if for all multi-index $\alpha \in \mathbb N^n$ with length $|\alpha|\leq k-1$, $\partial^\alpha u$ has a $(k-|\alpha|,2)$-quasicontinuous representative that vanishes ${\rm Cap}_{k-|\alpha|,2}$-quasi everywhere on $\partial U$, {\it i.e.} outside a set of zero ${\rm Cap}_{k-|\alpha|,2}$-capacity (see \cite[Theorem 9.1.3]{AH}). In the sequel, we will only be interested in the cases $k=1$ or $k=2$ in dimension $n=2$.

\subsection{Definition of the problem}

We now describe the underlying fracture mechanics model and the related variational problem.

\medskip

\noindent {\bf Reference configuration.}  Let us consider a homogeneous isotropic linearly elastic body occupying $\Omega \subset \R^2$ in its reference configuration. The Hooke  law associated to this material is given by 
$$\mathbf A \xi=\lambda ({\rm tr}\xi) I+2 \mu \xi \quad \text{ for all }\xi \in \mathbb M^{2 \times 2}_{\rm sym},$$
where $\lambda$ and $\mu$ are the Lam\'e coefficients satisfying $\mu>0$ and $\lambda+\mu>0$. 
Note that this expression can be inverted into 
$$\mathbf A^{-1} \sigma=\frac{1}{2\mu}\sigma - \frac{\lambda}{4\mu(\lambda+\mu)}({\rm tr}\sigma) I=\frac{1+\nu}{E}\sigma-\frac{\nu}{E}({\rm tr} \sigma)I \quad \text{ for all }\sigma \in \mathbb M^{2 \times 2}_{\rm sym},$$
where $E:=4\mu(\lambda+\mu)/(\lambda+2\mu)$ is the Young modulus and $\nu:=\lambda/(\lambda+2\mu)$ is the Poisson ratio.

\medskip

\noindent {\bf Admissible displacements/cracks pairs.}  Let $\Omega' \subset \R^2$ be a bounded open set such that ${\rm diam}(\Omega') \leq 2{\rm diam}(\Omega)$ and $\overline \Omega \subset \Omega'$. We say that a pair set/function is admissible, and we write $(u,K) \in \mathcal A(\Omega)$, if $K\subset \overline \Omega$ {is  closed}, and $u \in LD(\Omega' \setminus K)$.

\medskip

\noindent {\bf Griffith energy.} For all $(u,K) \in \mathcal{A}(\Omega)$, we define the Griffith energy functional by
$$\mathcal{G}(u,K):=\int_{\Omega \setminus K}\mathbf A e(u):e(u)\, dx +   \mathcal{H}^{1}(K).$$
In this work, we are interested in (interior) regularity properties of the global minimizers of the Griffith energy under a Dirichlet boundary condition, i.e., solutions to the (strong) minimization problem
\begin{equation}\label{eq:strong-min}
\inf\left\{\mathcal G(v,K'): \quad (v,K') \in   \mathcal A(\Omega), \, v=\psi \text{ a.e. in }\Omega' \setminus\overline \Omega\right\},
\end{equation}
where  $\psi \in W^{1,\infty}(\R^2;\R^2)$ is a prescribed boundary displacement. Note that, this formulation of the Dirichlet boundary condition permits to account for possible cracks on $\partial \Omega$, where the displacement does not match the prescribed displacement $\psi$. 

The question of the existence of solutions to \eqref{eq:strong-min} has been addressed in \cite{CC2} (see also \cite{CC,FS}), extending up to the boundary the  regularity results \cite{CFI,CCI}. For this, by analogy with the classical Mumford-Shah problem, it is convenient to introduce a weak formulation of \eqref{eq:strong-min} as follows
$$\inf\left\{\int_\Omega \mathbf A e(v):e(v)\, dx + \mathcal H^1(J_v) : \; v \in GSBD^2(\Omega'), \, v=\psi \text{ a.e. in }\Omega' \setminus\overline \Omega\right\},$$
with $GSBD^2$ a suitable subspace of that of generalized special functions of bounded deformation (see \cite{DM}) where the previous energy functional is well defined. According to \cite[Theorem 4.1]{CC}, if $\Omega$ has Lipschitz boundary, the previous minimization problem admits at least a solution, denoted by $u$. In addition,  if $\Omega$ is of class $\mathcal C^1$, thanks to \cite[Theorems 5.6 and 5.7]{CC2}, there exist $\theta_0>0$ and $R_0>0$, only depending on $\mathbf A$, such that the following property holds: for all $x_0 \in \overline{J_u}$ and all $r \in (0,R_0)$ such that $B(x_0,r) \subset \Omega'$, then
$$\mathcal H^1(J_u \cap B(x_0,r)) \geq \theta_0 r.$$
The previous property of $J_u$ ensures that, setting $K:=\overline{J_u}$, then $\mathcal H^1(K \setminus J_u \cap { \overline \Omega})=0$, so that the pair $(u,K) \in \mathcal A(\Omega)$ is a solution of the strong problem \eqref{eq:strong-min}. In addition, the crack set $K$ is $\mathcal H^1$-rectifiable and Ahlfors regular: for all $x_0 \in K$ and all $r \in (0,R_0)$ such that $B(x_0,r) \subset { \Omega'}$, then
\begin{equation}\label{ahlfors}
\theta_0 r\leq \mathcal H^1(K \cap B(x_0,r)) \leq C r,
\end{equation}
where $C$ is a constant depending only on $\Omega$. The second inequality is obtained by comparing $(u,K)$ with the most standard competitor $(v,K')$ where $v:=u{\bf 1}_{\Omega' \setminus (\Omega\cap B(x_0,r))}$ and $K':=[K\setminus (\Omega\cap B(x_0,r))] \cup \partial (\Omega \cap B(x_0,r))$.

Next, taking in particular $K'=K$ and any $v\in LD(\Omega\setminus K)$ as competitor implies that $u \in LD(\Omega\setminus K)$ is also a  solution of the minimization problem
$$\min\left\{ \int_{\Omega \setminus K} \mathbf A e(v):e(v)\, dx : v \in LD(\Omega\setminus K), \; v=\psi \text{ on }\partial\Omega \setminus K  \right\}.$$
Note that $u$ is unique up to an additive rigid movement in each connected component of $\Omega \setminus K$ disjoint from $\partial \Omega \setminus K$.  It turns out that $u$ satisfies the following variational formulation: for all test functions $\varphi \in H^1(\Omega\setminus K;\R^2)$ with $ \varphi=0$ on $\partial \Omega \setminus K$,
\begin{equation}\label{eq:var-form}
\int_{\Omega \setminus K} \mathbf A e(u) : e(\varphi)\, dx = 0.
\end{equation}
In particular, $u$ is a solution to the elliptic system
$$-{\rm div} (\mathbf A e(u))=0 \quad \text{ in }\mathcal D'(\Omega\setminus K;\R^2),$$
and, as a consequence, elliptic regularity shows that $u \in \mathcal C^\infty(\Omega \setminus K;\R^2)$.



\section{The main quantities and proof of the $\mathcal C^{1,\alpha}$ regularity}\label{sec3}


We now introduce the main quantities that will be at the heart of our analysis.

\subsection{The normalized energy}

Let $(u,K) \in \mathcal{A}(\Omega)$. Then for any $x_0$ and $r>0$ such that $\overline B(x_0,r)\subset \Omega$ we define the {\it normalized elastic energy} by 
$$\omega(x_0,r):=\frac{1}{r}\int_{B(x_0,r) \setminus K}\mathbf A e(u):e(u) \; dx.$$
Sometimes we will write $\omega_u(x_0,r)$ to emphasize the underlying displacement $u$.

\begin{remark}
{\rm By definition of the normalized energy, for all $0<t<r$, we have
\begin{equation}\label{brutest1}\omega(x_0,t)\leq \frac{r}{t}\omega(x_0,r),
\end{equation}
 If $K'=\frac{1}{r}(K-x_0)$ and $v=\frac{1}{\sqrt{r}}u(r(\cdot +x_0))$ then
$$\omega_u(x_0,r)=\omega_{v}(0,1).$$
}
\end{remark}

\subsection{The flatness} 
Let $K$ be a closed subset of $\R^2$. For any $x_0\in \R^2$ and $r>0$, we define the {\it flatness} by
$$\beta(x_0,r):=\frac{1}{r}\inf_{L}  \max\left\{ \sup_{y\in K\cap \overline{B}(x_0,r)}{\rm dist }(y,L), \sup_{y\in L\cap \overline{B}(x_0,r)}{\rm dist }(y,K)\right\} ,$$
where the infimum is taken over all affine lines $L$ passing through $x_0$. In other words
$$\beta(x_0,r)=\frac{1}{r}\inf_{L} d_{\mathcal H}(K\cap \overline{B}(x_0,r), L\cap \overline{B}(x_0,r)).$$
 Sometimes we will write $\beta_K(x_0,r)$ to emphasize the underlying crack $K$.

\begin{remark}
{\rm By definition of the flatness, we always have that for all $0<t<r$,
\begin{equation}\label{brutest2}
\beta_K(x_0,t)\leq \frac{r}{t}\beta_K(x_0,r),
\end{equation}
and if $K'=\frac{1}{r}(K-x_0)$, then
$$\beta_K(x_0,r)=\beta_{K'}(0,1).$$
}
\end{remark}

In the sequel, we will consider the situation where
\begin{equation}\label{eq:omegaepsilon}
\beta_K(x_0,r)\leq \varepsilon,
\end{equation}
for $\varepsilon>0$ small.  This implies in particular that $K\cap \overline B(x_0,r)$ is contained in a narrow  strip  of thickness $\varepsilon r$ passing through the center of the ball. 

Let $L(x_0,r)$ be a line containing $x_0$ and satisfying 
\begin{equation}\label{optimal}
d_{\mathcal H}(L(x_0,r)\cap \overline B(x_0,r),K\cap \overline B(x_0,r))\leq r \beta_K(x_0,r).
\end{equation}
We will often use a local basis (depending on $x_0$ and $r$) denoted by $(e_1,e_2)$, where $e_1$ is a tangent vector to the line $L(x_0,r)$, while $e_2$ is an orthogonal vector to $L(x_0,r)$. The coordinates of a point $y$ in that basis will be denoted by $(y_1,y_2)$.

Provided \eqref{eq:omegaepsilon} is satisfied with $\varepsilon \in (0,1/2)$, we can define two discs  $D^+(x_0,r)$ and $D^{-}(x_0,r)$ of radius $r/4$ and such that $D^{\pm}(x_0,r)\subset B(x_0,r)\setminus K$. Indeed, using the notation introduced above, setting $x_0^{\pm}:= x_0\pm\frac{3}{4}r e_2$, we can check that $D^\pm(x_0,r):=B(x_0^\pm,r/4)$ satisfy the above requirements.
\medskip

A property that will be fundamental in our analysis is the separation in a closed ball.

\begin{definition}
Let $K$ be a closed subset of $\R^2$,  $x_0\in \R^2$ and $r>0$ be such that $\beta_K(x_0,r) \leq 1/2$. 
We say that $K$ \emph{separates} $D^\pm(x_0,r)$ in $\overline{B}(x_0,r)$ if the balls $D^\pm(x_0,r)$ are contained into two different connected components of $\overline{B}(x_0,r)\setminus K$.
\end{definition}

The following lemma guarantees that when passing from a ball $B(x_0,r)$ to a smaller one $B(x_0,t)$, {and provided that $\beta_K(x,r)$ is relatively small,} the property of separating is preserved for $t$ varying in a range depending on {$\beta_K(x,r)$}. 
\begin{lemma} \label{topological1}
Let $\tau\in (0,1/16)$, let $K \subset \Omega$ be a relatively closed set, and let $x_0 \in K$, $r>0$ be such that $\overline B(x_0,r) \subset \Omega$. Assume that $\beta_K(x_0,r)  \leq \tau$ and that $K$ separates $D^\pm(x_0,r)$ in $\overline{B}(x_0,r)$. Then for all $t \in (16\tau r, r)$, we have $\beta_K(x_0,t)\leq \frac{1}{2}$ and $K$ still separates $D^\pm(x_0,t)$ in $\overline{B}(x_0,t)$.
\end{lemma}

\begin{proof} 
We will need the following elementary inequality resulting from the mean value Theorem
\begin{equation}
\arcsin(t) \leq \sup_{s\in [0,\frac{1}{2}]} \frac{1}{\sqrt{1-s^2}} t = \frac{2}{\sqrt{3}} t\leq 2t \quad \text{for all } t \in \left[0,\frac{1}{2}\right]. \label{arcsin}
\end{equation}
Using the notations introduced above, considering the local basis $\{e_1,e_2\}$ such that $e_1$ is a tangent vector to $L(x_0,r)$ and $e_2$ is a normal vector to $L(x_0,r)$, we have
\begin{equation}
K\cap \overline B(x_0,r)\subset \{ y \in \overline B(x_0,r) \; :\; |y_2|\leq \tau r\}. \label{Kstrip}
\end{equation}
For all $t \in (16\tau r , r)$, we have
\begin{equation}
\beta(x_0,t)\leq \frac{r}{t} \beta(x_0,r)\leq \frac{1}{16\tau} \beta(x_0,r)\leq  \frac{1}{16}\leq  \frac{1}{2}, \label{estimbetat}
\end{equation}
so that $D^\pm(x_0,t)$ are well defined. Denoting by $\nu(x_0,t)$ a normal vector to the line $L(x_0,t)$, we can assume that $\nu(x_0,t) \cdot e_2>0$.

We first note that, similarly to \eqref{estimbetat}, we can estimate
\begin{equation}
{\rm dist} (x,L(x_0,t)) \leq t \beta(x_0,t)  \leq r \beta(x_0,r) \leq   \tau r \quad \text{ for all } x\in K \cap \overline B(x_0,t).  \label{Kstrip2}
\end{equation}
From  \eqref{Kstrip2}  and \eqref{Kstrip} we deduce
$$  B(x_0,t) \cap  L(x_0,t)\subset \{ y \in \R^2  :\; |y_2|\leq 2\tau r \}.$$
Denoting by $\alpha=\arccos (\nu(x_0,t) \cdot e_2)$ the angle between $\nu(x_0,t)$ and $e_2$, the previous inclusion implies
\begin{equation}
\alpha \leq \arcsin\left(\frac{2\tau r}{t}\right)\leq \frac{4\tau r}{t}\leq \frac{1}{4},\label{estimalpha}
\end{equation}
where we have used \eqref{arcsin},  and that $t> 16\tau r$.

Let $y_0:=x_0+\frac{3}{4}\nu(x_0,t) t$ be the center of the disc $D^+(x_0,t)$. We have that $  |(y_0-x_0)_2|=\cos(\alpha)\frac{3}{4}t$. In particular, using the elementary inequality $|\cos(\alpha)-1|\leq \alpha$ and \eqref{estimalpha} we get 

$${\rm dist}(y_0,L(x_0,r))= |(y_0-x_0)_2|=\cos(\alpha)\frac{3}{4}t\geq \frac{3}{4}(1-\alpha)t\geq \frac{t}{2},$$
hence, since $t> 16\tau r$, we infer that for all $y\in B(y_0,\frac{t}{4})$,
$${\rm dist}(y,L(x_0,r))\geq \frac{t}{2}-\frac{t}{4}=\frac{t}{4}\geq 4\tau r.$$
All in all, we have proved that 
$$D^+(x_0,t)=B\left(y_0,\frac{t}{4}\right)\subset \{y \in \R^2:\; y_2\geq 4\tau r  \}.$$
Arguing similarly for $D^-(x_0,t)$, we get
$$D^-(x_0,t)\subset \{y \in \R^2:\; y_2\leq -4\tau r  \}.$$
Since by \eqref{Kstrip} we have that $K \cap \overline B(x_0,t) \subset \{y \in \overline B(x_0,t): \; |y_2|{{}}\leq \tau r\}$, we deduce that $D^\pm(x_0,t)$ must belong to two distinct connected components of $\overline B(x_0,{ t}) \setminus K$, and thus that $K$ actually separates $D^\pm(x_0,t)$  in $\overline{B}(x_0,t)$. 
\end{proof}

The following topological result is well-known.
\begin{lemma} \label{existGamma}
Let $K \subset \Omega$ be a relatively closed set, and let $x_0 \in K$, $r>0$ be such that $\overline B(x_0,r) \subset \Omega$. Assume that $\mathcal{H}^1(K\cap B(x_0,r))< +\infty$, $\beta_K(x_0,r)  \leq 1/2$ and that $K$ separates $D^\pm(x_0,r)$ in $\overline{B}(x_0,r)$. Then, there exists an injective Lipschitz curve $\Gamma\subset K$ that still separates $D^\pm(x_0,r)$ in $\overline{B}(x_0,r)$.
\end{lemma}

\begin{proof} 
Since $K$ separates $D^\pm(x_0,r)$ in $\overline{B}(x_0,r)$ there exists a compact and connected set  $\tilde K \subset   K \cap \overline B(x_0,r)$ which still separates (see \cite[Theorem 14.3]{Newman}). Since $\tilde K$ has also finite $\mathcal H^1$ measure, it follows from \cite[Corollary 30.2]{d} that $\tilde K$  is arcwise connected. We denote by $(e_1,e_2)$ an orthonormal system such that $L(x_0,r)$ is directed by $e_1$. Since $\tilde K$ separates   $D^\pm(x_0,r)$ in $\overline{B}(x_0,r)$, then $\tilde K$ must contain at least one point in both connected components of $\partial B(x,r)\cap \{ y \in \R^2 :\; |y_2|\leq \frac{r}{2}\}.$ Denoting by $z$ and $z'$ those two points, there exists a Lipschitz injective curve $\Gamma$ in $\tilde K \cap \overline{B}(x_0,r)$ joining $z$ and $z'$, which separates $D^\pm(x_0,r)$ in $\overline{B}(x_0,r)$ (see for example \cite[Proposition 30.14]{d}). 
\end{proof}

\subsection{Initialization of the main quantities}

We prove that, if $(u,K)$ is a minimizer of the Griffith functional, one can find many balls $\overline B(x_0,r)\subset \Omega$ such that $K$ separates $D^\pm(x_0,r)$ in $\overline{B}(x_0,r)$ and such that $\beta_K(x_0,r)$ and $\omega_u(x_0,r)$ are small for $r>0$ small enough, and for $\mathcal H^1$-a.e. $x_0 \in \Gamma$, where $\Gamma\subset K$ is any connected component of $K\cap \Omega$. The restriction to a connected component $\Gamma$  is only due to ensure the separation property on $K$. Notice that in the following proposition  we do not need the connected component to be isolated.

\begin{proposition} \label{iniT} 
Let  $(u,K)\in \mathcal A(\Omega)$ be a minimizer of the Griffith functional and let $\Gamma\subset K$ be a connected component of  {$K\cap \Omega$} such that $\mathcal H^1(\Gamma)>0$. Then for every $\varepsilon \in {(0,10^{-3})}$ there exists an exceptional set $Z\subset \Gamma$ with $\mathcal{H}^1(Z)=0$ such that the following property holds. For every $x_0 \in \Gamma \setminus Z$, there exists  $r_0>0$, such that 
$$\beta_K(x_0,r_0)\leq \varepsilon, \quad \omega_u(x_0,r_0)\leq \varepsilon$$
and $K$ separates $D^\pm(x_0,r_0)$  {in } $\overline B(x_0,r_0)$. 
\end{proposition}

\begin{proof} 
The initialization for the quantity $\beta$ is standard (see for instance \cite[Exercice 41.21.3]{d} and \cite[Exercice 41.23.1]{d}), we sketch below the proof for the sake of completeness. 
 
Since $K$ is a rectifiable set, we know that there exists $Z_1\subset K$ with $\mathcal{H}^1(Z_1)=0$ such that, at every point $x_0 \in K\setminus Z_1$, $K$ admits an approximate tangent line $T_{x_0}$, that is
\begin{equation}\label{approxtang}
\lim_{r\to 0}\frac{\mathcal{H}^1(K\cap B(x_0,r)\setminus T_{x_0,\varepsilon r})}{r}=0,
\end{equation}
for all $\varepsilon\in(0,1)$, where $T_{x_0,\varepsilon r}:=\{y\in\R^2:{\rm dist}(y,T_{x_0})\leq \varepsilon r\}$. Since $K$ is also Ahlfors-regular by assumption, it is easily seen that $T_{x_0}$ is the usual tangent, in the sense that for all $\varepsilon\in(0,1)$ there exists $r_\varepsilon>0$ such that 
$$K\cap B(x_0,r)\subset T_{x_0,\varepsilon r},$$
for all $r\leq r_\varepsilon$. Indeed, assume that there exist $\varepsilon_0\in(0,1)$ and two sequences, $r_k\to 0$ and $y_k\in K\cap B(x_0,r_k)$, such that
$${\rm dist}(y_k,T_{x_0})>\varepsilon_0 r_k.$$
Then, by Ahlfors regularity \eqref{ahlfors} we have
$$\mathcal{H}^1\left(K\cap B\left(x_0,2r_k\right)\setminus T_{x_0,\frac{\varepsilon_0 r_k}{2}}\right)\geq \mathcal{H}^1\left(K\cap B\left(y_k,\frac{\varepsilon_0 r_k}{2}\right)\right)\geq \theta_0 \varepsilon_0\frac{r_k}{2}.$$
We conclude that 
$$\liminf_{k\to +\infty}\frac{\mathcal{H}^1\left(K\cap B\left(x_0,2r_k\right)\setminus T_{x_0,\frac{\varepsilon_0 r_k}{2}}\right)}{2r_k}\geq \frac{\theta_0\varepsilon_0}{4}>0,$$
which is against \eqref{approxtang}. Hence, by definition, for all $x_0\in K\setminus Z_1$ and all $\varepsilon\in (0,1)$, there exists $r_1>0$ such that 
$$\beta(x_0,r)\leq \varepsilon \text{ for all } r\leq r_1.$$

Now we consider $\omega(x_0,r)$, which again can be initialized by the same argument used for the standard Mumford-Shah functional (see for instance \cite[Proposition 7.9]{afp}). Let us reproduce it here. We consider the measure $\mu :=\mathbf{A} e(u):e(u)  \mathcal L^2$. For all $t>0$, let 
$$E_t:=\left\{ x\in K : \; \limsup_{r\to 0}\frac{\mu(B(x,r))}{r} >t \right\}.$$
By a standard covering argument  (see \cite[Theorem 2.56]{afp}) one has that 
$$t \mathcal{H}^1(E_t)\leq \mu(E_t).$$
But $E_t \subset K$ and $\mu(K)=0$, thus $\mathcal{H}^1(E_t)=0$ for all $t>0$. By taking a sequence $t_n\searrow 0^+$ and defining $Z_2 := \bigcup_{n} E_{t_n}$, we have that $\mathcal{H}^1(Z_2)=0$ and, for all $x_0\in K\setminus Z_2$,
$$ \lim_{r\to 0}\omega(x_0,r)=0.$$
In other words,  for every $x_0\in K\setminus (Z_1\cup Z_2)$, there exists $r_2<r_1$ such that 
$$  \beta(x_0,r)\leq \varepsilon, \quad \omega(x_0,r)\leq \varepsilon,$$
for all $r\leq r_2$.

It remains to prove the separation property of $K$. To this aim, let us consider a connected component $\Gamma$ of $K\cap \Omega$ which is relatively closed in $K\cap \Omega$. Since $\overline{\Gamma}$ is a compact and connected set in $\R^2$ with $\mathcal{H}^1(\overline{\Gamma})<+\infty$, according to \cite[Proposition 30.1]{d} it is the range of an injective Lipschitz mapping $\gamma:[0,1]\to \overline{\Gamma}$. This implies that $\overline \Gamma$ has an approximate tangent line $L_{x_0}$ for $\mathcal H^1$-a.e. $x_0 \in \overline \Gamma$. In addition, according to \cite[Proposition 2.2.{(iii)}]{BLS}, there exists an exceptional set $Z_3\subset \overline \Gamma$ with $\mathcal{H}^1(Z_3)=0$ and with the property that, for all $x_0\in \Gamma\setminus Z_3$, one can find $r_3>0$ such that, 
\begin{equation}
\pi(\overline{\Gamma}\cap B( x_0,r))\supset L_{x_0}\cap B( x_0,(1-10^{-3})r) \quad \text{ for all } r\leq r_3,\label{projectionEstim}
\end{equation}
where  $\pi:\R^2\to L_{x_0}$ denotes the orthogonal projection onto the line $L_{x_0}$. In particular, if moreover $\beta(x_0,r)\leq \varepsilon \leq 10^{-3}$, it follows that the balls $D^\pm(x_0,(1-10^{-3})r)$ are well defined and, thanks to \eqref{projectionEstim}, that $\overline{\Gamma}$ must separate $D^\pm(x_0,(1-10^{-3})r)$  in $\overline{B}(x_0,(1-10^{-3})r)$, hence  $K$ must separate too. Setting $Z:=Z_1\cup Z_2 \cup Z_3$, then $\mathcal H^1(Z)=0$,  and we have proved that for all  $x_0  \in \Gamma \setminus Z$ and all $r \leq r_0:= \min(r_1,r_2,(1-10^{-3})r_3)$, we obtain that 
$$\beta(x_0,r) \leq \varepsilon, \quad \omega(x_0,r) \leq \varepsilon,$$
and $K$ separates $D^+(x_0,r)$ from $D^{-}(x_0,r)$ in $\overline B(x_0,r)$, as required.
\end{proof}
 
\subsection{Proof of Theorem~\ref {mainTH}}


The proof of our main result, Theorem~\ref {mainTH}, rests on both the following results, whose proofs are postponed to the subsequent sections. The first one is a flatness estimate in terms of the renormalized energy, which will be established in Section~\ref {sec5}.
\begin{proposition}\label{main3}  
There exist $\varepsilon_1>0$ and $C_1>0$ (only depending on $\theta_0$, the Ahlfors regularity constant of $K$) such that the following property holds. Let $(u,K) \in \mathcal A(\Omega)$ be a minimizer of the Griffith functional. For all $x_0 \in K$ and $r>0$ such that $\overline B(x_0,r)\subset \Omega$,
$$\omega_u(x_0,r)+\beta_K(x_0,r) \leq \varepsilon_1,$$ 
and $K$ separates $D^{\pm}(x_0,r)$ in $\overline{B}(x_0,r)$, we have
$$\beta_K\left(x_0,\frac{r}{50}\right) \leq  C_1 \omega_u(x_0,r)^{\frac{1}{14}}.$$
\end{proposition}

The second result is the following normalized energy decay which will be proved in Section~\ref {sec6}.

\begin{proposition}\label{JFprop}
For all $\tau>0$, there exists $a \in (0,1)$ and $\varepsilon_2 >0$ such that the following property holds. Let $(u,K) \in \mathcal A(\Omega)$ be a minimizer of the Griffith functional, and let $\Gamma$ be an isolated connected component of  $K\cap \Omega$ such that $\mathcal{H}^1(\Gamma)>0$. Let $x_0 \in \Gamma$ and $r>0$ be such that $\overline B(x_0,r) \subset \Omega$ and 
$$K\cap \overline B(x_0,r)=\Gamma\cap \overline B(x_0,r), \quad \beta_K(x_0,r)\leq \varepsilon_2,$$
then we have $$\omega_u\left(x_0,ar \right) \leq \tau \,\omega_u(x_0,r).$$
\end{proposition}

With both previous results at hands, we are in position to bootstrap the preceding decay estimates in order to get a $\mathcal C^{1,\alpha}$-regularity estimate. Indeed, the conclusion of Theorem~\ref {mainTH} will follow from the following result. 

\begin{proposition}
Let $(u,K) \in \mathcal A(\Omega)$ be a minimizer of the Griffith functional and let $\Gamma$ be an isolated connected component of  $K\cap \Omega$. Then, there exists a relatively closed set $Z\subset \Gamma$ with $\mathcal{H}^1(Z)=0$ such that for every $x_0\in \Gamma \setminus Z$, one can find $r_0>0$ such that $\Gamma \cap B(x_0,r_0)$ is a $\mathcal C^{1,\alpha}$ curve.
\end{proposition}

\begin{proof} 
Let $\varepsilon_1>0$ and $C_1>0$ be the constants given by Proposition~\ref {main3}, and let $a>0$ and $\varepsilon_2>0$ be the constants given by Proposition~\ref {JFprop} corresponding to $\tau = {10}^{-2}$. We define
$$\delta_1:=\min \left\{\varepsilon_1,\varepsilon_2,10^{-3} a\right\},\quad \delta_2:=\min\left\{\delta_1,\left(\frac{a \delta_1}{C_1}\right)^{14}\right\}.$$

We can assume that $\mathcal{H}^1(\Gamma)>0$, otherwise the Proposition is trivial. Using  that  $\Gamma$ is an isolated connected component of $K\cap \Omega$ and Proposition~\ref {iniT}  (applied with $\varepsilon=\min(\delta_1,\delta_2)/2$), we can find an exceptional set $Z\subset \Gamma$ with $\mathcal{H}^1(Z)=0$ such that the following property holds: for every $x_0 \in \Gamma \setminus Z$, there exists  $r>0$, such that 
$$K\cap \overline B(x_0,r)=\Gamma\cap \overline B(x_0,r),$$
\begin{equation}
 \omega(x_0,r)\leq   \frac{\delta_2}{2}\;,\; \quad \beta(x_0,r) \leq   \frac{\delta_1}{2},\text{ and } K \text{ separates }  D^{\pm}(x_0,r) \text{ in } \overline{B}(x_0,r).\label{initialisation0}
\end{equation}
We start by showing a first self-improving estimate which stipulates that  the quantities $\omega(x_0,r)$ and $\beta(x_0,r)$ will remain small at all smaller scales. 

\medskip

{\bf Step 1.} We define $b:=a/50$, and we claim that if 
$$\omega(x_0,r)\leq \delta_2\;,\; \quad \beta(x_0,r) \leq \delta_1,\text{ and } K \text{ separates }  D^{\pm}(x_0,r) \text{ in } \overline{B}(x_0,r),$$
then the following three assertions are true:
\begin{equation}
\omega(x_0,br)\leq \delta_2\;,\; \quad \beta(x_0,br) \leq \delta_1,\text{ and } K \text{ separates }  D^{\pm}(x_0,br) \text{ in } \overline{B}(x_0,br)
 \label{claimp}
\end{equation}
\begin{eqnarray}
\omega(x_0,b r)&\leq&   \frac{1}{2}\omega(x_0,r)  \label{ooo} \\
\beta\left(x_0,br\right) &\leq&    \frac{C_1}{a} \omega(x_0,r)^{\frac{1}{14}}. \label{bbb}
\end{eqnarray}
Let us start with the renormalized energy. By Proposition~\ref {JFprop} we get 
$$\omega(x_0,ar)\leq 10^{-2}\omega(x_0,r),$$
which yields using \eqref{brutest1}
$$\omega(x_0,b r)\leq  50\, \omega(x_0,a r)\leq \frac{1}{2}\omega(x_0,r)\leq \delta_2. $$
For what concerns the flatness, we can apply Proposition~\ref {main3}, so that 
$$\beta\left(x_0,\frac{r}{50}\right) \leq  C_1 \omega(x_0,r)^{\frac{1}{14}}.$$
Thus by \eqref{brutest2} we get
$$\beta(x_0,br) \leq    \frac{1}{a}\beta\left(x_0,\frac{r}{50}\right)\leq \frac{C_1}{a} \omega(x_0,r)^{\frac{1}{14}} \leq \delta_1,$$
because  $\delta_2\leq \left(\frac{a \delta_1}{C_1}\right)^{14}$.
Finally, since $\delta_1 \leq 10^{-3} a$, we infer that $16\delta_1 r \leq br <r$, so that $K$ still separates $D^\pm(x_0,br)$ in $\overline{B}(x_0,br)$ owing to Lemma~\ref {topological1} (applied with $\tau=\delta_1$), and thus the claim is proved.

\medskip

{\bf Step 2.} Iterating the decay estimate established in Step 1, we get that \eqref{claimp}, \eqref{ooo}, and \eqref{bbb} hold true in each ball $B(x_0,b^kr)$, $k \in \mathbb N$. We thus obtain that 
$$\omega(x_0,b^kr)\leq 2^{-k}\omega(x_0,r)$$
and subsequently, using now \eqref{bbb},
$$\beta(x_0,b^kr)\leq \frac{C_1}{a} 2^{-\frac{k-1}{14}} \omega(x_0,r)^{\frac{1}{14}}.$$
If $t \in (0,1)$ we let $k\geq 0$ be the integer such that  
$$b^{k+1}\leq t< b^{k}.$$
Notice in particular that
$$k+1\geq \frac{\ln(1/t)}{\ln(1/b)} > k  $$
thus $2^{-\frac{k+1}{14}}\leq t^\alpha$ with $\alpha=\frac{\ln(2)}{14|\ln(b)|}\in (0,1)$. We deduce that 
$$\beta(x_0,tr)\leq \frac{b^k}{t}\beta(x_0,b^k r)\leq \frac{C_1}{ab}2^{-\frac{k-1}{14}}  \omega(x_0,r)^{\frac{1}{14}}\leq C t^\alpha \omega(x_0,r)^\frac{1}{14},$$
for some constant $C>0$ only depending on $C_1$ and $a$.

\medskip

{\bf Step 3.} We now conclude the proof of the proposition. Indeed, according to \eqref{initialisation0}, for every $x\in K\cap B(x_0,r/2)$ we still have 
$$\omega(x,r/2)\leq \delta_2, \quad \beta(x,r/2)\leq \delta_1$$
and $K$ separates $D^\pm(x,r/2)$ in $\overline{B}(x,r/2)$.   Thus, by Steps 1 and 2 applied in each ball $B(x,r/2)$ with $x \in K\cap B(x_0,r/2)$, we deduce that 
$$\beta(x,tr)\leq   C\delta_2^{\frac{1}{14}} t^\alpha \quad \text{ for all }t\in (0,1/2),$$
and since this is true for all $x\in K\cap B(x_0,r/2)$, we deduce that $K\cap B(x_0,a_0r)$ is a $\mathcal C^{1,\alpha}$ curve for some $a_0\in(0,1/2)$ thanks to Lemma~\ref {c1estimates} in the appendix.
\end{proof}

\section{Proof of the flatness estimate}\label{sec5}

In order to prove Proposition~\ref{main3}, we need to construct a competitor in a ball $B(x_0,r)$, where the flatness $\beta(x_0,r)$ and the renormalized energy $\omega(x_0,r)$ are small enough. The main difficulty is to control how the crack behaves close to the boundary of the ball. A first rough competitor is constructed in Propositions~\ref {prop1} and \ref {prop_para} by introducing a wall set of length $r\beta(x_0,r)$ on the boundary. It leads to density estimates in balls (or alternatively in rectangles) which state that, provided the crack is flat enough, the energy density scales like {the diameter} of the ball (or the width of the rectangle), up to a small error depending on $\beta(x_0,r)$ and $\omega(x_0,r)$. 

Unfortunately, this rough competitor is not sufficient to get a convenient flatness estimate leading to the desired regularity result. A better competitor is obtained by suitably localizing the crack in two almost opposite boxes of size $\eta>0$, arbitrarily small (see Lemma~\ref{bow-tie}). Then we can define a competitor inside a larger rectangle $U$, whose vertical sides intersect both the small boxes. The crack competitor is then defined by taking   {an almost horizontal segment} inside the rectangle, together with a new wall set $\Sigma \subset \partial U$ of arbitrarily small length, made of the intersection of the rectangle with the boxes. It is then possible to introduce a displacement competitor (see Lemma~\ref{extLem}) by extending the value of $u$ on $\partial U \setminus \Sigma$ inside $U$. The price to pay is that the bound on the elastic energy associated to this competitor might diverge as the length of the wall set is small. It is however possible to optimize the competition between the flatness and the renormalized energy associated to this competitor by taking $\eta=\omega(x_0,r)^{1/7}$, leading to the conclusion of Proposition~\ref{main3}.

\subsection{Density estimates}
  

In this section we prove some density estimates for the set $K$. Such estimates will be useful to select good radii, in a way that the corresponding spheres intersect the set $K$ in two almost opposite points. One of the main tools to construct competitors will be the following extension lemma.

\begin{lemma}[Harmonic extension in a ball from an arc of circle]\label{extensionL} 
Let $0<\delta\leq 1/2$, $x_0 \in \mathbb R^2$, $r>0$ and 
let $\mathscr C_\delta\subset\partial B(x_0,r)$ 
be the arc of circle defined by $$\mathscr C_\delta := \{(x_1,x_2)\in \partial B(x_0,r) \; : \; (x-x_0)_2> \delta r \}.$$
Then, there exists a constant $C>0$ (independent of $\delta$, $x_0$,  and $r$) such that every function $u \in H^1(\mathscr C_\delta;\R^2)$ extends to a function $g \in H^1(B(x_0,r);\R^2)$ with $g=u$ on $\mathscr C_\delta$ and 
$$\int_{B(x_0,r)} |\nabla g|^2  \, dx\leq C r\int_{\mathscr C_\delta} |\partial_\tau u|^2  \, d\mathcal{H}^1.$$
\end{lemma}

\begin{proof} 
Let $\Phi : \mathscr C_\delta\to \mathscr C_0$ be a bilipschitz mapping with Lipschitz constants independent of $\delta \in (0,1/2]$, $x_0$, and  $r>0$. Since $u\circ \Phi^{-1}\in H^1(\mathscr C_0;\R^2)$, we can define the extension by reflection $\tilde u \in H^1(\partial  B(x_0,r); \R^2)$ on the whole circle $\partial B(x_0,r)$, that satisfies
$$\int_{\partial B(x_0,r)}|\partial_\tau \tilde u|^2 d\mathcal{H}^1 \leq C \int_{\mathscr C_\delta} |\partial_\tau u|^2  \, d\mathcal{H}^1,$$
where $C>0$ is a constant which is independent of $\delta$.

We next define $g$ as the harmonic extension of $\tilde u$ in $B(x_0,r)$. Using \cite[Lemma 22.16]{d}, we obtain
$$\int_{B(x_0,r)} |\nabla  g|^2  \, dx\leq Cr\int_{\partial B(x_0,r)}|\partial_\tau \tilde u|^2\, d\mathcal{H}^1 
 \leq C  r\int_{\mathscr C_\delta}|\partial_\tau u|^2  \, d\mathcal{H}^1,$$
which completes the proof.
\end{proof}

\begin{lemma}[Extension lemma, first version]\label{extension} 
Let $(u,K) \in \mathcal A(\Omega)$ be a minimizer of the Griffith functional, and let $x_0 \in K$ and  $r>0$ be such that $\overline B(x_0,r)\subset \Omega$ and $\beta_K(x_0,r)\leq 1/10$. Let $S$ be the strip defined by 
$$S:=\{y \in  \overline{B}(x_0,r) \; : \; {\rm dist}(y,L(x_0,r))\leq r\beta(x_0,r) \}.$$
Then there exist a universal constant $C>0$, $\rho \in (r/2,r)$, and $v^\pm \in H^1(B(x_0,\rho);\R^2)$, such that $v^\pm = u \text{ on }\mathscr C^\pm$, $\mathscr C^\pm$ being the connected components of $\partial B(x_0,\rho)\setminus S$, and
$$\int_{B(x_0,\rho)}|e( v^\pm)|^2 \,dx \leq  C\int_{B(x_0,r)\setminus K}|e(u)|^2 \, dx.$$
\end{lemma}
 
\begin{proof}
Let $A^\pm$ be the connected components of $B(x_0,r)\setminus S$. Since $K\cap A^\pm=\emptyset$, by the Korn  inequality there exist two skew-symmetric matrices  $R^\pm$ such that the functions $x\mapsto u(x)-R^\pm x$ belong to $H^1(A^\pm;\R^2)$ and 
$$\int_{A^\pm} |\nabla u-R^\pm|^2 \,dx \leq C \int_{A^\pm} |e(u)|^2\,dx,$$
where the constant $C>0$ is universal since the domains $A^\pm$ are all uniformly Lipschitz for all possible values of $\beta(x_0,r)\leq 1/10$. Using the change of variables in polar coordinates, we infer that 
$$\int_{A^\pm} |\nabla u-R^\pm|^2 \, dx = \int_{0}^{r} \left(\int_{\partial B(x_0,\rho)\cap A^\pm} |\nabla u-R^\pm|^2 \, d\mathcal{H}^1\right) d\rho$$
which allows us to choose a radius $\rho \in (r/2,r)$ satisfying 

\begin{multline*}
\int_{\partial B(x_0,\rho)\cap A^+} |\nabla u-R^+|^2\,d\mathcal{H}^1+\int_{\partial B(x_0,\rho)\cap A^-} |\nabla u-R^-|^2\,d\mathcal{H}^1\\
\leq \frac{2}{r}\int_{A^+} |\nabla u-R^+|^2 \,dx + \frac{2}{r}\int_{A^-} |\nabla u-R^-|^2 \,dx\leq  \frac{C}{r}\int_{B(x_0,r) \setminus K} |e(u)|^2 \,dx\,.
\end{multline*}
Setting $\mathscr C^\pm:=\partial A^\pm \cap \partial B(x_0,\rho)$, in view of Lemma~\ref {extensionL} applied to the functions  $u^\pm:x \mapsto u(x)-R^\pm x$, which belong to
$H^1(\mathscr C^\pm;\R^2)$ since they are regular, for $\delta=r\beta(x_0,r)$ we get two functions $g^\pm \in H^1(B(x_0,\rho);\R^2)$ satisfying $g^\pm(x) = u(x)-R^\pm x$ for $\mathcal H^1$-a.e. $x \in \mathscr C^\pm$ and
$$\int_{B(x_0,\rho)}|\nabla g^\pm|^2 \,dx \leq C\rho\int_{\mathscr C^\pm}|\partial_\tau u^\pm|^2 \,d\mathcal H^1 \leq C\int_{B(x_0,r)\setminus K}|e(u)|^2 \, dx.$$
Finally, the functions $x \mapsto v^\pm(x):=g^\pm(x) +R^\pm x$ satisfy the required properties.
\end{proof}

We now use the extension to prove two density estimates, first in smaller balls, then in smaller rectangles.

\begin{proposition}[Density estimate in a ball]\label{prop1}  
Let $(u,K) \in \mathcal A(\Omega)$ be a minimizer of the Griffith functional, and let $x_0 \in K$ and  $r>0$ be such that $\overline B(x_0,r)\subset \Omega$ and $\beta_K(x_0,r)\leq 1/10$. Then there exist a universal constant $C>0$ and a radius $\rho \in (r/2,r)$ such that 
$$\int_{B(x_0,\rho)\setminus K} \mathbf Ae(u):e(u) \,dx+ \mathcal{H}^1(K\cap B(x_0,\rho))  \leq 2\rho + C\rho\big(\omega_u(x_0,r)+\beta_K(x_0,r)\big).$$
\end{proposition}

\begin{proof} 
We keep using the same notation than that used in the proof of Lemma~\ref {extension}. Let $\rho \in (r/2,r)$ and $v^{\pm} \in H^1(B(x_0,\rho);\R^2)$ be given by the conclusion of Lemma~\ref {extension}. We now construct a competitor in $B(x_0,\rho)$ as follows. First, we consider a ``wall'' set $Z\subset \partial B(x_0,\rho)$ defined by 
$$Z:=\{y\in \partial B(x_0,\rho) \; : \; {\rm dist}(y,L(x_0,r))\leq r \beta(x_0,r)\}.$$
Note that $K\cap \partial B(x_0,\rho)\subset Z$,
$$\partial B(x_0,\rho)=[\partial A^+\cap \partial B(x_0,\rho)]\cup [\partial A^-\cap \partial B(x_0,\rho)] \cup Z=\mathscr C^+\cup \mathscr C^- \cup Z,$$
and that
$$\mathcal{H}^1(Z)=4\rho \arcsin\left( \frac{r\beta(x_0,r)}{\rho} \right)\leq 4r\beta(x_0,r).$$

We are now ready to define the competitor $(v,K')$ by setting
$$K':=\big[ K\setminus B(x_0,\rho ) \big]\cup Z \cup \big[L(x_0,r)\cap B(x_0,\rho)\big],$$
and, denoting by $V^\pm$ the connected components of $B(x_0,\rho)\setminus L(x_0,r)$ which intersect $A^\pm$,
$$
v:=
\begin{cases}
{v^\pm} & \text{ in } V^\pm \\
u & \text{ otherwise.}
\end{cases}
$$
Since $\mathcal{H}^1(K'\cap \overline B(x_0,\rho))\leq 2\rho+4r \beta(x_0,r)$, we deduce that
\begin{multline*}
\int_{B(x_0,\rho)\setminus K}\mathbf Ae(u):e(u) \, dx +\mathcal{H}^1(K\cap \overline B(x_0,\rho) )\\
\leq \int_{B(x_0,\rho)\setminus K} \mathbf Ae(v):e(v) \, dx +\mathcal{H}^1(K'\cap \overline B(x_0,\rho) ) \\
\leq  C\int_{B(x_0,r)\setminus K} |e(u)|^2 \, dx + \rho(2+C\beta(x_0,r))  \\
\leq 2\rho + C \rho\big(\omega(x_0,r)+\beta(x_0,r)\big),
\end{multline*}
and the proposition follows.
\end{proof}


The following proposition is  similar to Proposition~\ref {prop1}, but here balls are replaced by rectangles. The assumption that $K$ separates $D^\pm(x_0,r)$ in $\overline B(x_0,r)$ is not crucial here and could be removed. We will keep it to simplify the proof of the proposition, since nothing changes for the purpose of proving Theorem \ref{mainTH}.

\begin{proposition}[Density estimates in a rectangle]\label{prop_para} 
Let $(u,K) \in \mathcal A(\Omega)$ be a minimizer of the Griffith functional, and let $x_0 \in K$ and $r>0$ be such that $\overline B(x_0,r)\subset \Omega$,  $\beta_K(x_0,r)\leq 1/10$, and $K$ separates $D^\pm(x_0,r)$ in $\overline B(x_0,r)$.  Let $\{e_1,e_2\}$ be an orthogonal system such that  $L(x_0,r)$ is directed by $e_1$. Then there exists a universal constant $C_*>0$, such that 
\begin{equation*}
\mathcal{H}^1\left({K} \cap \left\{y \in B(x_0,r): \; \frac{r}{5} \leq (y-x_0)_1 \leq \frac{2r}{5} \right\}\right)
\leq \frac{r}{5} + {C_* r\big(\beta_K(x_0,r)+\omega_u(x_0,r)\big)}.
\end{equation*}
\end{proposition}

 \begin{proof} 
We first apply Lemma~\ref {extension} to get a radius $\rho \in (r/2,r)$ and functions $v^{\pm} \in H^1(B(x_0,\rho);\R^2)$ which satisfy the conclusion of that result. In order to construct a competitor for $K$ in $\overline B(x_0,\rho)$, we would like to replace the set $K$ inside the rectangle 
$$R:=\{y \in \R^2 :\; (y-x_0)_1  \in [r/5,2r/5]  \; \text{ and }\; |(y-x_0)_2|\leq  r \beta(x,r)\},$$
by the segment $L(x_0,r)\cap R$ which has length exactly equal to $r/5$. Such a competitor may not separate the balls $D^\pm(x_0,\rho)$ in  $\overline B(x_0,\rho)$. If $D^\pm(x_0,\rho)$ belonged to the same connected component, we could only take $v^+$ (or $v^-$) as a competitor of $u$, introducing a big jump on the boundary of $B(x_0,\rho)$ and removing completely the jump on $K$.  To overcome this problem,  we consider a ``wall set'' (inside the vertical boundaries of $R$) 
$$Z':=\{y \in \R^2  : \;  (y-x_0)_1 \in \{r/5,2r/5\}  \; \text{ and }\; |(y-x_0)_2|\leq  r \beta(x,r)\},$$
as well as a second wall set on $\partial B(x_0,r)$ as before, defined by 
$$Z:=\{y\in \partial B(x_0,\rho) \; : \; {\rm dist}(y,L(x_0,r))\leq r \beta(x_0,r)\}.$$
We define
$$K':=\big[ K\cap \Omega \setminus R\big]\cup Z \cup Z' \cup  \big[L(x_0,r)\cap R\big].$$
Note that $K'$ is now separating the ball  $\overline B(x_0,\rho)$ (thanks to the wall set $Z')$ and 
$$\mathcal{H}^1(K'\cap B(x_0,\rho))\leq \frac{r}{5}+8r \beta(x_0,r)+\mathcal{H}^1(K\cap B(x_0,\rho)\setminus R).$$
 Now we define the competitor for the function $u$ in $\overline B(x_0,\rho)$. To this aim, using that $K'$ separates the ball $\overline B(x_0,\rho)$, we can find two  connected components $V^\pm$ of $B(x_0,\rho)\setminus K'$  whose closure intersect $\mathscr C^\pm$ and define 
$$
v:=
\begin{cases}
v^\pm & \text{ in } V^\pm \\
u & \text{otherwise}.
\end{cases}
$$
Let us recall that $u=v^\pm$ on $\partial B(x_0,\rho)\setminus Z$. Note that the presence of $Z$ in the singular set $K'$ is due to the fact that $v^\pm$ does not match $u$ on $Z$.
The pair $(v,K')$ is then a competitor for $(u,K)$ in $\overline B(x_0,\rho)$, and thus, 
\begin{eqnarray*}
\int_{B(x_0,\rho)\setminus K} \mathbf A e(u):e(u) \, dx &+&\mathcal{H}^1(K\cap   \overline B(x_0,\rho))\\
&\leq& \int_{B(x_0,\rho)\setminus K'} \mathbf A e(v):e(v) \, dx +\mathcal{H}^1(K'\cap   \overline B(x_0,\rho)) \\
&\leq&  C\int_{B(x_0,r)\setminus K} |e(u)|^2 \, dx + {\frac{r}{5}}+8r \beta(x_0,r) +\mathcal{H}^1(K\cap B(x_0,\rho)\setminus R) 
\end{eqnarray*}
from which we deduce that 
$$\mathcal{H}^1\left({K} \cap \left\{y \in B(x_0,r): \; \frac{r}{5} \leq (y-x_0)_1 \leq \frac{2r}{5} \right\}\right)\leq {\frac{r}{5}+Cr\left( \beta(x_0,r) +\omega(x_0,r)\right).}$$
which completes the proof of the result.
\end{proof}


An interesting consequence of the previous density estimates is a selection result of good radii, in a way that the corresponding spheres intersect the set $K$ at only two almost opposite points.

\begin{lemma}[Finding a good radius] \label{card2}
There exists a universal constant $\varepsilon_0>0$ such that the following property holds: let $(u,K) \in \mathcal A(\Omega)$ be a minimizer of the Griffith functional and let $x_0 \in K$ and $r>0$ be such that $\overline B(x_0,r)\subset \Omega$ and
$$\omega_u(x_0,r)+\beta_K(x_0,r) \leq \varepsilon_0.$$ 
If $K$ separates $D^\pm(x_0,r)$ in  $\overline B(x_0,r)$, then there exists  $s\in (r/8,r)$ such that $\#(K \cap \partial B(x_0,s))=2$.
\end{lemma}

\begin{proof} According to Proposition~\ref {prop1}, there exist a universal constant $C>0$ and a radius $\rho \in (r/2,r)$ such that 
$$\int_{B(x_0,\rho)\setminus K} \mathbf Ae(u):e(u) \,dx+ \mathcal{H}^1(K\cap B(x_0,\rho))  \leq 2\rho + C\rho\big(\omega(x_0,r)+\beta(x_0,r)\big).$$
We now fix 
\begin{equation}\label{eps1}
\varepsilon_0:=\min\left\{{\frac{1}{8C}}\;,\; \frac{1}{10}\right\}.
\end{equation}
Using  formula \eqref{coaire1} in Lemma~\ref {coaire1L} and \eqref{eps1}, we get the estimate 
$$\int_{0}^{\rho}\#(K \cap \partial B(x_0,s))\, ds \leq \mathcal{H}^1(K\cap B(x_0,\rho)) \leq \left(2+\frac{1}{8}\right)\rho.$$
Moreover, thanks to the fact that $r\beta(x_0,r)\leq \frac{1}{10}r\leq \frac{1}{5}\rho<\frac{1}{4}\rho$, we have that for all $s \in (\frac{\rho}{4},\rho)$ the circle $\partial B(x_0,s)$ is not totally contained in the strip $\{x \in \overline B(x_0,r): \; \text{dist}(x,L(x_0,r))\leq \beta(x_0,r) r\}$. Therefore, since $K$ is assumed to separate $D^\pm(x_0,r)$ in   $\overline B(x_0,r)$, we deduce that for all $s \in (\frac{\rho}{4},\rho)$
$$\#(K \cap \partial B(x_0,s)) \geq 2.$$
Setting 
$$A:=\{s\in (\rho/4,\rho) \; : \; \#(K\cap \partial B(x_0,s)) \geq 3 \},$$
we obtain
\begin{eqnarray*}
3\mathcal L^1(A) +2\mathcal{L}^1([\rho/4,\rho]\setminus A)& \leq & \int_{\rho/4}^\rho\#({K} \cap \partial B(x_0,s))   ds \\
& \leq & \left(2+\frac{1}{8}\right)\rho,
\end{eqnarray*}
and finally,
$$\mathcal L^1(A)\leq \left(2+\frac18-2\frac{3}{4}\right)\rho=\frac{5\rho}{8}<\frac{3\rho}{4}=\mathcal L^1((\rho/4,\rho)).$$
We then deduce the existence of some $s\in(\rho /4,\rho) \setminus A$, which thus satisfies $\#({K} \cap \partial B(x_0,s)) =2$. Since $\rho \in (r/2,r)$, this radius $s$ then belongs to $(r/8,r)$.
\end{proof}


\subsection{The main extension result}


\label{extensionMain}

The first rough density estimate given by Proposition~\ref {prop1} is based on the property that the crack is always contained in a small strip of thickness $r \beta(x_0,r)$. This enables one to construct a competitor outside a wall set with height of order $r \beta(x_0,r)$. However, in order to bootstrap the estimates on our main quantities, $\beta$ and $\omega$, we need to slightly improve such a density estimate obtaining a remainder of order $r  \eta$, $\eta$ well chosen (of order $\omega(x_0,r)^{1/7}$), instead of $r\beta(x_0,r)$. To this aim, we need a refined version of the extension Lemma~\ref {extension}, in which the boundary value of the competitor displacement is prescribed outside a wall set of height $r \eta$, instead of $r\beta(x_0,r)$. To construct such a suitable small wall set, we first find a nice region in the annulus $B(x_0,\frac{2r}{5})\setminus B(x_0,\frac{r}{5})$ where to cut, i.e. we find some little  boxes in which the set $K$ is totally trapped. This is the purpose of the following lemma. Notice that in all this subsection and in the next one we never use any connectedness assumption on $K$,  but we rather use a separating assumption only.


\begin{lemma}[Selection of cutting squares]  \label{bow-tie}
Let $(u,K) \in \mathcal A(\Omega)$ be a minimizer of the Griffith functional, and let $x_0 \in K$ and  $r>0$ be such that $\overline B(x_0,r)\subset \Omega$ and
$$\omega_u(x_0,r)+\beta_K(x_0,r)  \leq\frac{1}{5{C_*}}\min\left(1, 10^{-2}\theta_0\right),$$ 
where $\theta_0>0$ is the Ahlfors regularity constant of $K$, and $C_*>0$ is the universal constant given in Proposition~\ref {prop_para}. We also assume that $K$ separates $D^\pm(x_0,r)$ in   $\overline B(x_0,r)$.
Let $\{e_1,e_2\}$ be an orthogonal system such that  $L(x_0,r)$ is directed by $e_1$. Then for every $\eta\in(0,10^{-2})$ there exist two points $y_0$ and $z_0$ $\in \R^2$ such that 
\begin{eqnarray}
(y_0-x_0)_1\in (r/5,2r/5),&&  \quad (z_0-x_0)_1\in (-2r/5,-r/5), \label{AS1} \\
|(y_0-x_0)_2|\leq \beta(x_0,r)r,&&  \quad |(z_0-x_0)_2|\leq \beta(x_0,r)r,\label{AS2}
\end{eqnarray}
and
\begin{eqnarray}
& K \cap \{ y \in \R^2: \; |(y-y_0)_1|\leq \eta r \}  \subset  \{y \in \R^2  : \; |(y-y_0)_2|\leq 30 \eta r\},& \label{AS3} \\
& K \cap \{ y \in \R^2: \; |(y-z_0)_1|\leq \eta r \}  \subset  \{y \in \R^2  : \; |(y-z_0)_2|\leq 30\eta r\}.& \label{AS4}
\end{eqnarray}
\end{lemma}

\begin{proof}  It is enough to prove the existence of a point $y_0$ since the argument leading to the existence of $z_0$ is similar. For simplicity, we will denote by $\beta:=\beta(x_0,r)$, $\omega:=\omega(x_0,r)$.

\medskip

We start by finding a good vertical strip in which $K$ has small length.  Let us define the vertical strip
$$S:=\left\{y \in B(x_0,r) : \; \frac{r}{5} \leq (y-x_0)_1\leq \frac{2r}{5}\right\}.$$
Let $\eta<1/10$ and let $N\in \mathbb{N}$, $N\geq 2$, be such that $\frac{1}{5N} \leq   \eta < \frac{1}{5N-5}$. Then $(N-1)/N\geq \frac{1}{2}$ and
\begin{equation}
\frac{\eta}{2}\leq \frac{1}{5N}\leq \eta. \label{etabound}
\end{equation}

We split $S$ into the pairwise disjoint union of $N$ smaller sets $S_1,\ldots,S_N$ defined, for all $k \in \{1,\ldots,N\}$, by
$$S_k =: \left\{y \in S \; : \;  \frac{r}{5} +\frac{k-1}{5N}r \leq (y-x_0)_1 <  \frac{r}{5}+ \frac{k}{5N}r \right\}.$$
Since $\beta\leq 1/10$ we can apply  Proposition~\ref {prop_para} which implies

\begin{equation}
\sum_{k=1}^N \mathcal{H}^1( K \cap S_k) \leq  \mathcal{H}^1( K \cap S)\leq  \frac{(1+E)r}{5}, \label{bigSum} 
\end{equation}
with $E:=5C_*(\beta+\omega)$, where we recall that $\theta_0$ is the Ahlfors regularity constant of $K$, and $C_*>0$ is the universal constant given in Proposition~\ref {prop_para}. As it will be used later, we notice that under our assumptions we have in particular that 
\begin{equation}
E\leq \min(1, 10^{-2}\theta_0  ). \label{EboundD}
\end{equation}

From \eqref{bigSum} we deduce the existence of $k_0 \in \{1,\ldots,N\}$ such that    (see Figure \ref{fig1})
\begin{equation}
\mathcal{H}^1(  K\cap S_{k_0}) \leq\frac{(1+E)r}{5N}.\label{estimateNN}
\end{equation}

\begin{figure}[!ht]
\begin{center}
\scalebox{1} 
{
\begin{pspicture}(0,-5.215)(11.641894,5.215)
\definecolor{color1510}{rgb}{0.6,0.6,0.6}
\definecolor{color1514}{rgb}{0.4,0.4,1.0}
\definecolor{color1526}{rgb}{0.4,0.4,0.4}
\pscircle[linewidth=0.04,dimen=outer](5.54,0.0){5.0}
\usefont{T1}{ptm}{m}{n}
\rput(8.921455,4.545){$B(x_0,r)$}
\psline[linewidth=0.03cm,linecolor=color1510](6.54,-0.2)(6.54,0.4)
\psline[linewidth=0.03cm](6.54,5.2)(6.54,-5.2)
\psline[linewidth=0.03cm](7.54,5.2)(7.54,-5.2)
\psline[linewidth=0.03cm,linecolor=color1510](0.34,0.0)(10.74,0.0)
\pscustom[linewidth=0.04,linecolor=color1514]
{
\newpath
\moveto(0.0,0.12)
\lineto(0.05,0.13)
\curveto(0.075,0.135)(0.145,0.155)(0.19,0.17)
\curveto(0.235,0.185)(0.325,0.205)(0.37,0.21)
\curveto(0.415,0.215)(0.485,0.225)(0.51,0.23)
\curveto(0.535,0.235)(0.585,0.24)(0.61,0.24)
\curveto(0.635,0.24)(0.685,0.24)(0.71,0.24)
\curveto(0.735,0.24)(0.785,0.23)(0.81,0.22)
\curveto(0.835,0.21)(0.88,0.19)(0.9,0.18)
\curveto(0.92,0.17)(0.97,0.145)(1.0,0.13)
\curveto(1.03,0.115)(1.085,0.085)(1.11,0.07)
\curveto(1.135,0.055)(1.19,0.03)(1.22,0.02)
\curveto(1.25,0.01)(1.295,-0.01)(1.31,-0.02)
\curveto(1.325,-0.03)(1.36,-0.045)(1.38,-0.05)
\curveto(1.4,-0.055)(1.445,-0.065)(1.47,-0.07)
\curveto(1.495,-0.075)(1.545,-0.09)(1.57,-0.1)
\curveto(1.595,-0.11)(1.645,-0.125)(1.67,-0.13)
\curveto(1.695,-0.135)(1.745,-0.15)(1.77,-0.16)
\curveto(1.795,-0.17)(1.85,-0.185)(1.88,-0.19)
\curveto(1.91,-0.195)(1.965,-0.21)(1.99,-0.22)
\curveto(2.015,-0.23)(2.065,-0.245)(2.09,-0.25)
\curveto(2.115,-0.255)(2.165,-0.26)(2.19,-0.26)
\curveto(2.215,-0.26)(2.26,-0.265)(2.28,-0.27)
\curveto(2.3,-0.275)(2.345,-0.28)(2.37,-0.28)
\curveto(2.395,-0.28)(2.45,-0.27)(2.48,-0.26)
\curveto(2.51,-0.25)(2.57,-0.235)(2.6,-0.23)
\curveto(2.63,-0.225)(2.685,-0.21)(2.71,-0.2)
\curveto(2.735,-0.19)(2.79,-0.175)(2.82,-0.17)
\curveto(2.85,-0.165)(2.905,-0.15)(2.93,-0.14)
\curveto(2.955,-0.13)(3.0,-0.11)(3.02,-0.1)
\curveto(3.04,-0.09)(3.095,-0.07)(3.13,-0.06)
\curveto(3.165,-0.05)(3.22,-0.03)(3.24,-0.02)
\curveto(3.26,-0.01)(3.305,0.005)(3.33,0.01)
\curveto(3.355,0.015)(3.405,0.03)(3.43,0.04)
\curveto(3.455,0.05)(3.5,0.065)(3.52,0.07)
\curveto(3.54,0.075)(3.59,0.085)(3.62,0.09)
\curveto(3.65,0.095)(3.705,0.11)(3.73,0.12)
\curveto(3.755,0.13)(3.8,0.15)(3.82,0.16)
\curveto(3.84,0.17)(3.885,0.185)(3.91,0.19)
\curveto(3.935,0.195)(3.985,0.21)(4.01,0.22)
\curveto(4.035,0.23)(4.07,0.255)(4.08,0.27)
\curveto(4.09,0.285)(4.105,0.315)(4.12,0.36)
}
\pscustom[linewidth=0.04,linecolor=color1514]
{
\newpath
\moveto(2.28,-0.28)
\lineto(2.23,-0.22)
\curveto(2.205,-0.19)(2.16,-0.14)(2.14,-0.12)
\curveto(2.12,-0.1)(2.09,-0.065)(2.08,-0.05)
\curveto(2.07,-0.035)(2.06,0.015)(2.06,0.05)
\curveto(2.06,0.085)(2.06,0.155)(2.06,0.19)
\curveto(2.06,0.225)(2.045,0.275)(2.03,0.29)
\curveto(2.015,0.305)(1.985,0.325)(1.94,0.34)
}
\pscustom[linewidth=0.04,linecolor=color1514]
{
\newpath
\moveto(2.06,0.1)
\lineto(2.1,0.12)
\curveto(2.12,0.13)(2.17,0.15)(2.2,0.16)
\curveto(2.23,0.17)(2.28,0.185)(2.3,0.19)
\curveto(2.32,0.195)(2.38,0.2)(2.42,0.2)
\curveto(2.46,0.2)(2.525,0.205)(2.55,0.21)
\curveto(2.575,0.215)(2.615,0.225)(2.66,0.24)
}
\pscustom[linewidth=0.04,linecolor=color1514]
{
\newpath
\moveto(3.8,0.14)
\lineto(3.86,0.15)
\curveto(3.89,0.155)(3.995,0.16)(4.07,0.16)
\curveto(4.145,0.16)(4.29,0.145)(4.36,0.13)
\curveto(4.43,0.115)(4.53,0.075)(4.56,0.05)
\curveto(4.59,0.025)(4.665,-0.035)(4.71,-0.07)
\curveto(4.755,-0.105)(4.845,-0.175)(4.89,-0.21)
\curveto(4.935,-0.245)(5.0,-0.295)(5.02,-0.31)
\curveto(5.04,-0.325)(5.085,-0.345)(5.11,-0.35)
\curveto(5.135,-0.355)(5.18,-0.37)(5.2,-0.38)
\curveto(5.22,-0.39)(5.265,-0.405)(5.29,-0.41)
\curveto(5.315,-0.415)(5.365,-0.42)(5.39,-0.42)
\curveto(5.415,-0.42)(5.46,-0.415)(5.48,-0.41)
\curveto(5.5,-0.405)(5.545,-0.4)(5.57,-0.4)
\curveto(5.595,-0.4)(5.645,-0.4)(5.67,-0.4)
\curveto(5.695,-0.4)(5.745,-0.395)(5.77,-0.39)
\curveto(5.795,-0.385)(5.85,-0.375)(5.88,-0.37)
\curveto(5.91,-0.365)(5.96,-0.355)(5.98,-0.35)
\curveto(6.0,-0.345)(6.04,-0.335)(6.06,-0.33)
\curveto(6.08,-0.325)(6.115,-0.31)(6.13,-0.3)
\curveto(6.145,-0.29)(6.185,-0.26)(6.21,-0.24)
\curveto(6.235,-0.22)(6.3,-0.19)(6.34,-0.18)
\curveto(6.38,-0.17)(6.455,-0.145)(6.49,-0.13)
\curveto(6.525,-0.115)(6.575,-0.075)(6.59,-0.05)
\curveto(6.605,-0.025)(6.64,0.01)(6.66,0.02)
\curveto(6.68,0.03)(6.72,0.05)(6.74,0.06)
\curveto(6.76,0.07)(6.8,0.095)(6.82,0.11)
\curveto(6.84,0.125)(6.88,0.145)(6.9,0.15)
\curveto(6.92,0.155)(6.965,0.16)(6.99,0.16)
\curveto(7.015,0.16)(7.07,0.16)(7.1,0.16)
\curveto(7.13,0.16)(7.195,0.16)(7.23,0.16)
\curveto(7.265,0.16)(7.335,0.155)(7.37,0.15)
\curveto(7.405,0.145)(7.465,0.14)(7.49,0.14)
\curveto(7.515,0.14)(7.565,0.14)(7.59,0.14)
\curveto(7.615,0.14)(7.665,0.135)(7.69,0.13)
\curveto(7.715,0.125)(7.765,0.11)(7.79,0.1)
\curveto(7.815,0.09)(7.88,0.07)(7.92,0.06)
\curveto(7.96,0.05)(8.02,0.03)(8.04,0.02)
\curveto(8.06,0.01)(8.105,-0.005)(8.13,-0.01)
\curveto(8.155,-0.015)(8.2,-0.025)(8.22,-0.03)
\curveto(8.24,-0.035)(8.285,-0.045)(8.31,-0.05)
\curveto(8.335,-0.055)(8.395,-0.06)(8.43,-0.06)
\curveto(8.465,-0.06)(8.53,-0.06)(8.56,-0.06)
\curveto(8.59,-0.06)(8.645,-0.065)(8.67,-0.07)
\curveto(8.695,-0.075)(8.745,-0.085)(8.77,-0.09)
\curveto(8.795,-0.095)(8.84,-0.105)(8.86,-0.11)
\curveto(8.88,-0.115)(8.93,-0.13)(8.96,-0.14)
\curveto(8.99,-0.15)(9.06,-0.18)(9.1,-0.2)
\curveto(9.14,-0.22)(9.2,-0.25)(9.22,-0.26)
\curveto(9.24,-0.27)(9.285,-0.28)(9.31,-0.28)
\curveto(9.335,-0.28)(9.395,-0.28)(9.43,-0.28)
\curveto(9.465,-0.28)(9.535,-0.28)(9.57,-0.28)
\curveto(9.605,-0.28)(9.665,-0.275)(9.69,-0.27)
\curveto(9.715,-0.265)(9.765,-0.255)(9.79,-0.25)
\curveto(9.815,-0.245)(9.865,-0.235)(9.89,-0.23)
}
\pscustom[linewidth=0.04,linecolor=color1514]
{
\newpath
\moveto(8.66,-0.08)
\lineto(8.69,-0.03)
\curveto(8.705,-0.005)(8.735,0.055)(8.75,0.09)
\curveto(8.765,0.125)(8.775,0.185)(8.77,0.21)
\curveto(8.765,0.235)(8.745,0.285)(8.73,0.31)
\curveto(8.715,0.335)(8.68,0.37)(8.66,0.38)
\curveto(8.64,0.39)(8.6,0.41)(8.58,0.42)
\curveto(8.56,0.43)(8.53,0.445)(8.5,0.46)
}
\pscustom[linewidth=0.04,linecolor=color1514]
{
\newpath
\moveto(8.76,0.2)
\lineto(8.79,0.23)
\curveto(8.805,0.245)(8.84,0.28)(8.86,0.3)
\curveto(8.88,0.32)(8.915,0.355)(8.93,0.37)
\curveto(8.945,0.385)(8.985,0.4)(9.01,0.4)
\curveto(9.035,0.4)(9.095,0.4)(9.13,0.4)
\curveto(9.165,0.4)(9.23,0.4)(9.26,0.4)
\curveto(9.29,0.4)(9.35,0.4)(9.38,0.4)
\curveto(9.41,0.4)(9.505,0.4)(9.57,0.4)
\curveto(9.635,0.4)(9.79,0.4)(9.88,0.4)
\curveto(9.97,0.4)(10.115,0.4)(10.17,0.4)
\curveto(10.225,0.4)(10.305,0.4)(10.33,0.4)
\curveto(10.355,0.4)(10.405,0.395)(10.43,0.39)
\curveto(10.455,0.385)(10.505,0.375)(10.53,0.37)
\curveto(10.555,0.365)(10.6,0.355)(10.62,0.35)
\curveto(10.64,0.345)(10.68,0.335)(10.7,0.33)
\curveto(10.72,0.325)(10.76,0.31)(10.78,0.3)
\curveto(10.8,0.29)(10.84,0.275)(10.86,0.27)
\curveto(10.88,0.265)(10.925,0.26)(10.95,0.26)
\curveto(10.975,0.26)(11.015,0.255)(11.06,0.24)
}
\usefont{T1}{ptm}{m}{n}
\rput(11.291455,0.585){$K$}
\usefont{T1}{ptm}{m}{n}
\rput(7.0614552,3.765){$S$}
\psline[linewidth=0.03cm,linecolor=color1526](6.74,1.2)(6.74,-1.4)
\psline[linewidth=0.03cm,linecolor=color1526](6.94,1.2)(6.94,-1.4)
\psline[linewidth=0.03cm,linecolor=color1526](7.14,1.2)(7.14,-1.4)
\psline[linewidth=0.03cm,linecolor=color1526](7.34,1.2)(7.34,-1.4)
\psline[linewidth=0.03cm,linecolor=color1510,arrowsize=0.05291667cm 2.0,arrowlength=1.4,arrowinset=0.4]{<-}(7.06,0.84)(4.74,2.2)
\usefont{T1}{ptm}{m}{n}
\rput(4.161455,2.525){$S_{k_0}$}
\psdots[dotsize=0.12,linecolor=color1510](5.54,0.0)
\psdots[dotsize=0.12,linecolor=color1510](5.54,0.0)
\psdots[dotsize=0.24](5.54,0.0)
\usefont{T1}{ptm}{m}{n}
\rput(5.2514553,0.465){$x_0$}
\pscustom[linewidth=0.04,linecolor=color1514]
{
\newpath
\moveto(3.22,0.32)
\lineto(3.26,0.35)
\curveto(3.28,0.365)(3.32,0.385)(3.34,0.39)
\curveto(3.36,0.395)(3.395,0.41)(3.41,0.42)
\curveto(3.425,0.43)(3.465,0.44)(3.49,0.44)
\curveto(3.515,0.44)(3.565,0.44)(3.59,0.44)
\curveto(3.615,0.44)(3.675,0.44)(3.71,0.44)
\curveto(3.745,0.44)(3.79,0.44)(3.82,0.44)
}
\pscustom[linewidth=0.04,linecolor=color1514]
{
\newpath
\moveto(3.36,-0.24)
\lineto(3.41,-0.22)
\curveto(3.435,-0.21)(3.505,-0.19)(3.55,-0.18)
\curveto(3.595,-0.17)(3.67,-0.16)(3.7,-0.16)
\curveto(3.73,-0.16)(3.785,-0.16)(3.81,-0.16)
\curveto(3.835,-0.16)(3.895,-0.17)(3.93,-0.18)
\curveto(3.965,-0.19)(4.045,-0.215)(4.09,-0.23)
\curveto(4.135,-0.245)(4.2,-0.27)(4.22,-0.28)
\curveto(4.24,-0.29)(4.285,-0.3)(4.31,-0.3)
\curveto(4.335,-0.3)(4.365,-0.3)(4.38,-0.3)
}
\pscustom[linewidth=0.04,linecolor=color1514]
{
\newpath
\moveto(7.0,0.28)
\lineto(7.04,0.29)
\curveto(7.06,0.295)(7.095,0.31)(7.11,0.32)
\curveto(7.125,0.33)(7.16,0.345)(7.18,0.35)
\curveto(7.2,0.355)(7.25,0.36)(7.28,0.36)
\curveto(7.31,0.36)(7.365,0.355)(7.39,0.35)
\curveto(7.415,0.345)(7.46,0.335)(7.48,0.33)
\curveto(7.5,0.325)(7.545,0.32)(7.57,0.32)
\curveto(7.595,0.32)(7.65,0.33)(7.68,0.34)
\curveto(7.71,0.35)(7.755,0.365)(7.8,0.38)
}
\end{pspicture} 
}
\end{center}
\caption{{The choice of $S_{k_0}$.}}	
\label{fig1}
\end{figure}
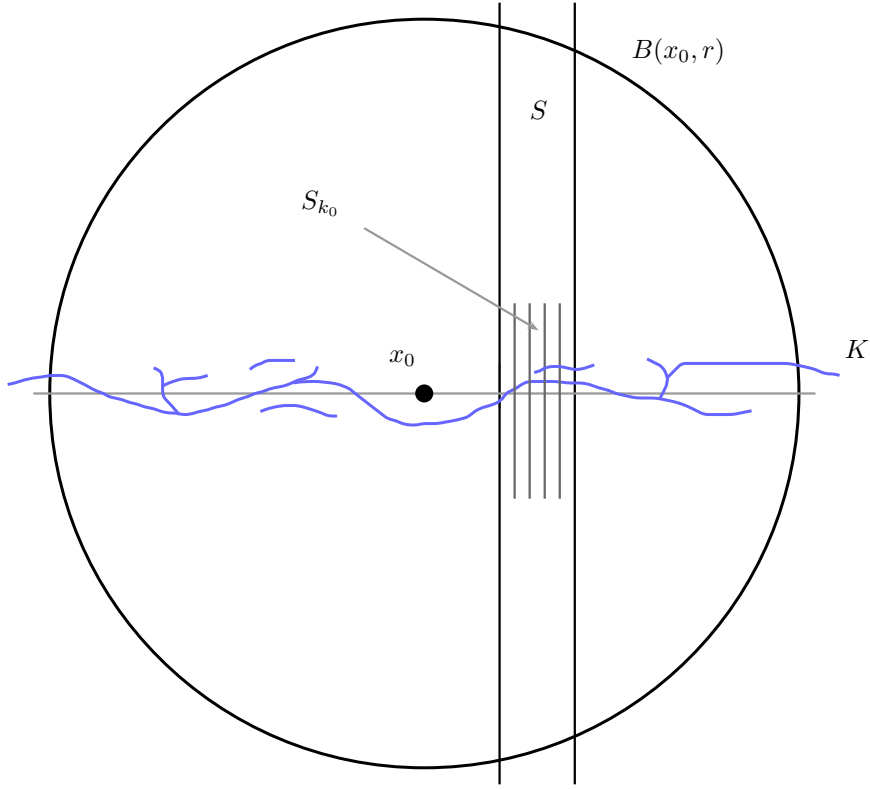

By the separation property of $K$, one can find inside  $ K\cap S_{k_0}$,  an injective Lipschitz curve $\Gamma$ connecting both vertical sides of $\partial S_{k_0}$ (see Lemma \ref{existGamma}). In particular, we have
\begin{equation}\label{eq:gamma}
\frac{r}{5N}\leq \mathcal{H}^1(\Gamma)\leq \frac{(1+E)r}{5N},
\end{equation}
and thus \eqref{estimateNN} leads to 
\begin{equation}
\mathcal{H}^1(K\cap S_{k_0}\setminus \Gamma)\leq \frac{Er}{5N}\leq   \eta  E r. \label{lenG}
\end{equation}

Thanks to the length estimate \eqref{eq:gamma}, if we denote by $z,z' \in \partial S_{k_0}$ both points of $\Gamma$ on the boundary of $S_{k_0}$, we have in particular, for every point $y \in \Gamma$, 
$$|y-z|\leq \mathcal{H}^1(\Gamma)\leq \frac{(1+E)}{5N}r \leq \frac{2}{5N}r,$$
because $E\leq 1$. In other words,

\begin{equation}
\sup_{y\in \Gamma} |y-z|\leq \frac{2}{5N}r\leq 2 \eta r. \label{estimate100}
\end{equation}

We now finally give a bound on the distance from the points of  $K$ to the curve $\Gamma$ in  a strip slightly thinner than $S_{k_0}$, by use of the Ahlfors-regularity of $K$.  For that purpose we let $S'\subset S_{k_0}$ be  defined by 
$$S':=\left\{y \in S_{k_0} : \;  \frac{r}{5} +\frac{k_0-1}{5N}r +\delta r\leq (y-x_0)_1 \leq  \frac{r}{5}+ \frac{k_0}{5N}r -\delta r\; \right\},$$
with $\delta:=\frac{2\eta E}{\theta_0}$. Since $E \leq 10^{-2}\theta_0$, we deduce that $\delta \leq \frac{10^{-1}}{5N}$, so that $S'$ is not empty.

We claim that 
 \begin{equation}
\sup_{y\in K\cap S'}{\rm dist}(y,\Gamma) \leq \frac{2 \eta E}{\theta_0}r  . \label{aprO}
\end{equation}
Indeed, if $y\in K \cap S'$  is  such that $d:={\rm dist}(y,\Gamma)> \delta r=\frac{2 \eta E}{\theta_0}r$, then $B(y,\delta r)\subset S_{k_0} \setminus \Gamma$ and, by Ahlfors regularity,
$$\mathcal{H}^1(K\cap B(y,\delta r))\geq \theta_0 \delta r=2 \eta E r, $$
which is a contradiction with \eqref{lenG} and proves \eqref{aprO}. 

To conclude, gathering \eqref{aprO} and \eqref{estimate100}, we have obtained 
\begin{equation}
\sup_{y\in K\cap S'} |y-z|\leq \frac{2\eta E}{\theta_0} r+2\eta r\leq 3\eta r,\label{distBB}
\end{equation}
since  $2E/\theta_0\leq 1$. Therefore,   if we define $y_0$ as being the middle point of the segment $[z,z+\frac{r}{5N} e_1]$ (in particular in the middle of $S'$), the conclusion \eqref{AS1} and \eqref{AS2} of the lemma are satisfied.

Next, we notice that by \eqref{etabound}, the  width of $S'$ is exactly  
$$\frac{1}{5N}r-2\delta r = \frac{1}{5N}r -4\frac{\eta E}{\theta_0} r \geq \left(\frac{\eta}{2}-4\frac{\eta E}{\theta_0}\right)r\geq \frac{\eta r}{4}, $$
provided that $E\leq \frac{\theta_0}{16}$, which is valid thanks to \eqref{EboundD}   (see Figure \ref{fig2}). Consequently, using \eqref{distBB} and that $(y_0)_2=z_2$, we deduce that  with this choice of $y_0$ it holds
$$K \cap \big\{ y \in \R^2 : \; |(y-y_0)_1|\leq \eta/8 r \big\}  \subset  K\cap S'\subset  \big\{y \in \R^2 \; : \; |(y-y_0)_2|\leq 3\eta r \big\}.$$

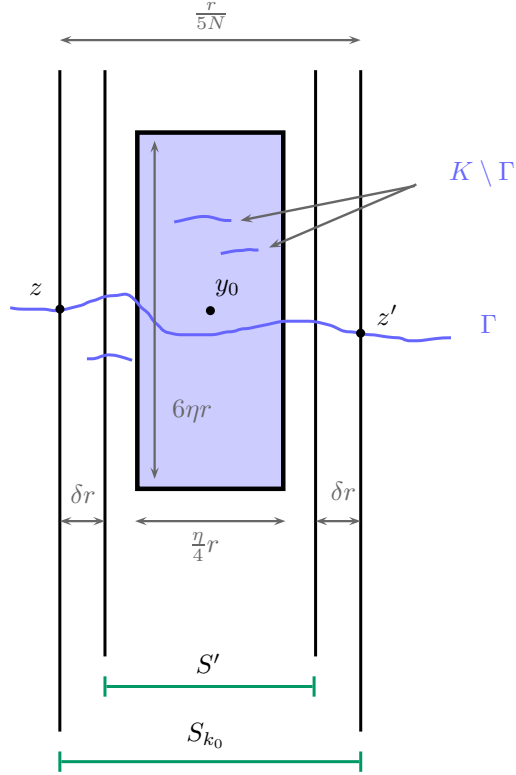
\begin{figure}[!ht]
\begin{center}
\scalebox{1} 
{
\begin{pspicture}(0,-5.0591993)(7.9818945,5.0791993)
\definecolor{color2590}{rgb}{0.4,0.4,0.4}
\definecolor{color1789b}{rgb}{0.8,0.8,1.0}
\definecolor{color1603}{rgb}{0.0,0.6,0.4}
\definecolor{color1757}{rgb}{0.4,0.4,1.0}
\psframe[linewidth=0.06,dimen=outer,fillstyle=solid,fillcolor=color1789b](3.66,3.3608007)(1.66,-1.4391992)
\psline[linewidth=0.04cm](0.66,4.160801)(0.66,-4.6391993)
\psline[linewidth=0.04cm](4.66,4.160801)(4.66,-4.6391993)
\psline[linewidth=0.04cm](1.26,4.160801)(1.26,-3.6391993)
\psline[linewidth=0.04cm](4.06,4.160801)(4.06,-3.6391993)
\psline[linewidth=0.04cm,linecolor=color1603,tbarsize=0.07055555cm 5.0]{|*-|*}(0.66,-5.0391994)(4.66,-5.0391994)
\psline[linewidth=0.04cm,linecolor=color1603,tbarsize=0.07055555cm 5.0]{|*-|}(1.26,-4.0391994)(4.06,-4.0391994)
\usefont{T1}{ptm}{m}{n}
\rput(2.601455,-4.694199){$S_{k_0}$}
\usefont{T1}{ptm}{m}{n}
\rput(2.6314552,-3.7341993){$S'$}
\psline[linewidth=0.03cm,linecolor=color2590,arrowsize=0.05291667cm 2.0,arrowlength=1.4,arrowinset=0.4]{<->}(0.66,-1.8391992)(1.26,-1.8391992)
\psline[linewidth=0.03cm,linecolor=color2590,arrowsize=0.05291667cm 2.0,arrowlength=1.4,arrowinset=0.4]{<->}(4.06,-1.8391992)(4.66,-1.8391992)
\usefont{T1}{ptm}{m}{n}
\rput(0.9714551,-1.4741992){\color{color2590}$\delta r$}
\usefont{T1}{ptm}{m}{n}
\rput(4.431455,-1.4341992){\color{color2590}$\delta r$}
\psline[linewidth=0.03cm,linecolor=color2590,arrowsize=0.05291667cm 2.0,arrowlength=1.4,arrowinset=0.4]{<->}(0.66,4.5608006)(4.66,4.5608006)
\usefont{T1}{ptm}{m}{n}
\rput(2.6714551,4.885801){\color{color2590}$\frac{r}{5N}$}
\usefont{T1}{ptm}{m}{n}
\rput(0.34145507,1.2458007){$z$}
\usefont{T1}{ptm}{m}{n}
\rput(5.011455,0.90580076){$z'$}
\usefont{T1}{ptm}{m}{n}
\rput(6.361455,0.78580076){\color{color1757}$\Gamma$}
\pscustom[linewidth=0.04,linecolor=color1757]
{
\newpath
\moveto(2.8,1.7208008)
\lineto(2.85,1.7308007)
\curveto(2.875,1.7358007)(2.925,1.7458007)(2.95,1.7508007)
\curveto(2.975,1.7558007)(3.025,1.7608008)(3.05,1.7608008)
\curveto(3.075,1.7608008)(3.12,1.7658008)(3.14,1.7708008)
\curveto(3.16,1.7758008)(3.2,1.7758008)(3.22,1.7708008)
\curveto(3.24,1.7658008)(3.28,1.7658008)(3.3,1.7708008)
}
\pscustom[linewidth=0.04,linecolor=color1757]
{
\newpath
\moveto(2.18,2.1408007)
\lineto(2.23,2.1508007)
\curveto(2.255,2.1558008)(2.32,2.1708007)(2.36,2.1808007)
\curveto(2.4,2.1908007)(2.475,2.2058008)(2.51,2.210801)
\curveto(2.545,2.2158008)(2.61,2.2158008)(2.64,2.210801)
\curveto(2.67,2.2058008)(2.73,2.1908007)(2.76,2.1808007)
\curveto(2.79,2.1708007)(2.845,2.1608007)(2.87,2.1608007)
\curveto(2.895,2.1608007)(2.925,2.1608007)(2.94,2.1608007)
}
\pscustom[linewidth=0.04,linecolor=color1757]
{
\newpath
\moveto(1.62,0.30080077)
\lineto(1.56,0.32080078)
\curveto(1.53,0.33080077)(1.46,0.35080078)(1.42,0.36080077)
\curveto(1.38,0.3708008)(1.31,0.3758008)(1.28,0.3708008)
\curveto(1.25,0.36580077)(1.2,0.35080078)(1.18,0.3408008)
\curveto(1.16,0.33080077)(1.115,0.32080078)(1.09,0.32080078)
\curveto(1.065,0.32080078)(1.035,0.32080078)(1.02,0.32080078)
}
\psline[linewidth=0.032cm,linecolor=color2590,arrowsize=0.05291667cm 2.0,arrowlength=1.4,arrowinset=0.4]{<-}(3.44,1.8008008)(5.4,2.6408007)
\psline[linewidth=0.032cm,linecolor=color2590,arrowsize=0.05291667cm 2.0,arrowlength=1.4,arrowinset=0.4]{<-}(3.06,2.1608007)(5.38,2.6208007)
\usefont{T1}{ptm}{m}{n}
\rput(6.2714553,2.7858007){\color{color1757}$K\setminus \Gamma$}
\psdots[dotsize=0.12](2.66,0.96080077)
\usefont{T1}{ptm}{m}{n}
\rput(2.8914552,1.2658008){$y_0$}
\pscustom[linewidth=0.04,linecolor=color1757]
{
\newpath
\moveto(0.0,1.0008007)
\lineto(0.05,0.9908008)
\curveto(0.075,0.9858008)(0.14,0.9808008)(0.18,0.9808008)
\curveto(0.22,0.9808008)(0.29,0.9808008)(0.32,0.9808008)
\curveto(0.35,0.9808008)(0.405,0.9808008)(0.43,0.9808008)
\curveto(0.455,0.9808008)(0.5,0.97580075)(0.52,0.97080076)
\curveto(0.54,0.96580076)(0.585,0.96080077)(0.61,0.96080077)
\curveto(0.635,0.96080077)(0.68,0.96580076)(0.7,0.97080076)
\curveto(0.72,0.97580075)(0.76,0.9858008)(0.78,0.9908008)
\curveto(0.8,0.9958008)(0.84,1.0058007)(0.86,1.0108008)
\curveto(0.88,1.0158008)(0.92,1.0258008)(0.94,1.0308008)
\curveto(0.96,1.0358008)(0.995,1.0508008)(1.01,1.0608008)
\curveto(1.025,1.0708008)(1.065,1.0858008)(1.09,1.0908008)
\curveto(1.115,1.0958008)(1.16,1.1108007)(1.18,1.1208007)
\curveto(1.2,1.1308007)(1.245,1.1458008)(1.27,1.1508008)
\curveto(1.295,1.1558008)(1.345,1.1608008)(1.37,1.1608008)
\curveto(1.395,1.1608008)(1.44,1.1658008)(1.46,1.1708008)
\curveto(1.48,1.1758008)(1.515,1.1808008)(1.53,1.1808008)
\curveto(1.545,1.1808008)(1.575,1.1708008)(1.59,1.1608008)
\curveto(1.605,1.1508008)(1.635,1.1258007)(1.65,1.1108007)
\curveto(1.665,1.0958008)(1.695,1.0608008)(1.71,1.0408008)
\curveto(1.725,1.0208008)(1.755,0.9908008)(1.77,0.9808008)
\curveto(1.785,0.97080076)(1.815,0.9458008)(1.83,0.9308008)
\curveto(1.845,0.9158008)(1.87,0.8808008)(1.88,0.8608008)
\curveto(1.89,0.84080076)(1.92,0.8058008)(1.94,0.7908008)
\curveto(1.96,0.77580076)(1.995,0.7408008)(2.01,0.72080076)
\curveto(2.025,0.7008008)(2.065,0.6758008)(2.09,0.6708008)
\curveto(2.115,0.6658008)(2.16,0.65580076)(2.18,0.65080076)
\curveto(2.2,0.64580077)(2.245,0.6408008)(2.27,0.6408008)
\curveto(2.295,0.6408008)(2.345,0.6408008)(2.37,0.6408008)
\curveto(2.395,0.6408008)(2.445,0.6408008)(2.47,0.6408008)
\curveto(2.495,0.6408008)(2.545,0.6408008)(2.57,0.6408008)
\curveto(2.595,0.6408008)(2.645,0.6408008)(2.67,0.6408008)
\curveto(2.695,0.6408008)(2.74,0.64580077)(2.76,0.65080076)
\curveto(2.78,0.65580076)(2.82,0.6658008)(2.84,0.6708008)
\curveto(2.86,0.6758008)(2.905,0.6808008)(2.93,0.6808008)
\curveto(2.955,0.6808008)(3.0,0.6908008)(3.02,0.7008008)
\curveto(3.04,0.71080077)(3.085,0.72080076)(3.11,0.72080076)
\curveto(3.135,0.72080076)(3.18,0.72580075)(3.2,0.7308008)
\curveto(3.22,0.7358008)(3.265,0.7458008)(3.29,0.7508008)
\curveto(3.315,0.7558008)(3.365,0.7658008)(3.39,0.77080077)
\curveto(3.415,0.77580076)(3.46,0.78580076)(3.48,0.7908008)
\curveto(3.5,0.7958008)(3.545,0.8058008)(3.57,0.8108008)
\curveto(3.595,0.8158008)(3.645,0.8208008)(3.67,0.8208008)
\curveto(3.695,0.8208008)(3.745,0.8208008)(3.77,0.8208008)
\curveto(3.795,0.8208008)(3.845,0.8208008)(3.87,0.8208008)
\curveto(3.895,0.8208008)(3.945,0.8208008)(3.97,0.8208008)
\curveto(3.995,0.8208008)(4.045,0.8158008)(4.07,0.8108008)
\curveto(4.095,0.8058008)(4.14,0.7958008)(4.16,0.7908008)
\curveto(4.18,0.78580076)(4.22,0.77580076)(4.24,0.77080077)
\curveto(4.26,0.7658008)(4.3,0.7508008)(4.32,0.7408008)
\curveto(4.34,0.7308008)(4.38,0.71080077)(4.4,0.7008008)
\curveto(4.42,0.6908008)(4.465,0.6808008)(4.49,0.6808008)
\curveto(4.515,0.6808008)(4.57,0.6758008)(4.6,0.6708008)
\curveto(4.63,0.6658008)(4.68,0.65580076)(4.7,0.65080076)
\curveto(4.72,0.64580077)(4.765,0.6408008)(4.79,0.6408008)
\curveto(4.815,0.6408008)(4.865,0.6358008)(4.89,0.6308008)
\curveto(4.915,0.6258008)(4.98,0.6158008)(5.02,0.6108008)
\curveto(5.06,0.6058008)(5.125,0.60080075)(5.15,0.60080075)
\curveto(5.175,0.60080075)(5.215,0.59080076)(5.23,0.5808008)
\curveto(5.245,0.5708008)(5.285,0.5608008)(5.31,0.5608008)
\curveto(5.335,0.5608008)(5.385,0.5608008)(5.41,0.5608008)
\curveto(5.435,0.5608008)(5.49,0.5658008)(5.52,0.5708008)
\curveto(5.55,0.5758008)(5.625,0.58580077)(5.67,0.59080076)
\curveto(5.715,0.59580076)(5.785,0.60080075)(5.86,0.60080075)
}
\psline[linewidth=0.03cm,linecolor=color2590,arrowsize=0.05291667cm 2.0,arrowlength=1.4,arrowinset=0.4]{<->}(1.92,3.2208009)(1.92,-1.2791992)
\usefont{T1}{ptm}{m}{n}
\rput(2.4214551,-0.3741992){\color{color2590}$6\eta r$}
\psline[linewidth=0.03cm,linecolor=color2590,arrowsize=0.05291667cm 2.0,arrowlength=1.4,arrowinset=0.4]{<->}(1.66,-1.8391992)(3.66,-1.8391992)
\usefont{T1}{ptm}{m}{n}
\rput(2.581455,-2.174199){\color{color2590}$\frac{\eta}{4}r$}
\psdots[dotsize=0.12](0.66,0.9808008)
\psdots[dotsize=0.12](4.66,0.66080076)
\end{pspicture} 
}

\end{center}
\caption{{The set $K$   is trapped into a rectangle of size $\simeq \eta r$.}}	
\label{fig2}
\end{figure}

The proof of the lemma follows by relabeling $\eta/8$ as $\eta$.
\end{proof}

We are now in the position to establish an improved version of the extension lemma. Its proof is similar to that of Proposition~\ref {prop1}, the difference being the definition of the wall set that  has now size $\eta r $ instead of $r\beta(x_0,r)$.

\begin{lemma}[Extension Lemma] \label{extLem}   
Let $(u,K) \in \mathcal A(\Omega)$ be a minimizer of the Griffith functional, and let $x_0 \in K$ and $r>0$ be such that $\overline B(x_0,r)\subset \Omega$ and
$$\omega_u(x_0,r)+\beta_K(x_0,r)  \leq \frac{1}{5C_*}\min\left(1, 10^{-2}\theta_0\right),$$ 
where $\theta_0$ is the Ahlfors regularity constant of $K$ and $C_*>0$ is the universal constant given in Proposition~\ref {prop_para}. We also assume that $K$ separates $D^\pm(x_0,r)$ in   $\overline B(x_0,r)$.

Then for all $0< \eta < 10^{-4}$ there exist:
\begin{itemize}
\item {an open} rectangle  $U$ such that $\overline B(x_0,r/5)\subset U \subset  {B}(x_0,r)$, 
\item a wall set (i.e. union of two vertical segments) $\Sigma \subset \partial U$ such that 
$K \cap \partial U \subset \Sigma$, $u \in H^1(\partial U \setminus \Sigma;\R^2)$ 
and $\mathcal{H}^1(\Sigma)\leq 120\eta r$.
\end{itemize}

In addition, if $K' \subset \Omega$ is a closed set such that 
$K' \setminus U=K \setminus U$ and $D^\pm(x_0,r/5)$ are contained in two different connected components of $U \setminus K'$, then  there exists a function   $v \in H^1(\Omega \setminus K';\R^2)$ such that 
$$v=u \text{ on } (\Omega\setminus U)\setminus \Sigma$$
and
\begin{equation}\label{ext_rect}
\int_{U \setminus K'} |e(v)|^2 \, dx \leq \frac{C}{\eta^6} \int_{B(x_0,r)\setminus K} |e(u)|^2 \, dx,
\end{equation}
where $C>0$ is universal. 
\end{lemma}

\begin{proof}
We denote by $\{e_1,e_2\}$ an orthogonal system such that $L(x_0,r)$ is directed by $e_1$ and we apply Lemma~\ref {bow-tie}, which gives the existence of   two points $y_0$ and $z_0$ $\in B(x_0,2r/5)\setminus \overline B(x_0,r/5)$ satisfying \eqref{AS1}--\eqref{AS4}. In order to construct the rectangle  $U$ and the wall set $\Sigma$, we need to introduce a  domain $A$ which is a  ``rectangular annulus'' of thickness of order $\eta r$.

\medskip

{\bf Step 1: Construction of a rectangular annulus A.} The vertical parts of $A$ are defined as being the following open rectangles
$$V_1 :=  \left\{x \in \R^2 :\;  |(x-y_0)_1|<  \eta r ,\; \;  \;   |(x-x_0)_2|< \frac{1}{3}r\right\}$$
and 
$$V_2:=\left\{x \in \R^2 :\;  |(x-z_0)_1|<  \eta r, \; \;  \;   |(x-x_0)_2|< \frac{1}{3}r\right\}.$$
Notice that $(y_0-x_0)_1\leq \frac{2}{5}r$ and $\eta r\leq 10^{-2} r$, so that 
$$\sup_{y \in V_1} (y-x_0)_1 \leq \frac{2}{5}r+10^{-2}r=\frac{41}{100}r,$$
which means that the right corners of $V_1$ have a distance to $x_0$ bounded by $\sqrt{\frac{41^2}{100^2}+\frac{1}{9}}r< r$ and therefore 
$$V_1\subset B(x_0,r).$$
By symmetry,  $V_2\subset B(x_0,r)$ as well.

Now  the horizontal parts of $A$ are  given by the following open rectangles
$$H_1:= \left\{x\in \R^2 : \;   (z_0)_1 -\eta r  < (x-x_0)_1  < (y_0)_1 +\eta r, \quad  \frac{1}{3}r-\eta r < (x-x_0)_2 < \frac{1}{3}r  \right\}$$
and 
$$H_2:= \left\{x \in \R^2 : \;   (z_0)_1 -\eta r  < (x-x_0)_1  < (y_0)_1 +\eta r, \quad  -\frac{1}{3}r < (x-x_0)_2 < -\frac{1}{3}r +\eta r \right\}.$$

Note that   the four rectangles $V_1$, $V_2$, $H_1$, and $H_2$ are all contained in the ball $B(x_0,r)$. Finally, we define the ``rectangular annulus'' $A$ by
$$A:=  V_1  \cup V_2\cup H_1\cup  H_2,$$
 which satisfies $\overline B(x_0,r/5) \subset A \subset B(x_0,r)$, because $\frac{1}{3}r-\eta r\geq \frac{1}{3}r-\frac{1}{100} r=\frac{97}{300}r> \frac{r}{5}$  (see Figure \ref{fig3}).

Next, we consider the two closed boxes 
$$T_1 :=  \Big\{x \in \R^2 : \;   | (x-y_0)_1 | \leq  \eta r \text{ and } |(x-y_0)_2| \leq 30\eta r \Big\}\subset V_1\subset A,$$
$$T_2 := \Big\{x \in \R^2 : \;   | (x-z_0)_1 | \leq  \eta r \text{ and } |(x-z_0)_2| \leq 30\eta r \Big\}\subset V_2\subset A,$$

the main point being that $K\cap A \subset T_1\cup T_2$. 

\begin{figure}[!ht]
\begin{center}
\scalebox{1} 
{
\begin{pspicture}(0,-5.0)(11.641894,5.0)
\definecolor{color558}{rgb}{0.4,0.4,0.4}
\definecolor{color292b}{rgb}{0.8,0.8,1.0}
\definecolor{color298}{rgb}{0.4,0.4,1.0}
\definecolor{color306}{rgb}{0.6,0.6,0.6}
\psline[linewidth=0.03cm,linecolor=color558](2.94,2.0)(2.94,-2.4)
\psline[linewidth=0.03cm,linecolor=color558](3.34,1.6)(3.34,-2.0)
\psline[linewidth=0.03cm,linecolor=color558](6.94,1.6)(6.9549623,-1.9849999)
\psline[linewidth=0.03cm,linecolor=color558](7.34,2.0)(7.34,-2.4)
\psframe[linewidth=0.04,dimen=outer,fillstyle=solid,fillcolor=color292b](3.34,0.44)(2.94,-0.6)
\psframe[linewidth=0.04,dimen=outer,fillstyle=solid,fillcolor=color292b](7.34,0.64)(6.92,-0.36)
\pscircle[linewidth=0.04,dimen=outer](5.54,0.0){5.0}
\usefont{T1}{ptm}{m}{n}
\rput(8.921455,4.545){$B(x_0,r)$}
\pscustom[linewidth=0.04,linecolor=color298]
{
\newpath
\moveto(0.0,0.12)
\lineto(0.05,0.13)
\curveto(0.075,0.135)(0.145,0.155)(0.19,0.17)
\curveto(0.235,0.185)(0.325,0.205)(0.37,0.21)
\curveto(0.415,0.215)(0.485,0.225)(0.51,0.23)
\curveto(0.535,0.235)(0.585,0.24)(0.61,0.24)
\curveto(0.635,0.24)(0.685,0.24)(0.71,0.24)
\curveto(0.735,0.24)(0.785,0.23)(0.81,0.22)
\curveto(0.835,0.21)(0.88,0.19)(0.9,0.18)
\curveto(0.92,0.17)(0.97,0.145)(1.0,0.13)
\curveto(1.03,0.115)(1.085,0.085)(1.11,0.07)
\curveto(1.135,0.055)(1.19,0.03)(1.22,0.02)
\curveto(1.25,0.01)(1.295,-0.01)(1.31,-0.02)
\curveto(1.325,-0.03)(1.36,-0.045)(1.38,-0.05)
\curveto(1.4,-0.055)(1.445,-0.065)(1.47,-0.07)
\curveto(1.495,-0.075)(1.545,-0.09)(1.57,-0.1)
\curveto(1.595,-0.11)(1.645,-0.125)(1.67,-0.13)
\curveto(1.695,-0.135)(1.745,-0.15)(1.77,-0.16)
\curveto(1.795,-0.17)(1.85,-0.185)(1.88,-0.19)
\curveto(1.91,-0.195)(1.965,-0.21)(1.99,-0.22)
\curveto(2.015,-0.23)(2.065,-0.245)(2.09,-0.25)
\curveto(2.115,-0.255)(2.165,-0.26)(2.19,-0.26)
\curveto(2.215,-0.26)(2.26,-0.265)(2.28,-0.27)
\curveto(2.3,-0.275)(2.345,-0.28)(2.37,-0.28)
\curveto(2.395,-0.28)(2.45,-0.27)(2.48,-0.26)
\curveto(2.51,-0.25)(2.57,-0.235)(2.6,-0.23)
\curveto(2.63,-0.225)(2.685,-0.21)(2.71,-0.2)
\curveto(2.735,-0.19)(2.79,-0.175)(2.82,-0.17)
\curveto(2.85,-0.165)(2.905,-0.15)(2.93,-0.14)
\curveto(2.955,-0.13)(3.0,-0.11)(3.02,-0.1)
\curveto(3.04,-0.09)(3.095,-0.07)(3.13,-0.06)
\curveto(3.165,-0.05)(3.22,-0.03)(3.24,-0.02)
\curveto(3.26,-0.01)(3.305,0.005)(3.33,0.01)
\curveto(3.355,0.015)(3.405,0.03)(3.43,0.04)
\curveto(3.455,0.05)(3.5,0.065)(3.52,0.07)
\curveto(3.54,0.075)(3.59,0.085)(3.62,0.09)
\curveto(3.65,0.095)(3.705,0.11)(3.73,0.12)
\curveto(3.755,0.13)(3.8,0.15)(3.82,0.16)
\curveto(3.84,0.17)(3.885,0.185)(3.91,0.19)
\curveto(3.935,0.195)(3.985,0.21)(4.01,0.22)
\curveto(4.035,0.23)(4.07,0.255)(4.08,0.27)
\curveto(4.09,0.285)(4.105,0.315)(4.12,0.36)
}
\pscustom[linewidth=0.04,linecolor=color298]
{
\newpath
\moveto(2.28,-0.28)
\lineto(2.23,-0.22)
\curveto(2.205,-0.19)(2.16,-0.14)(2.14,-0.12)
\curveto(2.12,-0.1)(2.09,-0.065)(2.08,-0.05)
\curveto(2.07,-0.035)(2.06,0.015)(2.06,0.05)
\curveto(2.06,0.085)(2.06,0.155)(2.06,0.19)
\curveto(2.06,0.225)(2.045,0.275)(2.03,0.29)
\curveto(2.015,0.305)(1.985,0.325)(1.94,0.34)
}
\pscustom[linewidth=0.04,linecolor=color298]
{
\newpath
\moveto(2.06,0.1)
\lineto(2.1,0.12)
\curveto(2.12,0.13)(2.17,0.15)(2.2,0.16)
\curveto(2.23,0.17)(2.28,0.185)(2.3,0.19)
\curveto(2.32,0.195)(2.38,0.2)(2.42,0.2)
\curveto(2.46,0.2)(2.525,0.205)(2.55,0.21)
\curveto(2.575,0.215)(2.615,0.225)(2.66,0.24)
}
\pscustom[linewidth=0.04,linecolor=color298]
{
\newpath
\moveto(3.8,0.14)
\lineto(3.86,0.15)
\curveto(3.89,0.155)(3.995,0.16)(4.07,0.16)
\curveto(4.145,0.16)(4.29,0.145)(4.36,0.13)
\curveto(4.43,0.115)(4.53,0.075)(4.56,0.05)
\curveto(4.59,0.025)(4.665,-0.035)(4.71,-0.07)
\curveto(4.755,-0.105)(4.845,-0.175)(4.89,-0.21)
\curveto(4.935,-0.245)(5.0,-0.295)(5.02,-0.31)
\curveto(5.04,-0.325)(5.085,-0.345)(5.11,-0.35)
\curveto(5.135,-0.355)(5.18,-0.37)(5.2,-0.38)
\curveto(5.22,-0.39)(5.265,-0.405)(5.29,-0.41)
\curveto(5.315,-0.415)(5.365,-0.42)(5.39,-0.42)
\curveto(5.415,-0.42)(5.46,-0.415)(5.48,-0.41)
\curveto(5.5,-0.405)(5.545,-0.4)(5.57,-0.4)
\curveto(5.595,-0.4)(5.645,-0.4)(5.67,-0.4)
\curveto(5.695,-0.4)(5.745,-0.395)(5.77,-0.39)
\curveto(5.795,-0.385)(5.85,-0.375)(5.88,-0.37)
\curveto(5.91,-0.365)(5.96,-0.355)(5.98,-0.35)
\curveto(6.0,-0.345)(6.04,-0.335)(6.06,-0.33)
\curveto(6.08,-0.325)(6.115,-0.31)(6.13,-0.3)
\curveto(6.145,-0.29)(6.185,-0.26)(6.21,-0.24)
\curveto(6.235,-0.22)(6.3,-0.19)(6.34,-0.18)
\curveto(6.38,-0.17)(6.455,-0.145)(6.49,-0.13)
\curveto(6.525,-0.115)(6.575,-0.075)(6.59,-0.05)
\curveto(6.605,-0.025)(6.64,0.01)(6.66,0.02)
\curveto(6.68,0.03)(6.72,0.05)(6.74,0.06)
\curveto(6.76,0.07)(6.8,0.095)(6.82,0.11)
\curveto(6.84,0.125)(6.88,0.145)(6.9,0.15)
\curveto(6.92,0.155)(6.965,0.16)(6.99,0.16)
\curveto(7.015,0.16)(7.07,0.16)(7.1,0.16)
\curveto(7.13,0.16)(7.195,0.16)(7.23,0.16)
\curveto(7.265,0.16)(7.335,0.155)(7.37,0.15)
\curveto(7.405,0.145)(7.465,0.14)(7.49,0.14)
\curveto(7.515,0.14)(7.565,0.14)(7.59,0.14)
\curveto(7.615,0.14)(7.665,0.135)(7.69,0.13)
\curveto(7.715,0.125)(7.765,0.11)(7.79,0.1)
\curveto(7.815,0.09)(7.88,0.07)(7.92,0.06)
\curveto(7.96,0.05)(8.02,0.03)(8.04,0.02)
\curveto(8.06,0.01)(8.105,-0.005)(8.13,-0.01)
\curveto(8.155,-0.015)(8.2,-0.025)(8.22,-0.03)
\curveto(8.24,-0.035)(8.285,-0.045)(8.31,-0.05)
\curveto(8.335,-0.055)(8.395,-0.06)(8.43,-0.06)
\curveto(8.465,-0.06)(8.53,-0.06)(8.56,-0.06)
\curveto(8.59,-0.06)(8.645,-0.065)(8.67,-0.07)
\curveto(8.695,-0.075)(8.745,-0.085)(8.77,-0.09)
\curveto(8.795,-0.095)(8.84,-0.105)(8.86,-0.11)
\curveto(8.88,-0.115)(8.93,-0.13)(8.96,-0.14)
\curveto(8.99,-0.15)(9.06,-0.18)(9.1,-0.2)
\curveto(9.14,-0.22)(9.2,-0.25)(9.22,-0.26)
\curveto(9.24,-0.27)(9.285,-0.28)(9.31,-0.28)
\curveto(9.335,-0.28)(9.395,-0.28)(9.43,-0.28)
\curveto(9.465,-0.28)(9.535,-0.28)(9.57,-0.28)
\curveto(9.605,-0.28)(9.665,-0.275)(9.69,-0.27)
\curveto(9.715,-0.265)(9.765,-0.255)(9.79,-0.25)
\curveto(9.815,-0.245)(9.865,-0.235)(9.89,-0.23)
}
\pscustom[linewidth=0.04,linecolor=color298]
{
\newpath
\moveto(8.66,-0.08)
\lineto(8.69,-0.03)
\curveto(8.705,-0.005)(8.735,0.055)(8.75,0.09)
\curveto(8.765,0.125)(8.775,0.185)(8.77,0.21)
\curveto(8.765,0.235)(8.745,0.285)(8.73,0.31)
\curveto(8.715,0.335)(8.68,0.37)(8.66,0.38)
\curveto(8.64,0.39)(8.6,0.41)(8.58,0.42)
\curveto(8.56,0.43)(8.53,0.445)(8.5,0.46)
}
\pscustom[linewidth=0.04,linecolor=color298]
{
\newpath
\moveto(8.76,0.2)
\lineto(8.79,0.23)
\curveto(8.805,0.245)(8.84,0.28)(8.86,0.3)
\curveto(8.88,0.32)(8.915,0.355)(8.93,0.37)
\curveto(8.945,0.385)(8.985,0.4)(9.01,0.4)
\curveto(9.035,0.4)(9.095,0.4)(9.13,0.4)
\curveto(9.165,0.4)(9.23,0.4)(9.26,0.4)
\curveto(9.29,0.4)(9.35,0.4)(9.38,0.4)
\curveto(9.41,0.4)(9.505,0.4)(9.57,0.4)
\curveto(9.635,0.4)(9.79,0.4)(9.88,0.4)
\curveto(9.97,0.4)(10.115,0.4)(10.17,0.4)
\curveto(10.225,0.4)(10.305,0.4)(10.33,0.4)
\curveto(10.355,0.4)(10.405,0.395)(10.43,0.39)
\curveto(10.455,0.385)(10.505,0.375)(10.53,0.37)
\curveto(10.555,0.365)(10.6,0.355)(10.62,0.35)
\curveto(10.64,0.345)(10.68,0.335)(10.7,0.33)
\curveto(10.72,0.325)(10.76,0.31)(10.78,0.3)
\curveto(10.8,0.29)(10.84,0.275)(10.86,0.27)
\curveto(10.88,0.265)(10.925,0.26)(10.95,0.26)
\curveto(10.975,0.26)(11.015,0.255)(11.06,0.24)
}
\usefont{T1}{ptm}{m}{n}
\rput(11.291455,0.585){$K$}
\psdots[dotsize=0.12,linecolor=color306](5.54,0.0)
\psdots[dotsize=0.12,linecolor=color306](5.54,0.0)
\psdots[dotsize=0.16](5.54,0.0)
\usefont{T1}{ptm}{m}{n}
\rput(5.391455,0.405){$x_0$}
\pscustom[linewidth=0.04,linecolor=color298]
{
\newpath
\moveto(3.36,-0.24)
\lineto(3.41,-0.22)
\curveto(3.435,-0.21)(3.505,-0.19)(3.55,-0.18)
\curveto(3.595,-0.17)(3.67,-0.16)(3.7,-0.16)
\curveto(3.73,-0.16)(3.785,-0.16)(3.81,-0.16)
\curveto(3.835,-0.16)(3.895,-0.17)(3.93,-0.18)
\curveto(3.965,-0.19)(4.045,-0.215)(4.09,-0.23)
\curveto(4.135,-0.245)(4.2,-0.27)(4.22,-0.28)
\curveto(4.24,-0.29)(4.285,-0.3)(4.31,-0.3)
\curveto(4.335,-0.3)(4.365,-0.3)(4.38,-0.3)
}
\pscustom[linewidth=0.04,linecolor=color298]
{
\newpath
\moveto(7.0,0.28)
\lineto(7.04,0.29)
\curveto(7.06,0.295)(7.095,0.31)(7.11,0.32)
\curveto(7.125,0.33)(7.16,0.345)(7.18,0.35)
\curveto(7.2,0.355)(7.25,0.36)(7.28,0.36)
\curveto(7.31,0.36)(7.365,0.355)(7.39,0.35)
\curveto(7.415,0.345)(7.46,0.335)(7.48,0.33)
\curveto(7.5,0.325)(7.545,0.32)(7.57,0.32)
\curveto(7.595,0.32)(7.65,0.33)(7.68,0.34)
\curveto(7.71,0.35)(7.755,0.365)(7.8,0.38)
}
\psdots[dotsize=0.12](7.14,0.16)
\psdots[dotsize=0.12](3.14,-0.06)
\psline[linewidth=0.03cm,linecolor=color558](3.34,-2.0)(6.94,-2.0)
\psline[linewidth=0.03cm,linecolor=color558](7.34,-2.4)(2.94,-2.4)
\psline[linewidth=0.03cm,linecolor=color558](2.94,2.0)(7.34,2.0)
\psline[linewidth=0.03cm,linecolor=color558](3.34,1.6)(6.94,1.6)
\usefont{T1}{ptm}{m}{n}
\rput(8.371455,-1.095){$y_0$}
\usefont{T1}{ptm}{m}{n}
\rput(2.1614552,-0.895){$z_0$}
\psline[linewidth=0.03cm,arrowsize=0.05291667cm 2.0,arrowlength=1.4,arrowinset=0.4]{->}(2.16,-0.56)(3.14,-0.08)
\psline[linewidth=0.03cm,arrowsize=0.05291667cm 2.0,arrowlength=1.4,arrowinset=0.4]{->}(8.06,-0.74)(7.2,0.06)
\usefont{T1}{ptm}{m}{n}
\rput(5.051455,1.785){$A$}
\psline[linewidth=0.03cm,arrowsize=0.05291667cm 2.0,arrowlength=1.4,arrowinset=0.4]{<-}(7.16,0.52)(8.12,1.56)
\psline[linewidth=0.03cm,arrowsize=0.05291667cm 2.0,arrowlength=1.4,arrowinset=0.4]{<-}(3.08,0.28)(2.18,1.44)
\usefont{T1}{ptm}{m}{n}
\rput(8.451455,1.825){$T_1$}
\usefont{T1}{ptm}{m}{n}
\rput(2.051455,1.705){$T_2$}
\psline[linewidth=0.03cm,linecolor=color558,arrowsize=0.05291667cm 2.0,arrowlength=1.4,arrowinset=0.4]{<->}(5.14,-2.0)(5.14,-2.4)
\usefont{T1}{ptm}{m}{n}
\rput(5.571455,-2.195){\color{color558}$\delta r$}
\end{pspicture} 
}

\end{center}
\caption{{The rectangular annulus $A$.}}	
\label{fig3}
\end{figure}
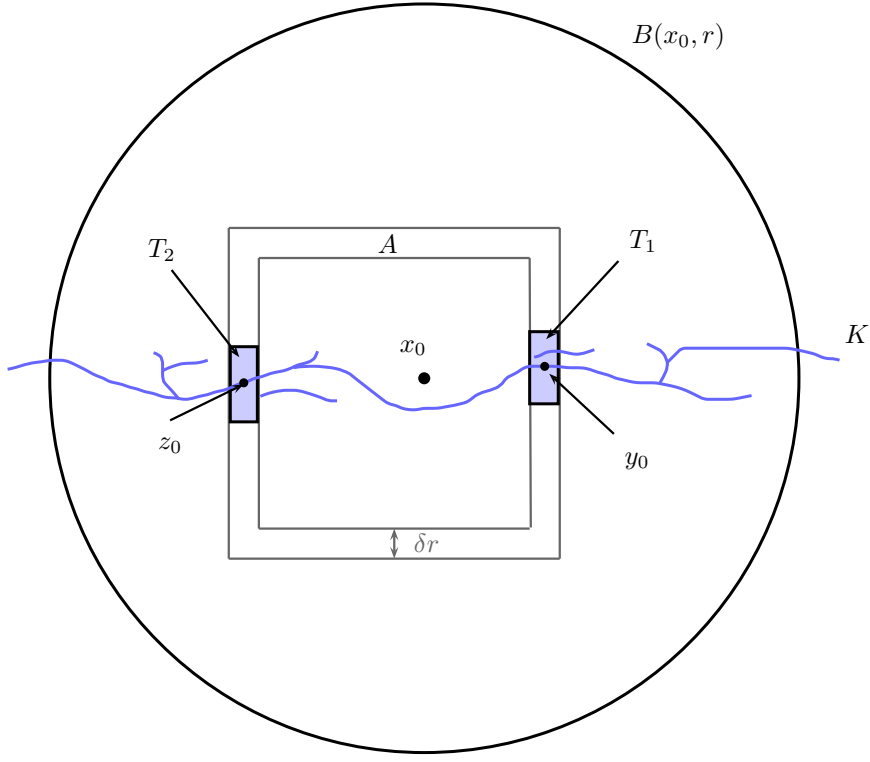

Let us finally consider the subset of $A$ outside the cutting boxes,
$$A' :=A \setminus (T_1 \cup T_2),$$
and let $A^\pm$ be both connected components of $A'$. The open sets $A^\pm$ are Lipschitz domains, and they are actually unions of vertical and horizontal rectangles of thickness of order $\eta$ and  lengths of order $r$ (notice that $30 \eta \leq10^{-2}$). In addition, since by construction we have $K\cap A^\pm=\emptyset$, 
it follows that $u \in H^1(A^\pm;\R^2)$ and that the Korn  inequality (see Lemma~\ref {KornRect}) applies in each rectangle composing $A^\pm$. Therefore, there exist two skew-symmetric matrices $R^\pm$ such that
\begin{equation}
\int_{A^\pm}|\nabla u-R^\pm|^2 \,dx \leq \frac{C }{\eta^5}\int_{A^\pm}|e(u)|^2 \,dx\leq \frac{C}{\eta^5}\int_{B(x_0,r) \setminus K}|e(u)|^2 \, dx, \label{R+}
\end{equation}
for some constant $C>0$ universal, where $\eta^{-5}$ appears estimating the distance between the skew-symmetric matrices in the intersection of two overlapping vertical and horizontal rectangles.

\medskip

{\bf Step 2: Construction of the rectangle $U$.} Let us denote by $R:=R^+ {\bf 1}_{A^+} + R^-{\bf 1}_{A^-}$, then 
$$\int_{A'}|\nabla u-R|^2 \,dx \leq   \frac{C}{\eta^5}\int_{B(x_0,r)\setminus K}|e(u)|^2 \, dx.$$

For any $t\in [-\eta r, \eta r]$ we denote the vertical  line passing through $y(t):= y_0+t e_1$ by $L_{t}:=y(t)+\R e_2$. According to Fubini's Theorem, we have
$$\int_{-\eta r}^{\eta r}\int_{L_t \cap A'} |\nabla u-R|^2  \, d\mathcal{H}^1 \,dt\leq    \frac{C}{\eta^5}\int_{B(x_0,r) \setminus K}|e(u)|^2 \, dx.$$
We can thus find $t_0 \in [-\eta r, \eta r ]$ such that $u \in H^1(L_{t_0} \cap A';\R^2)$ and
$$2\eta r \int_{L_{t_0} \cap A'} |\nabla u-R|^2 \, d\mathcal{H}^1 \leq   \frac{C}{\eta^5}\int_{B(x_0,r) \setminus K}|e(u)|^2\,  dx.$$
We perform the same argument at the point $z_0$, finding some $t_1\in [-\eta  r , \eta r]$ such that, denoting by $L_{t_1}$ the line $z_0+t_1 e_1 + \R e_2$,  one has $u \in H^1(L_{t_1} \cap A';\R^2)$ and
$$\int_{L_{t_1}\cap A'} |\nabla u-R|^2 \,d\mathcal H^1  \leq   \frac{C}{r\eta^6}\int_{B(x_0,r) \setminus K}|e(u)|^2 dx.$$

Arguing similarly for the top horizontal part of $A^+$,   we get a horizontal line $L_{H^+}$ such that  $u \in H^1(L_{H^+} \cap A^+;\R^2)$ and
$$\int_{L_{H^+} \cap A^+} |\nabla u-R^+|^2 \,d\mathcal H^1  \leq   \frac{C}{r\eta^6}\int_{B(x_0,r) \setminus K}|e(u)|^2 dx.$$
The vertical line $L_{t_0}$ intersects $L_{H^+}$ at a single point $a^+_0$, and $L_{t_1}$ intersects $L_{H^+}$ at another single point $a^+_1$. 

We perform a similar construction on the lower part $A^-$ of $A'$ which leads to  another horizontal line $L_{H^-}$ such that  $u \in H^1(L_{H^-} \cap A^-;\R^2)$ and
$$\int_{L_{H^-} \cap A^-} |\nabla u-R^-|^2 \,d\mathcal H^1  \leq   \frac{C}{r\eta^6}\int_{B(x_0,r) \setminus K}|e(u)|^2 dx.$$
The vertical line $L_{t_0}$ intersects $L_{H^-}$ at a single point $a^-_0$, and $L_{t_1}$ intersects $L_{H^-}$ at another single point $a^-_1$. 

Finally, we define $U$ as the rectangle with vertices $(a^-_0,a^-_1,a^+_0,a^+_1)$ (See Figure \ref{fig4}) and we define $\Sigma$ as
$$\Sigma := (T_1 \cup T_2) \cap \partial U,$$
so that $K\cap \partial U\subset \Sigma$, $\mathcal{H}^1(\Sigma)=120\eta r$, and 
\begin{equation}
\int_{\partial U \setminus \Sigma} |\nabla u-R|^2  \, d\mathcal{H}^1 \,dt = \int_{\partial U\cap A'} |\nabla u-R|^2  \, d\mathcal{H}^1 \,dt  \leq    \frac{C}{r \eta^6}\int_{B(x_0,r) \setminus K}|e(u)|^2 \, dx. \label{goodbound}
\end{equation}

\begin{figure}[!ht]
\begin{center}
\scalebox{1} 
{
\begin{pspicture}(0,-3.638047)(10.641894,3.638047)
\definecolor{color567}{rgb}{0.4,0.4,0.4}
\definecolor{color392b}{rgb}{0.8,0.8,1.0}
\definecolor{color398}{rgb}{0.4,0.4,1.0}
\psline[linewidth=0.03cm,linecolor=color567,linestyle=dashed,dash=0.16cm 0.16cm](2.14,3.223683)(2.14,-3.1603515)
\psline[linewidth=0.03cm,linecolor=color567,linestyle=dashed,dash=0.16cm 0.16cm](2.7204542,2.6433163)(2.7204542,-2.5799847)
\psline[linewidth=0.03cm,linecolor=color567,linestyle=dashed,dash=0.16cm 0.16cm](7.94,2.6396484)(7.94,-2.5603516)
\psline[linewidth=0.03cm,linecolor=color567,linestyle=dashed,dash=0.16cm 0.16cm](8.54,3.2396483)(8.54,-3.1603515)
\psframe[linewidth=0.04,dimen=outer,fillstyle=solid,fillcolor=color392b](2.74,0.41964844)(2.12,-0.80035156)
\psframe[linewidth=0.04,dimen=outer,fillstyle=solid,fillcolor=color392b](8.56,0.73964846)(7.94,-0.54035157)
\pscustom[linewidth=0.04,linecolor=color398]
{
\newpath
\moveto(0.0,-0.040351562)
\lineto(0.05,-0.030351562)
\curveto(0.075,-0.025351562)(0.145,-0.0053515625)(0.19,0.009648438)
\curveto(0.235,0.024648437)(0.325,0.04464844)(0.37,0.049648438)
\curveto(0.415,0.054648437)(0.485,0.064648435)(0.51,0.06964844)
\curveto(0.535,0.07464844)(0.585,0.079648435)(0.61,0.079648435)
\curveto(0.635,0.079648435)(0.685,0.079648435)(0.71,0.079648435)
\curveto(0.735,0.079648435)(0.785,0.06964844)(0.81,0.05964844)
\curveto(0.835,0.049648438)(0.88,0.029648438)(0.9,0.019648438)
\curveto(0.92,0.009648438)(0.97,-0.015351563)(1.0,-0.030351562)
\curveto(1.03,-0.04535156)(1.085,-0.075351566)(1.11,-0.09035156)
\curveto(1.135,-0.10535156)(1.19,-0.13035156)(1.22,-0.14035156)
\curveto(1.25,-0.15035157)(1.295,-0.17035156)(1.31,-0.18035156)
\curveto(1.325,-0.19035156)(1.36,-0.20535156)(1.38,-0.21035156)
\curveto(1.4,-0.21535157)(1.445,-0.22535156)(1.47,-0.23035157)
\curveto(1.495,-0.23535156)(1.545,-0.25035155)(1.57,-0.26035157)
\curveto(1.595,-0.27035156)(1.645,-0.28535157)(1.67,-0.29035157)
\curveto(1.695,-0.29535156)(1.745,-0.31035155)(1.77,-0.32035157)
\curveto(1.795,-0.33035156)(1.85,-0.34535158)(1.88,-0.35035157)
\curveto(1.91,-0.35535157)(1.965,-0.37035155)(1.99,-0.38035157)
\curveto(2.015,-0.39035156)(2.065,-0.40535155)(2.09,-0.41035157)
\curveto(2.115,-0.41535157)(2.165,-0.42035156)(2.19,-0.42035156)
\curveto(2.215,-0.42035156)(2.26,-0.42535156)(2.28,-0.43035156)
\curveto(2.3,-0.43535155)(2.345,-0.44035158)(2.37,-0.44035158)
\curveto(2.395,-0.44035158)(2.45,-0.43035156)(2.48,-0.42035156)
\curveto(2.51,-0.41035157)(2.57,-0.39535156)(2.6,-0.39035156)
\curveto(2.63,-0.38535157)(2.685,-0.37035155)(2.71,-0.36035156)
\curveto(2.735,-0.35035157)(2.79,-0.33535156)(2.82,-0.33035156)
\curveto(2.85,-0.32535157)(2.905,-0.31035155)(2.93,-0.30035156)
\curveto(2.955,-0.29035157)(3.0,-0.27035156)(3.02,-0.26035157)
\curveto(3.04,-0.25035155)(3.095,-0.23035157)(3.13,-0.22035156)
\curveto(3.165,-0.21035156)(3.22,-0.19035156)(3.24,-0.18035156)
\curveto(3.26,-0.17035156)(3.305,-0.15535156)(3.33,-0.15035157)
\curveto(3.355,-0.14535156)(3.405,-0.13035156)(3.43,-0.12035156)
\curveto(3.455,-0.11035156)(3.5,-0.09535156)(3.52,-0.09035156)
\curveto(3.54,-0.085351564)(3.59,-0.075351566)(3.62,-0.07035156)
\curveto(3.65,-0.06535156)(3.705,-0.050351564)(3.73,-0.040351562)
\curveto(3.755,-0.030351562)(3.8,-0.010351563)(3.82,0.0)
\curveto(3.84,0.009648438)(3.885,0.024648437)(3.91,0.029648438)
\curveto(3.935,0.034648437)(3.985,0.049648438)(4.01,0.05964844)
\curveto(4.035,0.06964844)(4.07,0.094648436)(4.08,0.10964844)
\curveto(4.09,0.12464844)(4.105,0.15464844)(4.12,0.19964844)
}
\pscustom[linewidth=0.04,linecolor=color398]
{
\newpath
\moveto(1.88,-1.0403515)
\lineto(1.83,-0.98035157)
\curveto(1.805,-0.95035154)(1.76,-0.9003516)(1.74,-0.88035154)
\curveto(1.72,-0.86035156)(1.69,-0.82535154)(1.68,-0.81035155)
\curveto(1.67,-0.79535156)(1.66,-0.74535155)(1.66,-0.7103516)
\curveto(1.66,-0.67535156)(1.66,-0.60535157)(1.66,-0.57035154)
\curveto(1.66,-0.5353516)(1.645,-0.48535156)(1.63,-0.47035158)
\curveto(1.615,-0.45535156)(1.585,-0.43535155)(1.54,-0.42035156)
}
\pscustom[linewidth=0.04,linecolor=color398]
{
\newpath
\moveto(1.66,-0.6603516)
\lineto(1.7,-0.64035153)
\curveto(1.72,-0.63035154)(1.77,-0.61035156)(1.8,-0.6003516)
\curveto(1.83,-0.5903516)(1.88,-0.57535154)(1.9,-0.57035154)
\curveto(1.92,-0.56535155)(1.98,-0.56035155)(2.02,-0.56035155)
\curveto(2.06,-0.56035155)(2.125,-0.55535156)(2.15,-0.55035156)
\curveto(2.175,-0.54535156)(2.215,-0.5353516)(2.26,-0.5203516)
}
\pscustom[linewidth=0.04,linecolor=color398]
{
\newpath
\moveto(3.2,0.17964844)
\lineto(3.26,0.18964843)
\curveto(3.29,0.19464844)(3.395,0.19964844)(3.47,0.19964844)
\curveto(3.545,0.19964844)(3.69,0.18464844)(3.76,0.16964844)
\curveto(3.83,0.15464844)(3.93,0.11464844)(3.96,0.08964844)
\curveto(3.99,0.064648435)(4.065,0.0046484377)(4.11,-0.030351562)
\curveto(4.155,-0.06535156)(4.245,-0.13535157)(4.29,-0.17035156)
\curveto(4.335,-0.20535156)(4.4,-0.25535157)(4.42,-0.27035156)
\curveto(4.44,-0.28535157)(4.485,-0.30535156)(4.51,-0.31035155)
\curveto(4.535,-0.31535158)(4.58,-0.33035156)(4.6,-0.34035155)
\curveto(4.62,-0.35035157)(4.665,-0.36535156)(4.69,-0.37035155)
\curveto(4.715,-0.37535155)(4.765,-0.38035157)(4.79,-0.38035157)
\curveto(4.815,-0.38035157)(4.86,-0.37535155)(4.88,-0.37035155)
\curveto(4.9,-0.36535156)(4.945,-0.36035156)(4.97,-0.36035156)
\curveto(4.995,-0.36035156)(5.045,-0.36035156)(5.07,-0.36035156)
\curveto(5.095,-0.36035156)(5.145,-0.35535157)(5.17,-0.35035157)
\curveto(5.195,-0.34535158)(5.25,-0.33535156)(5.28,-0.33035156)
\curveto(5.31,-0.32535157)(5.36,-0.31535158)(5.38,-0.31035155)
\curveto(5.4,-0.30535156)(5.44,-0.29535156)(5.46,-0.29035157)
\curveto(5.48,-0.28535157)(5.515,-0.27035156)(5.53,-0.26035157)
\curveto(5.545,-0.25035155)(5.585,-0.22035156)(5.61,-0.20035157)
\curveto(5.635,-0.18035156)(5.7,-0.15035157)(5.74,-0.14035156)
\curveto(5.78,-0.13035156)(5.855,-0.10535156)(5.89,-0.09035156)
\curveto(5.925,-0.075351566)(5.975,-0.035351563)(5.99,-0.010351563)
\curveto(6.005,0.0146484375)(6.04,0.049648438)(6.06,0.05964844)
\curveto(6.08,0.06964844)(6.12,0.08964844)(6.14,0.09964844)
\curveto(6.16,0.10964844)(6.2,0.13464844)(6.22,0.14964844)
\curveto(6.24,0.16464844)(6.28,0.18464844)(6.3,0.18964843)
\curveto(6.32,0.19464844)(6.365,0.19964844)(6.39,0.19964844)
\curveto(6.415,0.19964844)(6.47,0.19964844)(6.5,0.19964844)
\curveto(6.53,0.19964844)(6.595,0.19964844)(6.63,0.19964844)
\curveto(6.665,0.19964844)(6.735,0.19464844)(6.77,0.18964843)
\curveto(6.805,0.18464844)(6.865,0.17964844)(6.89,0.17964844)
\curveto(6.915,0.17964844)(6.965,0.17964844)(6.99,0.17964844)
\curveto(7.015,0.17964844)(7.065,0.17464843)(7.09,0.16964844)
\curveto(7.115,0.16464844)(7.165,0.14964844)(7.19,0.13964844)
\curveto(7.215,0.12964843)(7.28,0.10964844)(7.32,0.09964844)
\curveto(7.36,0.08964844)(7.42,0.06964844)(7.44,0.05964844)
\curveto(7.46,0.049648438)(7.505,0.034648437)(7.53,0.029648438)
\curveto(7.555,0.024648437)(7.6,0.0146484375)(7.62,0.009648438)
\curveto(7.64,0.0046484377)(7.685,-0.0053515625)(7.71,-0.010351563)
\curveto(7.735,-0.015351563)(7.795,-0.020351563)(7.83,-0.020351563)
\curveto(7.865,-0.020351563)(7.93,-0.020351563)(7.96,-0.020351563)
\curveto(7.99,-0.020351563)(8.045,-0.025351562)(8.07,-0.030351562)
\curveto(8.095,-0.035351563)(8.145,-0.04535156)(8.17,-0.050351564)
\curveto(8.195,-0.055351563)(8.24,-0.06535156)(8.26,-0.07035156)
\curveto(8.28,-0.075351566)(8.33,-0.09035156)(8.36,-0.100351565)
\curveto(8.39,-0.11035156)(8.46,-0.14035156)(8.5,-0.16035156)
\curveto(8.54,-0.18035156)(8.6,-0.21035156)(8.62,-0.22035156)
\curveto(8.64,-0.23035157)(8.685,-0.24035156)(8.71,-0.24035156)
\curveto(8.735,-0.24035156)(8.795,-0.24035156)(8.83,-0.24035156)
\curveto(8.865,-0.24035156)(8.935,-0.24035156)(8.97,-0.24035156)
\curveto(9.005,-0.24035156)(9.065,-0.23535156)(9.09,-0.23035157)
\curveto(9.115,-0.22535156)(9.165,-0.21535157)(9.19,-0.21035156)
\curveto(9.215,-0.20535156)(9.265,-0.19535156)(9.29,-0.19035156)
}
\pscustom[linewidth=0.04,linecolor=color398]
{
\newpath
\moveto(8.16,-0.16035156)
\lineto(8.19,-0.13035156)
\curveto(8.205,-0.115351565)(8.24,-0.08035156)(8.26,-0.06035156)
\curveto(8.28,-0.040351562)(8.315,-0.0053515625)(8.33,0.009648438)
\curveto(8.345,0.024648437)(8.385,0.039648436)(8.41,0.039648436)
\curveto(8.435,0.039648436)(8.495,0.039648436)(8.53,0.039648436)
\curveto(8.565,0.039648436)(8.63,0.039648436)(8.66,0.039648436)
\curveto(8.69,0.039648436)(8.75,0.039648436)(8.78,0.039648436)
\curveto(8.81,0.039648436)(8.905,0.039648436)(8.97,0.039648436)
\curveto(9.035,0.039648436)(9.19,0.039648436)(9.28,0.039648436)
\curveto(9.37,0.039648436)(9.515,0.039648436)(9.57,0.039648436)
\curveto(9.625,0.039648436)(9.705,0.039648436)(9.73,0.039648436)
\curveto(9.755,0.039648436)(9.805,0.034648437)(9.83,0.029648438)
\curveto(9.855,0.024648437)(9.905,0.0146484375)(9.93,0.009648438)
\curveto(9.955,0.0046484377)(10.0,-0.0053515625)(10.02,-0.010351563)
\curveto(10.04,-0.015351563)(10.08,-0.025351562)(10.1,-0.030351562)
\curveto(10.12,-0.035351563)(10.16,-0.050351564)(10.18,-0.06035156)
\curveto(10.2,-0.07035156)(10.24,-0.085351564)(10.26,-0.09035156)
\curveto(10.28,-0.09535156)(10.325,-0.100351565)(10.35,-0.100351565)
\curveto(10.375,-0.100351565)(10.415,-0.10535156)(10.46,-0.12035156)
}
\usefont{T1}{ptm}{m}{n}
\rput(10.291455,0.22464843){$K$}
\pscustom[linewidth=0.04,linecolor=color398]
{
\newpath
\moveto(4.96,-0.13178009)
\lineto(5.0668626,-0.08892304)
\curveto(5.120294,-0.067494504)(5.269902,-0.024637451)(5.3660784,-0.0032086181)
\curveto(5.4622545,0.01821991)(5.622549,0.039648436)(5.6866665,0.039648436)
\curveto(5.750784,0.039648436)(5.8683333,0.039648436)(5.9217644,0.039648436)
\curveto(5.975196,0.039648436)(6.103431,0.01821991)(6.178235,-0.0032086181)
\curveto(6.253039,-0.024637451)(6.42402,-0.07820862)(6.520196,-0.11035156)
\curveto(6.616372,-0.1424945)(6.7552943,-0.19606568)(6.7980394,-0.2174945)
\curveto(6.840785,-0.23892303)(6.93696,-0.26035157)(6.9903917,-0.26035157)
\curveto(7.0438232,-0.26035157)(7.1079407,-0.26035157)(7.14,-0.26035157)
}
\pscustom[linewidth=0.04,linecolor=color398]
{
\newpath
\moveto(4.4,0.19964844)
\lineto(4.44,0.20964843)
\curveto(4.46,0.21464844)(4.495,0.22964844)(4.51,0.23964843)
\curveto(4.525,0.24964844)(4.56,0.26464844)(4.58,0.26964843)
\curveto(4.6,0.27464843)(4.65,0.27964842)(4.68,0.27964842)
\curveto(4.71,0.27964842)(4.765,0.27464843)(4.79,0.26964843)
\curveto(4.815,0.26464844)(4.86,0.25464845)(4.88,0.24964844)
\curveto(4.9,0.24464844)(4.945,0.23964843)(4.97,0.23964843)
\curveto(4.995,0.23964843)(5.05,0.24964844)(5.08,0.25964844)
\curveto(5.11,0.26964843)(5.155,0.28464845)(5.2,0.29964843)
}
\psline[linewidth=0.03cm,linecolor=color567,linestyle=dashed,dash=0.16cm 0.16cm](2.74,-2.5603516)(7.94,-2.5603516)
\psline[linewidth=0.03cm,linecolor=color567,linestyle=dashed,dash=0.16cm 0.16cm](8.54,-3.1603515)(2.14,-3.1603515)
\psline[linewidth=0.03cm,linecolor=color567,linestyle=dashed,dash=0.16cm 0.16cm](2.14,3.223683)(8.524996,3.223683)
\psline[linewidth=0.03cm,linecolor=color567,linestyle=dashed,dash=0.16cm 0.16cm](2.7204542,2.6433163)(7.9445415,2.6433163)
\psline[linewidth=0.04cm](8.34,2.8396485)(8.34,-2.9603515)
\psline[linewidth=0.04cm](2.54,2.8396485)(8.34,2.8396485)
\psline[linewidth=0.04cm](2.54,2.8396485)(2.54,-2.9603515)
\psline[linewidth=0.04cm](2.54,-2.9603515)(8.34,-2.9603515)
\psdots[dotsize=0.16](2.54,2.8396485)
\psdots[dotsize=0.16](8.34,2.8396485)
\psdots[dotsize=0.16](8.34,-2.9603515)
\psdots[dotsize=0.16](2.54,-2.9603515)
\usefont{T1}{ptm}{m}{n}
\rput(5.371455,1.4046484){$U$}
\usefont{T1}{ptm}{m}{n}
\rput(8.8014555,3.4446485){$a_0^+$}
\usefont{T1}{ptm}{m}{n}
\rput(1.9414551,3.4446485){$a_1^+$}
\usefont{T1}{ptm}{m}{n}
\rput(8.861455,-3.4153516){$a_0^-$}
\usefont{T1}{ptm}{m}{n}
\rput(1.8414551,-3.3753517){$a_1^-$}
\psline[linewidth=0.1cm](2.56,0.37964845)(2.54,-0.76035154)
\psline[linewidth=0.1cm](8.34,0.69964844)(8.34,-0.54035157)
\psline[linewidth=0.05cm,linecolor=color567,arrowsize=0.05291667cm 2.0,arrowlength=1.4,arrowinset=0.4]{<-}(2.58,-0.82035154)(5.26,-2.2403517)
\psline[linewidth=0.05cm,linecolor=color567,arrowsize=0.05291667cm 2.0,arrowlength=1.4,arrowinset=0.4]{<-}(8.34,-0.56035155)(5.86,-2.2403517)
\usefont{T1}{ptm}{m}{n}
\rput(5.511455,-2.2753515){$\Sigma$}
\end{pspicture} 
}

\end{center}
\caption{{The rectangular domain $U$ and the wall set $\Sigma$.}}	
\label{fig4}
\end{figure}
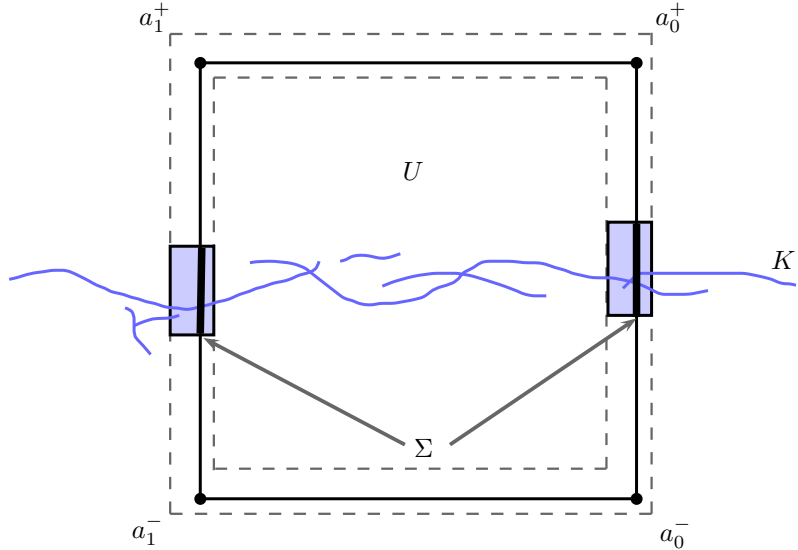

{\bf Step 3: Construction of the competitor \boldmath{$v$}.} Since $U$ is a rectangle with  ``uniform shape'', there exists a bilipschitz mapping $\Phi:\R^2 \to \R^2$ such that $\Phi(U)=B:=B(0,1)$, $\Phi(\partial U)=\partial B$ and $\Phi(\partial U^+)=  \mathscr C_\delta$ for some $\delta<1/2$, where $\partial U^+:=\partial U\cap A^+$ and $\mathscr C_\delta$ is as in the statement of Lemma~\ref {extensionL}. Note that the Lipschitz constants of $\Phi$ and $\Phi^{-1}$ are bounded by $Cr^{-1}$ and $Cr$ where $C$ is universal. Let $R^+$ be the skew symmetric matrix appearing in \eqref{R+}. Since $u \in H^1(\partial U^+;\R^2)$, we infer that the function $x \mapsto u \circ \Phi^{-1}(x)-R^+\Phi^{-1}(x)$ belongs to $H^1(\mathscr C_\delta;\R^2)$. Applying Lemma~\ref {extensionL}, we obtain a function $h^+ \in H^1(B;\R^2)$ such that $h=u \circ \Phi^{-1}-R^+\Phi^{-1}$ on $\mathscr C_\delta$ and 
 $$\int_{B}|\nabla h^+|^2 \, dx\leq C\int_{\mathscr C_\delta} \left|\partial_\tau (u \circ \Phi^{-1}-R^+\Phi^{-1}) \right|^2 \,d\mathcal H^1,$$
where $C>0$ is a universal constant. Then, defining $v^+:=h^+\circ \Phi \in H^1(U;\R^2)$ and noticing that if $\tau$ is a tangent vector to $\partial B$, then $\nabla\Phi^{-1}\tau$ is a tangent vector to $\partial U^+$ $\mathcal H^1$-a.e. in $\partial U^+$, we infer that $v^+(x)=u(x)-R^+x$ for $\mathcal H^1$-a.e. $x \in \partial U^+$ and
 $$\int_{U}|\nabla v^+|^2 \, dx\leq Cr \int_{\partial U \cap A^+} |\partial_\tau u -R^+\tau|^2 \,d\mathcal H^1.$$
Arguing similarly for $\partial U^-:=\partial U\cap A^-$ leads to a function $v^- \in H^1(U;\R^2)$ such that $v^-(x)=u(x)-R^-x$ for $\mathcal H^1$-a.e.  $x \in \partial U^-$ and
$$\int_{U}|\nabla v^-|^2 \, dx\leq Cr \int_{\partial U \cap A^-} |\partial_\tau u -R^-\tau|^2 \,d\mathcal H^1,$$
where $R^-$ is the skew-symmetric matrix appearing in  \eqref{R+}.
 
Let $K' \subset \Omega$ be as in the statement. We construct a function $v \in   H^1(\Omega \setminus K';\R^2)$ by setting
$$v(x):=v^\pm(x)+R^\pm x,$$
if $x$ belongs to the connected component of $U\setminus K'$ containing $D^\pm(x_0,r/5)$, and $v:=u$ otherwise.

Note that by construction
$v=u$ on $\partial U \setminus \Sigma$, and 
$$\int_{U \setminus K'} |e(v)|^2 \, dx \leq Cr \int_{\partial U \setminus \Sigma} |\nabla u-R^\pm|^2\, d\mathcal H^1.$$
Notice that $K \cap \partial U=K'\cap \partial U$ by assumption, because $K \setminus U=K' \setminus U$ and $U$ is open. Thus,  from \eqref{goodbound} it follows that
$$\int_{U \setminus K'} |e(v)|^2 \, dx\leq \frac{C}{\eta^6} \int_{B(x_0,r)\setminus K} |e(u)|^2 \, dx,$$
as required.
 \end{proof}


\subsection{Proof of Proposition~\ref {main3}}


In Lemma~\ref {extLem} we have constructed the key displacement competitor associated to a separating crack competitor, which will be employed to show the flatness estimate. The construction of the crack competitor will be similar to that of Proposition~\ref {prop1}, i.e. it will be obtained by replacing $K$ by a segment in some ball. The difference here will be in the error appearing in the density estimate, which will depend only on $\omega(x_0,r)$, and not anymore on $\beta(x_0,r)$.

\begin{proposition}\label{main1}  
There exist $\varepsilon'_0 >0$ and $C'>0$ such that the following holds. Let $(u,K) \in \mathcal A(\Omega)$ be a minimizer of the Griffith functional, and let $x_0 \in K$ and  $r>0$ be such that $\overline B(x_0,r)\subset \Omega$,
$$\omega_u(x_0,r)+\beta_K(x_0,r) \leq \varepsilon'_0,$$ 
and $K$ separates $D^{\pm}(x_0,r)$ in   $\overline B(x_0,r)$. Then there exists $s \in (r/40,r/5)$ such that $K \cap \partial B(x_0,s)=\{z,z'\}$, for some $z \neq z'$, and
\begin{equation}\label{eq:flat-est}
\mathcal{H}^1( K\cap B(x_0,s)) \leq   |z-z'|+C'r \omega_u(x_0,r)^{\frac{1}{7}}.
\end{equation}
\end{proposition}

\begin{proof} 
We define 
$$\varepsilon'_0:=\min\left(10^{-28},\frac{\varepsilon_0}{5}, \frac{10^{-2}\theta_0}{5C_*},\;\frac{1}{5C_*}\right),$$
where $\varepsilon_0>0$ is the universal constant of Lemma~~\ref{card2}, $\theta_0>0$ is the Ahlfors regularity constant, and $C_*>0$ is the universal constant of Proposition~\ref {prop_para}. We notice that   $\omega(x_0,r/5)+\beta(x_0,r/5)\leq \varepsilon_0$ and that $K$ still separates $D^\pm(x_0,r/5)$ in $\overline B(x_0,r/5)$, since they are contained in two different connected components of   $\overline B(x_0,r/5)\setminus \{y \in \R^2: \; |(y-x_0)_2|>\beta(x_0,r)r\}$.
Thus, according to Lemma~~\ref{card2} applied in $B(x_0,r/5)$, we can indeed find $s\in (r/40,r/5)$ such that $\#( K \cap \partial B(x_0,s))=2$, and we denote by $z$ and $z'$ both points of $ K \cap \partial B(x_0,s)$. 
 
Let now $\eta \in(0,10^{-4})$ be fixed. Let $U$ be the rectangle, satisfying $\overline B(x_0,r/5)\subset U \subset  B(x_0,r)$, and $\Sigma$ be the wall set, satisfying $K\cap \partial U\subset \Sigma \subset \partial U$ and  $\mathcal{H}^1(\Sigma)\leq 120\eta r$, given by Lemma~\ref {extLem} in $B(x_0,r)$ for $\eta \in (0,10^{-4})$ fixed above.

Consider the set
$$K':=[z,z']\cup ( K  \setminus B(x_0,s)).$$
By construction $K' \setminus U=K \setminus U$  and  $D^\pm(x_0,r/5)$ are contained in different connected components of $U \setminus K'$. Then, Lemma~\ref {extLem} provides a function $v \in  H^1(\Omega\setminus K';\R^2)$ which coincides with $u$ on $(\Omega\setminus U)\setminus \Sigma$ and satisfies \eqref{ext_rect}. The pair $(v,K'')$, with
$$K'':=K'\cup \Sigma,$$
is thus a competitor for $(u,K)$, and by minimality of $(u,K)$ we have that
$$\int_{B(x_0,r)\setminus K} \mathbf Ae(u):e(u) \, dx + \mathcal{H}^1(K\cap \overline B(x_0,r))
\leq \int_{B(x_0,r) \setminus K''} \mathbf A e(v):e(v) \, dx  + \mathcal{H}^1(K''\cap \overline B(x_0,r)) .$$
Since $u=v$ outside of $U$ and $K'' \cap \partial B(x_0,r)=K \cap \partial B(x_0,r)$,  
we deduce by \eqref{ext_rect} that 
\begin{eqnarray*}
\mathcal{H}^1(K\cap B(x_0,r))&\leq&   \mathcal{H}^1(K''\cap B(x_0,r)) +\frac{C }{\eta^6}r \omega(x_0,r)\\
&\leq & \mathcal{H}^1(  K \cap B(x_0,r) \setminus B(x_0,s))+|z-z'|+ \mathcal{H}^1(\Sigma)+ \frac{C }{\eta^6}r\omega(x_0,r).
\end{eqnarray*}
Since  $\mathcal H^1(\Sigma)\leq 120 \eta r$  we get that
$$\mathcal{H}^1(K\cap B(x_0,s))\leq |z-z'|+ 120\eta r+ \frac{C }{\eta^6}r\omega(x_0,r).$$
Finally, we use that $\eta >0$ was assumed to be arbitrary in the interval $(0,10^{-4})$. Since $\omega(x_0,r)\leq 10^{-28}$ by assumption, we can choose $\eta=\omega(x_0,r)^{\frac{1}{7}}\leq 10^{-4}$, so that
$$\mathcal{H}^1(K\cap B(x_0,s))\leq |z-z'|+ Cr\omega(x_0,r)^{\frac{1}{7}},$$
as required. 
\end{proof}

We are now ready to prove the main flatness estimate.

\begin{proof}[Proof of Proposition~\ref {main3}]
Let us define
\begin{equation}\label{eq:eps}
\varepsilon_1 = \min\left\{\varepsilon'_0,\; \left(\frac{\min(1,\theta_0)}{400 C'}\right)^{7}\right\},
\end{equation}
where $\varepsilon'_0>0$ is the threshold of Proposition~\ref {main1} and $C'>0$ is the universal constant in \eqref{eq:flat-est}. By Proposition~\ref {main1}, we know that there exists $s \in (r/40,r/5)$ such that $K \cap \partial B(x_0,s)=\{z,z'\}$, for some $z \neq z'$, and
\begin{equation}\label{eq:flat-estE}
\mathcal{H}^1( K\cap B(x_0,s)) \leq   |z-z'|+C' r\omega(x_0,r)^{\frac{1}{7}}.  
\end{equation}
Notice that
$$\max\{|(z-x_0)_2|,|(z'-x_0)_2|\}\leq 40\varepsilon'_0 r,$$
since $\beta(x_0,s)\leq 40\varepsilon'_0$, and that
\begin{equation}
\mathcal{H}^1( K\cap B(x_0,s)) \leq  2r+C' r \omega(x_0,r)^{\frac{1}{7}} \leq 3r, \label{estimLe}
\end{equation}
because $C'\omega(x_0,r)^{\frac{1}{7}}\leq 1$ by \eqref{eq:eps}.
Let $L$ be the line passing through $x_0$ which is parallel to the segment $[z,z']$.

\medskip

{\bf Step 1.} We first prove that
\begin{equation}\label{eq:beta1}
\sup_{y \in K \cap \overline B(x_0,r/50)} {\rm dist}(y,L) \leq  C'' r \omega(x_0,r)^{\frac{1}{14}},
\end{equation}
where $C''>0$  only depends on $\theta_0>0$.

Since $K \cap B(x_0,s)$ separates $D^\pm(x_0,s)$ in $\overline B(x_0,s)$, by Lemma~\ref {existGamma} there exists an injective Lipschitz curve $\Gamma \subset K\cap \overline B(x_0,s)$ joining $z$ and $z'$. 
Being $\mathcal H^1(\Gamma) \geq |z-z'|$, according to estimate \eqref{eq:flat-estE} we have
\begin{equation}\label{eq:leq}
\mathcal H^1(K \cap B(x_0,s) \setminus \Gamma)  \leq  \mathcal H^1( K \cap B(x_0,s))-\mathcal H^1(\Gamma) \leq  C'r \omega(x_0,r)^{\frac{1}{7}}.
\end{equation}

We claim that for all $y \in K \cap \overline B(x_0,r/50)$,
\begin{equation}\label{eq:distGamma}
{\rm dist}(y,\Gamma)\leq \frac{2C'}{\theta_0}r \omega(x_0,r)^{\frac{1}{7}}.
\end{equation}
Indeed, assume by contradiction that there exists $y_0 \in K \cap \overline B(x_0,r/50)$ such that ${\rm dist}(y_0,\Gamma) > \delta r$, with $\delta=\frac{2C'}{\theta_0}r\omega(x_0,r)^{\frac{1}{7}}$.
According to condition \eqref{eq:eps} we have that $\delta < \frac{1}{200}$, so that 
$$B(y_0,\delta r) \subset B(x_0,r/40) \setminus \Gamma\subset B(x_0,s) \setminus \Gamma.$$ 
By Ahlfors regularity of $K$, 
$$\mathcal H^1(K \cap B(x_0,s) \setminus \Gamma) \geq \mathcal H^1(K \cap B(y_0,\delta r))\geq \theta_0 \delta r = 2 C'\omega(x_0,r)^{\frac{1}{7}}r,$$
which is in contradiction with \eqref{eq:leq} and establishes the validity of the claim \eqref{eq:distGamma}.

Now, an application of Lemma~\ref {pythag} ensures that for all $w \in \Gamma$
\begin{eqnarray}
{\rm dist}(w,[z,z'])^2 &\leq& \mathcal{H}^1(\Gamma) \Big(\mathcal{H}^1(\Gamma)-|z'-z|\Big) \notag \\
&\leq & \mathcal{H}^1(K\cap B(x_0,s)) \Big(\mathcal{H}^1(K\cap B(x_0,s))-|z'-z|\Big) \notag \\
&\leq &  3C' r^2\omega(x_0,r)^{\frac{1}{7}}, \label{dist11}
\end{eqnarray}
by \eqref{estimLe} and \eqref{eq:flat-estE}.

According to \eqref{eq:distGamma}, \eqref{dist11}, and the triangle inequality, we infer that for all $y \in K \cap \overline  B(x_0,r/50)$
\begin{equation}\label{yzz}
{\rm dist}(y,[z,z']) \leq \sqrt{3C'} r\omega(x_0,r)^{\frac{1}{14}}+\frac{2C'}{\theta_0}r \omega(x_0,r)^{\frac{1}{7}}\leq \tilde C'r\omega(x_0,r)^{\frac{1}{14}},
\end{equation}
with $\tilde C'>0$ depending on $\theta_0$, where we used that $\omega(x_0,r)\leq 1$ to estimate $\omega(x_0,r)^{\frac{1}{7}}\leq \omega(x_0,r)^{\frac{1}{14}}$. Finally, if $L$ denotes the line passing through $x_0$ which is parallel to the segment $[z,z']$, we deduce that \eqref{eq:beta1} holds by the triangle inequality and \eqref{yzz} applied to $x_0 \in L\cap K \cap \overline B(x_0,r/50)$.

\medskip

{\bf Step 2. } We now prove that 
\begin{equation}\label{eq:beta2}
\sup_{x \in L \cap \overline B(x_0,r/50)} {\rm dist}(x,K) \leq  C''' r \omega(x_0,r)^{\frac{1}{14}},
\end{equation}
where $C'''>0$ possibly depends on $\theta_0$.
For this purpose, we recall that $K$ separates $D^\pm(x_0,r)$ in  $\overline B(x_0,r)$, thus in particular, for every $x\in L\cap  \overline B(x_0,r/50)$, the orthogonal line to $L$ passing through $x$ meets a point $y\in K$. If $y \in \overline B(x_0,{r}/{50})$, then by Step 1 we know that $|x-y|\leq C''  \omega(x_0,r)^{\frac{1}{14}}r$, and then 
$${\rm dist}(x,K) \leq  C'' r\omega(x_0,r)^{\frac{1}{14}}.$$
Now, if $y\not\in \overline B(x_0,{r}/{50})$, this is only possible for  $x$ very close to $\partial B(x_0,r/50)$, because $K\cap B(x_0,r)$ is contained in  a strip around $L$ of height $C'' r\omega(x_0,r)^{\frac{1}{14}}$, which is small. More precisely, one sees using Pythagoras Theorem that the second case occurs only for points $x\in L$ satisfying
$${\rm dist}(x, \partial B(x_0,r/50))\leq   \frac{r}{50}-\left(\left(\frac{r}{50}\right)^2-\left(C''r\omega(x_0,r)^{\frac{1}{14}}\right)^2\right)^{\frac{1}{2}}\leq rMC''\omega(x_0,r)^{\frac{1}{14}},$$
where $M>0$ is a universal constant obtained from the elementary inequality 
$$\frac{1}{50}-\left(\left(\frac{1}{50}\right)^2-t^2\right)^{\frac{1}{2}}\leq Mt \quad \text{ for all }0<t<10^{-3},$$
which results from the mean value theorem. By the triangle inequality we then obtain
$${\rm dist}(x,K) \leq  (M+1)C'' r \omega(x_0,r)^{\frac{1}{14}}.$$
Gathering \eqref{eq:beta1} and \eqref{eq:beta2}, and using \eqref{optimal}, we deduce that $\beta(x_0,r/50) \leq C_1 r \omega(x_0,r)^{\frac{1}{14}}$ for some constant $C_1>0$ depending on $\theta_0$, which concludes the proof of the proposition.
\end{proof}


\section{Proof of the normalized energy decay}\label{sec6}


In this section we prove a decay estimate for the normalized energy of a Griffith minimizer. The strategy is based  upon a  compactness argument and a $\Gamma$-convergence type analysis where one shows the stability of the Neumann problem in planar elasticity along a sequence of sets $K_n$ which converge in the Hausdorff sense to a diameter within a ball. It gives an alternative approach even for the scalar case (albeit only 2-dimensional and under topological conditions) to the corresponding decay estimates of the normalized energy in the standard proofs of regularity for the Mumford-Shah minimizers (\cite{afp}, \cite[Theorem 1.10]{l3}). 

We will start establishing some auxiliary results on the Airy function.

\subsection{The Airy function} 

We  state here a general result concerning the  existence of the Airy function associated to a minimizer of the Griffith energy. It follows a construction similar to that in \cite[Proposition~4.3]{BCL}, itself inspired by that introduced in \cite{cdens}. The Airy function will be useful in order to get compactness results along a sequence of minimizers. The main difference with the situation in \cite{BCL} is that now $K$ is not assumed to be connected. {The proofs are very similar to those of \cite{BCL}, and for that reason we do not write all the arguments but only point out the main changes with respect to the original proof.}

\medskip

First we recall the following result coming from De Rham's Theorem and proved in \cite[Lemma 4.1]{BCL} in the case where $\Omega$ is a ball. The extension to a general bounded open set with Lipschitz boundary is straightforward, since the only property used in that proof is the existence of traces of Sobolev functions on the boundary.

\begin{lemma}\label{XY} 
Let $\Omega \subset \R^2$ be a bounded open set with Lipschitz boundary and let $L \subset \overline{ \Omega}$ be a  closed set. Let us consider the following subspaces of $L^2(\Omega;\R^2)$
\begin{eqnarray}
X_{L}&:=&\{\sigma \in \mathcal{C}^\infty(\overline{\Omega};\R^2) \,:\, {\rm supp}(\sigma)\cap L =\emptyset ,\; {\rm div}\sigma=0 \text{ in }  \Omega\}, \notag \\
Y_{L}&:=&\{\nabla v \; : \; v\in H^1(\Omega\setminus  L) \;,\; v=0 \text{ on }  \partial \Omega\setminus L \}.\notag
\end{eqnarray} 
Then $\overline{X_{L}}=Y_{ L}^{\bot}$ in $L^2(\Omega;\R^2)$.
\end{lemma}

From the previous Lemma, one can construct the ``harmonic conjugate'' $v$ associated to a minimizer $(u,K)$ of the Griffith functional. The proof follows the lines of that in \cite[Proposition 4.2]{BCL}. The main difference with \cite{BCL} is that, here, the singular set $K$ is not assumed to be connected. This implies that it is not in general possible to conclude that  $v$ vanishes on the full crack $K$. However, the following proof makes it possible to ensure that, in some suitable weak sense, $v$ is constant in each connected component of $K$, but the constants might depend on the associated connected component. This is the reason why we renormalize the harmonic conjugate $v$ to vanish only on an arbitrary connected component of the crack of positive length.

\medskip

\begin{proposition}[Harmonic conjugate]\label{prop:harm-conj}
Let $\Omega \subset \R^2$ be a bounded and simply connected open set with Lipschitz boundary, and let $(u,K) \in \mathcal A(\Omega)$ be a minimizer of the Griffith functional. Then, for every connected component $L$ of  $K$ with $\mathcal H^1(L)>0$, there exists a function $v \in H^1_{0,L}(\Omega;\R^2) \cap  \mathcal{C}^\infty(\Omega\setminus K;\R^2)$ such that 
\begin{equation}\label{eq:harm-conj}
\sigma:=\mathbf A e(u)=\left(\!\!\!\begin{array}{cc}
-\partial_2 v_1  &  \partial_1 v_1 \\
-\partial_2 v_2 & \partial_1 v_2
\end{array}\!\!\!\right) 
\quad \text{a.e.  in } \Omega.
\end{equation}
\end{proposition}

\begin{proof} 
Let $L$ be a connected component of $K$ with $\mathcal H^1(L)>0$. According to the variational formulation \eqref{eq:var-form} and the fact that $\sigma(x) \in \mathbb M^{2 \times 2}_{\rm sym}$ for a.e. $x \in \Omega \setminus K$, for any $v \in H^1(\Omega \setminus K;\R^2)$ with $v=0$ on $\partial \Omega \setminus K$, we have
$$\int_\Omega \sigma : \nabla v\, dx=0.$$
This is {\it a fortiori} true for any $v \in H^1(\Omega\setminus L;\R^2)$ with $v=0$ on $\partial \Omega \setminus L$. Consequently, both lines of $\sigma$, denoted by 
$$\sigma^{(1)}:=\left(\!\!\!\begin{array}{c} \sigma_{11}\\ \sigma_{12} \end{array}\!\!\!\right), \quad \sigma^{(2)}:=\left(\!\!\!\begin{array}{c} \sigma_{12}\\ \sigma_{22} \end{array}\!\!\!\right),$$ 
belong to $Y_{L}^\perp$. Lemma~\ref {XY} ensures the existence of a sequence $(\sigma_n^{(1)}) \subset X_{L}$ such that $\sigma_n^{(1)} \to \sigma^{(1)}$ in $L^2(\Omega;\R^2)$. Since ${\rm div}\, \sigma_n^{(1)}=0$ in $\Omega$, $\Omega$ is simply connected and ${\rm supp}(\sigma_n^{(1)}) \cap {L}=\emptyset$, it follows  that 
$$\sigma_n^{(1)}=\nabla^\perp p_n^{(2)}:=\begin{pmatrix}
-\partial_2 p_n^{(2)}\\
\partial_1 p_n^{(2)}
\end{pmatrix},$$
for some $p_n^{(2)} \in \C^\infty(\ol \Omega)$ with ${\rm supp}(\nabla p_n^{(2)}) \cap {L}=\emptyset$. Since $L$ is connected, we can assume, up to add a constant, that $p_n^{(2)}=0$ on ${L}$. Consequently, since $\mathcal H^1(L)>0$, Poincar\'e's inequality implies that $p_n^{(2)} \to p^{(2)}$ in $H^1(\Omega)$ for some $p^{(2)} \in H^1_{0,L}(\Omega)$ satisfying $\sigma^{(1)}=\nabla^\perp p^{(2)}$. We prove similarly the existence of $p^{(1)} \in H^1_{0,{L}}(\Omega)$ satisfying $\sigma^{(2)}=-\nabla^\perp p^{(1)}$. We then define 
$$v:=\left(\!\!\!\begin{array}{c} p^{(2)} \\-p^{(1)} \end{array}\!\!\!\right)\in H^1_{0,L}(\Omega;\R^2)$$
which satisfies \eqref{eq:harm-conj}. Finally, since $\sigma \in \C^\infty(\Omega \setminus K;\Ms)$, then $v\in \C^\infty(\Omega \setminus K;\R^2)$.
\end{proof}

We next construct the Airy function $w$ associated to the displacement $u$ in $\Omega$ following an approach similar to \cite{BCL,cdens}, but once more with the difference that, here, $K$ is no more assumed to be connected. 

\begin{proposition}[Airy function]\label{prop:airy}
Let $\Omega \subset \R^2$ be a bounded and simply connected open set with Lipschitz boundary, and let $(u,K) \in \mathcal A(\Omega)$ be a minimizer of the Griffith functional. If  $L$ is a  connected component of $K$ such that $\mathcal H^1(L)>0$, then there exists a function $w \in  H^2(\Omega) \cap H^1_{0,L}(\Omega)$ such that 
\begin{equation}\label{eq:biharm}
\Delta^2w =0 \text{ in }\D'(\Omega \setminus K)
\end{equation}
and
\begin{equation}\label{eq:hessian}
\sigma=
\begin{pmatrix}
\partial_{22} w & -\partial_{12} w\\
-\partial_{12} w & \partial_{11} w
\end{pmatrix}
\quad \text{ a.e. in }\Omega.
\end{equation}
In addition, if $A \subset \R^2$ is an open set with $\overline A \subset \Omega$, then $w \in H^2_{0,{L}}(A)$.
\end{proposition}

\begin{proof}
Proposition~\ref {prop:harm-conj} ensures the existence of $p^{(1)}$ and $p^{(2)} \in  H^1_{0,L}(\Omega)$ such that
$$\sigma^{(1)}=\nabla^\perp p^{(2)} , \quad \sigma^{(2)}=-\nabla^\perp p^{(1)} \qquad \text{ a.e. in }\Omega.$$
Since $p^{(1)}=p^{(2)}=0$ on $L$, arguing as in \cite[Proposition 4.3]{BCL}, it follows that 
$$\left(\!\!\!\begin{array}{c}-p^{(2)}\\p^{(1)} \end{array}\!\!\!\right) \in Y_{{L}}^\perp=\ol X_{L},$$
owing again to Lemma~\ref {XY}. Next, arguing as in the proof of Proposition~\ref {prop:harm-conj}, we deduce the existence of a function $w \in H^1_{0,{L}}(\Omega)$ such that 
$$\nabla w=\left(\!\!\!\begin{array}{c}p^{(1)}\\p^{(2)} \end{array}\!\!\!\right) \in L^2(\Omega;\R^2).$$
By construction, the Airy function $w$ satisfies \eqref{eq:hessian} and, arguing as in \cite[Proposition 4.3]{BCL}, it also satisfies \eqref{eq:biharm}. 
	
It remains to show that if $A \subset \R^2$ is an open set with $\overline A \subset \Omega$, then $w \in H^2_{0,L}(A)$.
 We first note that $w \in H^1_{0,L}(\Omega) \cap H^2(\Omega)$ with $\nabla w \in H^1_{0,L}(\Omega;\R^2)$. In particular, since $w \in H^2(\Omega)$, it has a (H\"older) continuous representative, still denoted $w$, so that it makes sense to consider its pointwise values. 
 
Since $A\setminus L$ is not smooth, in order to show that $w \in H^2_{0,L}(A)$, we will use a capacity argument similar to that used in \cite[Proposition 4.3]{BCL} and in \cite[Theorem 1]{cdens}.
 
Let us consider a cut-off function $\eta \in \C^\infty_c(\Omega;[0,1])$ satisfying $\eta=0$ on $\partial \Omega$ and $\eta=1$ on $A$. Denoting by $z:=\eta w$, then $z \in H^2(\Omega) \cap H^1_0(\Omega\setminus L)$ and $\nabla z \in H^1_0(\Omega \setminus L;\R^2)$. 
 
By \cite[Theorem 9.1.3]{AH}, $z\in H^2_0(\Omega\setminus L)$ if a $\Ccap_{2,2}$-quasicontinuous representative of $z$ vanishes on $\partial(\Omega\setminus L)$ $\Ccap_{2,2}$ q.e., and a $\Ccap_{1,2}$-quasicontinuous representative of $\nabla z$ vanishes on $\partial(\Omega\setminus L)$ $\Ccap_{1,2}$ q.e..

Since $\nabla z\in H^1_0(\Omega\setminus L)$, the second property is a consequence of \cite[Theorem 3.3.42]{HP}. As for the first property, since $z$ is continuous, it coincides with its $\Ccap_{2,2}$-quasicontinuous representative. Moreover, since the empty set is the only set of zero $\Ccap_{2,2}$-capacity, we are reduced to show that $z = 0$ everywhere on $\partial (\Omega \setminus L)$. 
Let $F:=\{x \in \partial (\Omega \setminus L) : z(x)=0\}$. Then, $F$ is a compact set satisfying $\Ccap_{1,2}(\partial (\Omega \setminus {L} )\setminus F)=0$, being $z\in H^1_0(\Omega\setminus L)$. Let $G$ be a connected component of $\partial (\Omega \setminus L)\setminus F$. Since a compact and connected set of positive diameter has a positive $\Ccap_{1,2}$-capacity (see \cite[Corollary 3.3.25]{HP}), we deduce that $\diam(G)=0$, so that $G$ is (at most) a singleton. Moreover, $F$ being compact, its complementary $ \partial (\Omega \setminus L)\setminus F$ is open in the relative topology of $ \partial (\Omega \setminus L)$, and thus $G$ is (at most) an isolated point. Finally, since $\partial(\Omega \setminus L)=\partial\Omega\cup L$, being $L\subset \overline{\Omega}$ closed and $\mathcal{H}^1(L)<+\infty$, and since neither $\partial\Omega$ or $L$ have isolated points, we have $G=\emptyset$, and thus $z=0$ on $\partial (\Omega \setminus L)$.
	
As a consequence of \cite[Theorem 9.1.3]{AH}, we conclude that $z \in H^2_0(\Omega \setminus L)$, or in other words, that there exists a sequence $(z_n) \subset \C^\infty_c(\Omega \setminus L)$ such that $z_n \to z=\eta w$ in $H^2(\Omega \setminus L)$. Note in particular that  $z_n \in \C^\infty(\ol{A})$ and that $z_n$ vanishes in a neighborhood of $L \cap \ol{A}$. Therefore, since $z =w$ and $\nabla z=\nabla w$ in $A$, we deduce that $w \in H^2_{0,L}(A)$.
\end{proof}

\begin{remark} \label{lastremark}
{\rm If $\Gamma \subset K\cap \Omega$ is a connected component of $K\cap \Omega$, then there exists a connected component $L$ of $K$ such that $\Gamma\subset L$. If we consider the Airy function given by Proposition~\ref{prop:airy} associated with this  component $L$, then  for all $A \subset \R^2$ open  with $\overline A \subset \Omega$ we have that  $w \in H^2_{0,\Gamma}(A)$ because $H^2_{0,L}(A)\subset H^2_{0,\Gamma}(A)$.}
\end{remark}

\subsection{Proof of Proposition~\ref {JFprop}}

We argue by contradiction by assuming that the statement of the proposition is false. Then, there exists $\tau_0>0$ such that for every $n \in \mathbb N$, one can find a minimizer  $(\hat u_n, \hat K_n) \in \mathcal A (\Omega)$ of the Griffith functional (with the same Dirichlet boundary data $\psi$), an isolated connected component $\hat \Gamma_n$ of ${\hat K_n\cap \Omega}$, points $x_n \in \hat \Gamma_n$, radii $r_n>0$ with $\overline B(x_n,r_n) \subset \Omega$  such that 
$$\hat K_n \cap \overline B(x_n,r_n)=\hat \Gamma_n \cap \overline B(x_n,r_n), \quad\beta_{\hat K_n}(x_n,r_n)\to 0,$$ 
and 
$$\omega_{\hat u_n}\left(x_n,a r_n \right) > \tau_0 \, \omega_{\hat u_n}(x_n,r_n),$$
for some $a \in (0,1)$ (to be fixed later).
	
\medskip

\noindent {\bf Rescaling and compactness.} In order to prove compactness properties on the sequences of sets and displacements, we need to rescale them into a unit configuration. For simplicity, from now on, we denote by $B:=B(0,1)$. Let us first rescale the sets $\hat K_n$ and $\hat \Gamma_n$ by setting, for all $n \in \mathbb N$,
$$K_n:=\frac{\hat K_n-x_n}{r_n}, \quad \Gamma_n:=\frac{\hat \Gamma_n-x_n}{r_n}.$$
Let $\hat{L}_n:=L(x_n,r_n)$ be an affine line such that
$$d_{\mathcal H}(\hat{L}_n \cap \overline B(x_n,r_n),\hat K_n\cap \overline B(x_n,r_n))\leq r_n \beta_{\hat K_n}(x_n,r_n),$$
and $L_n:=\frac{\hat L_n-x_n}{r_n}$ its rescaling. {Up to a subsequence, and up to a change of orthonormal basis, we can assume that $L_n \cap \overline B \to T \cap \overline B$ in the sense of Hausdorff, where $T:=\R e_1$.} Then, since
$$d_{\mathcal H}( L_n\cap \overline B, K_n\cap \overline B)=\frac{1}{r_n}d_{\mathcal H}(\hat L_n\cap \overline B(x_n,r_n),\hat K_n\cap \overline B(x_n,r_n))\leq \beta_{\hat K_n}(x_n,r_n) \to 0,$$
we deduce that $\Gamma_n \cap \overline B=K_n \cap \overline B \to  T \cap \overline B$ in the sense of Hausdorff. 	
	
We next rescale the displacements $\hat u_n$ by setting, for all $n \in \mathbb N$ and a.e. $y \in B$,
$$u_n(y):=\frac{\hat u_n(x_n+r_n y)}{\sqrt{\omega_{\hat u_n}(x_n,r_n) r_n}}.$$
Then, we have
\begin{equation}\label{eq:energy1}
\int_{B\setminus K_n} \mathbf A e(u_n):e(u_n)\, dx=1,
\end{equation}
\begin{equation}\label{eq:omegaun}
\omega_{u_n}\left(0,a\right) > \tau_0.
\end{equation}
Note that $u_n \in LD(B \setminus  K_n)$ is a solution of 
$$\inf\left\{\int_{B \setminus K_n} \mathbf A e(z):e(z)\, dx : \; z \in LD(B \setminus K_n),\, z=u_n \text{ on }\partial B \setminus K_n \right\},$$
and in particular,
\begin{equation}\label{eq:minvn}
\int_{B \setminus K_n}\mathbf A e(u_n):e( u_n)\, dx \leq \int_{B \setminus  K_n}\mathbf A  e(u_n+\varphi) :e(u_n+\varphi)\, dx
\end{equation}
for all $\varphi \in LD(B \setminus K_n)$ with $\varphi=0$ on $\partial B \setminus K_n$.
	
According to the energy bound \eqref{eq:energy1}, up to a subsequence, we have 
$$e( u_{n}) {\bf 1}_{B \setminus   K_{n}}\rightharpoonup e \quad \text{ weakly in }L^2(B;\mathbb M^{2\times 2}_{\rm sym})$$
for some  $e \in L^2(B;\mathbb M^{2\times 2}_{\rm sym})$.
We next show that $e$ is  the symmetrized gradient of some displacement. To this aim, we consider, for any $0<\delta<1/10$, the Lipschitz domain
$$A_\delta:=\{x\in B : {\rm dist}(x,T)>\delta\} = A_\delta^+ \cup A_\delta^-,$$
where $A_\delta^\pm$ are both connected components of $A_\delta$. Note that for such $\delta$, $D^\pm := B\Big((0,\pm \frac{3}{4}) ,\frac{1}{4})\Big) \subset A^\pm_\delta$ and $K_n \cap U_\delta = \emptyset$ for $n$ large enough (depending on $\delta$). Denoting by 
$$r_n^\pm(x):=\frac{1}{|D^\pm|}\int_{D^\pm} u_n(y) \,dy +\left(\frac{1}{|D^\pm|}\int_{D^\pm} \frac{\nabla u_n(y)-\nabla u_n(y)^{T}}{2}dy\right)\left(x-\frac{1}{|D^\pm|}\int_{D^\pm} y \,dy\right),$$
the rigid body motion associated to $u_n$ in $D^\pm$, by virtue of the Poincar\'e-Korn inequality \cite[Theorem 5.2 and Example 5.3]{AMR}, we get that 
$$\|u_n-r^\pm_n\|_{H^1(A^\pm_\delta;\R^2)}\leq c_\delta \| e(u_n)\|_{L^2(A^\pm_\delta;\mathbb M^{2\times 2}_{\rm sym})},$$
for some constant $c_\delta>0$ depending on $\delta$. Thanks to a diagonalisation argument, for a further subsequence (not relabeled), we obtain a function $v \in LD(B \setminus T)$ such that $u_n-r^\pm_n \rightharpoonup v$ weakly in $H^1(A^\pm_\delta;\R^2)$, for any $0<\delta<1/10$. Necessarily we must have that $e=e(v)$ and thus,
$$e(u_n){\bf 1}_{B\setminus K_n} \rightharpoonup e(v) \quad \text{ weakly in }L^2(B;\mathbb M^{2 \times 2}_{\rm sym}).$$
	
\medskip
	
\noindent {\bf Minimality.} We next show that $v$ satisfies the minimality property
$$\int_{B \setminus T} \mathbf A e(v):e(v)\, dx \leq \int_{B \setminus T} \mathbf A e(v+\varphi):e(v+\varphi)\, dx$$
for all $\varphi \in LD(B \setminus T)$ such that $\varphi=0$ on $\partial B \setminus T$. According to \cite[Theorem 1]{cdens}, it is enough to consider competitors $\varphi \in H^1(B \setminus T;\R^2)$ such that $\varphi=0$ on $\partial B \setminus T$.
	
For a given arbitrary competitor $\varphi \in H^1(B \setminus T;\R^2)$ such that $\varphi=0$ on $\partial B \setminus T$, we construct a sequence of competitors for the minimisation problems \eqref{eq:minvn} using a jump transfert type argument (see \cite{fl} and \cite{BCL}). To this aim,  we denote by $C_n^\pm$ the connected component of $B \setminus K_n$ which contains the point $(0,\pm 1/2)$, and we define $\varphi_n$ as follows
\begin{itemize}
\item $\varphi_n(x_1,x_2)=\varphi(x_1,x_2)$ if $(x_1,x_2) \in [C_n^+ \cap \{x_2 \geq 0\}] \cup [C_n^- \cap \{x_2 \leq 0\}]$;
\item $\varphi_n(x_1,x_2)=\varphi(x_1,-x_2)$ if $(x_1,x_2) \in [C_n^+ \cap \{x_2 < 0\}] \cup [C_n^- \cap \{x_2 > 0\}]$;
\item $\varphi_n(x_1,x_2)=0$ otherwise.
\end{itemize}
Then, one can check that $\varphi_n \in H^1(B \setminus K_n;\R^2)$ and $\varphi_n=0$ on $\partial B \setminus K_n$. 
Moreover, $\varphi_n \to \varphi$ strongly in $L^2(B;\R^2)$ and $e(\varphi_n){\bf 1}_{B \setminus K_n}  \to e(\varphi)$ strongly in $L^2(B;\mathbb M^{2 \times 2}_{\rm sym})$. Therefore, thanks to the minimality property satisfied by $u_n$, we infer that
$$\int_{B \setminus K_n} \mathbf A e(u_n):e(u_n)\, dx \leq \int_{B \setminus K_n} \mathbf A e(u_n+\varphi_n):e(u_n+\varphi_n)\, dx,$$
which implies, by expanding the squares, that
$$0 \leq 2 \int_{B \setminus K_n} \mathbf A e(u_n):e(\varphi_n)\, dx +  \int_{B \setminus K_n} \mathbf A e(\varphi_n):e(\varphi_n)\, dx.$$
Using that $e(\varphi_n){\bf 1}_{B \setminus K_n} \to e(\varphi)$ strongly in $L^2(B;\mathbb M^{2 \times 2}_{\rm sym})$ and $e(u_n){\bf 1}_{B \setminus K_n} \rightharpoonup e(v)$ weakly in $L^2(B;\mathbb M^{2 \times 2}_{\rm sym})$, we can pass to the limit as $n \to +\infty$ to get that
$$0 \leq 2 \int_{B \setminus T} \mathbf A e(v):e(\varphi)\, dx + \int_{B \setminus T} \mathbf A e(\varphi):e(\varphi)\, dx,$$
or still
$$\int_{B \setminus T} \mathbf A e(v):e(v)\, dx \leq \int_{B \setminus T} \mathbf A e(v+\varphi):e(v+\varphi)\, dx.$$
As a consequence, $v$ is a smooth function in $B\setminus T$. Moreover, due to the Korn  inequality in both connected components of $B\setminus T$ (which are Lipschitz domains), we get that $v \in H^1(B \setminus T;\R^2)$.

\medskip
	
\noindent {\bf Convergence of the elastic energy.} In order to pass to the limit in inequality \eqref{eq:omegaun}, we need to show the convergence of the elastic energy, or in other words, the strong convergence of the sequence of elastic strains $(e(u_n))_{n \in \mathbb N}$. This will be achieved by using Proposition~\ref {prop:airy} which provides an Airy function $\hat w_n$ associated to the displacement $\hat u_n$, satisfying $\hat w_n \in H^2(\Omega) \cap H^1_{0,\hat\Gamma_n} (\Omega) \cap H^2_{0,\hat \Gamma_n}(A)$ for all open set $A \subset \R^2$ with $\overline A \subset \Omega$, (see also Remark \ref{lastremark}) such that
$$\Delta^2 \hat w_n= 0 \quad \text{ in }\Omega \setminus \hat K_n,$$
and
$$\mathbf A e(\hat u_n)=
\left(
\begin{array}{cc}
\partial _{22} \hat w_n & -\partial_{12}\hat w_n \\
-\partial_{12}\hat w_n & \partial _{11} \hat w_n
\end{array}
\right) \quad \text{ in }{{}}\Omega.$$
Since $\hat K_n \cap \overline B(x_n,r_n)=\hat \Gamma_n \cap \overline B(x_n,r_n)$ and $\overline B(x_n,r_n) \subset \Omega$, we infer that $\hat w_n \in H^2_{0,\hat K_n}(B(x_n,r_n))$. 
We rescale the Airy function $\hat w_n$ by setting, for all $n \in \mathbb N$ and a.e. $y \in B$,
$$w_n(y):=\frac{\hat w_n(x_n+r_n y)}{\sqrt{\omega_{\hat u_n}(x_n,r_n) r_n}}$$
in such a way that$w_n \in H^2_{0,K_n}(B)$,
$$\Delta^2 w_n= 0 \quad \text{ in }B \setminus K_n,$$
and
$$\mathbf A e(u_n)=
\left(
\begin{array}{cc}
\partial _{22}  w_n & -\partial_{12} w_n \\
-\partial_{12} w_n & \partial _{11}  w_n
\end{array}
\right).$$
In addition, since
$$\int_B |D^2 w_n|^2 \, dx = \int_B |\mathbf A e(u_n)|^2\, dx \leq C \int_B | e(u_n)|^2\, dx \leq C,$$
then Poincar\'e's inequality ensures that the sequence $(w_n)_{n \in \mathbb N}$ is bounded in $H^2(B)$, and thus, up to a subsequence $w_n \rightharpoonup w$ weakly in $H^2(B)$, for some $w \in H^2(B)$. A similar capacity argument than that used in the proof of \cite[Proposition 6.1]{BCL} shows that $w \in H^2_{0,T}(B(0,r))$ for all $r<1$, and
$$\Delta^2 w= 0 \quad \text{ in } B \setminus T,$$
and
\begin{equation}\label{eq:airy}
\mathbf A e(u)=
\left(
\begin{array}{cc}
\partial _{22}  w & -\partial_{12} w \\
-\partial_{12} w & \partial _{11}  w
\end{array}
\right).\end{equation}
In addition, using that the biharmonicity of $w_n$ is equivalent to the minimality
$$\int_B |D^2  w_n|^2\, dx \leq \int_B |D^2 z|^2\, dx$$
for all $z \in w_n +H^2_{0,K_n}(B)$, we can again reproduce the proof of \cite[Proposition 6.1]{BCL} to get that $ w_n \to w$ strongly in $H^2(B(0,r))$ for all $r<1$. In particular, it implies that $e(u_n) {\bf 1}_{B \setminus K_n}\to e(v)$ strongly in $L^2(B(0,r);\mathbb M^{2 \times 2}_{\rm sym})$, and thus passing to the limit in inequalities \eqref{eq:energy1} and \eqref{eq:omegaun} yields
\begin{equation}\label{eq:contradiction}
\omega_v(0,1)\leq 1\quad \text{ and }\quad \omega_v\left(0,a\right) \geq \tau_0.
\end{equation}
	
According to inequality \eqref{eq:contradiction}, we infer that 
$$\text{either }\quad {\frac1a}\int_{B(0,a) \cap \{x_2 >0\}} \mathbf A e(v):e(v)\, dx \geq \frac{\tau_0}{2}\quad
\text{ or }\quad\frac1a\int_{B(0,a)\cap\{x_2<0\}}\mathbf A e(v):e(v)\, dx \geq \frac{\tau_0}{2}.$$
Without loss of generality, we assume that 
\begin{equation}\label{eq:omegav2}
\frac1a\int_{B(0,a) \cap \{x_2 >0\}} \mathbf A e(v):e(v)\, dx \geq \frac{\tau_0}{2}.
\end{equation}
	
\medskip
	
\noindent {\bf Decay of the elastic energy.} We finally want to show a decay estimate on the elastic energy which will give a contradiction to \eqref{eq:omegav2}. To this aim, denoting by $B^\pm=B \cap \{\pm x_2>0\}$, we will work on the Airy function $w$ to construct an extension of $v|_{B^+}$ on $B$ which still solves the elasticity system in $B$.  According to (formula (3.28) in) \cite{Poritsky} (see also \cite{Duffin}), since $w \in \mathcal C^\infty(B^+)$ is a solution of
$$\Delta^2 w=0 \text{ in }B^+, \quad w=0, \nabla w=0 \text{ on } B \cap \{x_2=0\},$$
we can consider the biharmonic reflexion  $\tilde w \in \mathcal C^\infty(B)$ of $w|_{B^+}$ in $B$ defined by 
$$\tilde w(x)=
\begin{cases}
w(x) & \text{ if } x \in B^+,\\
-w(x_1,-x_2)-2x_2\partial_{2} w(x_1,-x_2) - x_2^2\Delta w(x_1,-x_2) & \text{ if }x \in B^-,
\end{cases}$$
which satisfies $\Delta^2 \tilde w=0$ in $B$. Thanks to this biharmonic extension, we are going to extend the displacement $v|_{B^+}$ on the whole ball $B$ into a function $\tilde v$ which minimizes the elastic energy. To this aim, let us define the stress by
$$\tilde \sigma:=
\begin{pmatrix}
\partial_{22} \tilde w & - \partial_{12} \tilde w\\
- \partial_{12} \tilde w & \partial_{11} \tilde w
\end{pmatrix},$$
and the strain
$$\tilde e:=
\begin{pmatrix}
\tilde e_{11} & \tilde e_{12}\\
\tilde e_{12} & \tilde e_{22}
\end{pmatrix}
:=\mathbf A^{-1} \tilde \sigma,$$
with
$$\tilde e_{11}= \frac{\tilde \sigma_{11}}{E} - \frac{\nu}{E} \tilde \sigma_{22},\quad
\tilde e_{22} = \frac{\tilde \sigma_{22}}{E} - \frac{\nu}{E} \tilde \sigma_{11}\quad
\tilde e_{12} = \frac{1+\nu}{E} \tilde \sigma_{12}.$$
Note that ${\rm div}\tilde \sigma=0$ in $B$, and the compatibility condition 
$$\partial_{22}\tilde e_{11}+\partial_{11}\tilde e_{22}-2\partial_{12}\tilde e_{12}=0 \quad \text{ in } B$$ ensures the existence of some $\tilde v \in \mathcal C^\infty(B;\R^2)$ such that $\tilde e=e(\tilde v)$ in $B$. In particular, according to \eqref{eq:airy}, we have
$$\mathbf A e(v)=\begin{pmatrix}
\partial_{22}  w & - \partial_{12}  w\\
- \partial_{12}  w & \partial_{11}  w
\end{pmatrix}
=\begin{pmatrix}
\partial_{22} \tilde w & - \partial_{12} \tilde w\\
- \partial_{12} \tilde w & \partial_{11} \tilde w
\end{pmatrix}
=\mathbf A e(\tilde v) \quad \text{ in } B^+,$$
which shows that $e(\tilde v)=e(v)$ in $B^+$, and thus that $v$ and $\tilde v$ only differ from a rigid body motion in $B^+$. We have thus constructed an extension $\tilde v$ of $v|_{B^+}$ which satisfies $-{\rm div}(\mathbf A e(\tilde v))=0$ in $B$, or equivalently,
$$\int_B \mathbf A e(\tilde v):e(\tilde v)\, dx \leq \int_B \mathbf A e(\tilde v+\varphi):e(\tilde v+\varphi)\, dx$$
for all $\varphi \in LD(B)$ such that $v=0$ on $\partial B$. 
	
\medskip
	
According to \eqref{eq:omegav2}, we have that
\begin{equation}\label{eq:a1}
\frac1a\int_{B(0,a)} \mathbf A e(\tilde v):e(\tilde v)\, dx \geq \frac{\tau_0}{2}.
\end{equation}
Moreover, by standards decay energy estimates for elliptic systems (see e.g. \cite[Proposition 3.4]{CFI2}), we infer that for all $\gamma\in (0,2)$, there exists a constant $c_\gamma=c(\gamma,\mathbf A)>0$ such that for all $r \leq 1$,
$$\int_{B(0,r)}\mathbf A e(\tilde v):e(\tilde v)\, dx \leq c_\gamma r^{2-\gamma} \int_{B} \mathbf A e(\tilde v):e(\tilde v)\, dx\leq c_\gamma r^{2-\gamma},$$
 where the last inequality comes from \eqref{eq:contradiction}, possibly changing $c_\gamma$. Taking $\gamma=1/2$ and $r=a$ yields
$$\frac1a\int_{B(0,a)}\mathbf A e(\tilde v):e(\tilde v)\, dx \leq c_{1/2} a^{1/2},$$
which is against \eqref{eq:a1} provided we choose $a<(\frac{\tau_0}{2c_{1/2}})^2$.
\hfill$\Box$


\section{Appendix}\label{app}
The following lemma is an easy consequence of the coarea formula.

\begin{lemma}\label{coaire1L} Let $K\subset \R^2$ be a $\mathcal H^1$-rectifiable set. Then for all $0<s<r$ and $x_0\in \R^2$ we have 
\begin{equation}
\int_{s}^{r}\#(K\cap \partial B(x_0,t))\, dt \leq \mathcal{H}^1(K\cap B(x_0,r)\setminus B(x_0,s)).\label{coaire1}
\end{equation}
\end{lemma} 

\begin{proof} 
Applying the coarea formula \cite[Theorem 2.91]{afp} to the $\mathcal H^1$-rectifiable set  $E:=K\cap B(x_0,r)\setminus B(x,s)$ and the Lipschitz function $f:x\mapsto |x|$  yields
$$\int_{s}^{r}\#(K\cap \partial B(x_0,t))\,dt=\int_{\R} \mathcal{H}^0(E\cap f^{-1}(t))  \; dt=\int_{E} {\bf J}d^E f \,d\mathcal{H}^1,$$
where, $\mathcal H^1$-a.e. in $E$,  ${\bf J}d^E f$ denotes the 1-dimensional coarea factor associated to the tangential differential $df^E$. Since $E$ admits an approximate tangent line oriented by a unit vector $\tau$ at $\mathcal H^1$-a.e. points, we deduce that
$${\bf J}d^E f_x = \left| \frac{x}{|x|}\cdot\tau \right| \leq 1 \quad \mathcal H^1\text{-a.e. in }E,$$
which leads to \eqref{coaire1}.
\end{proof}

We next recall a version of the Korn inequality in a rectangle.
\begin{lemma}[Korn's constant in a rectangle]\label{KornRect}
For $h\in (0,1)$, let $\Omega_h:=(-1,1)\times(-h,h)$ be a rectangle in $\R^2$ of height $2h$. There exists a constant $C>0$ (independent of $h$) such that for all $u \in LD(\Omega_h)$ one can find a skew symmetric matrix $R_h$ for which the following Korn inequality holds:
$$\int_{\Omega_h}|\nabla u-R_h|^2\,dx \leq \frac{C}{h^4} \int_{\Omega_h} |e(u)|^2 \,dx.$$
\end{lemma}
\begin{proof} For $u\in LD(\Omega_h)$ we define the function $v\in LD(\Omega_1)$ by
$$
\begin{cases}
v_1(x_1,x_2):=u_1(x_1,hx_2) \\
v_2(x_1,x_2):=hu_2(x_1,hx_2)
\end{cases}
\quad \text{ for a.e. }x=(x_1,x_2) \in \Omega_1.
$$
We note that 
$$\nabla v(x_1,x_2)=
\left(
\begin{array}{cc}
\partial_1 u_1 & h\partial_2 u_1   \\
h \partial_1 u_2  & h^2 \partial_2 u_2 
\end{array}
\right)(x_1,hx_2),
$$
so that 
$$e(v)(x_1,x_2)=
\left(
\begin{array}{cc}
e_{11}(u)  & h e_{12}(u)   \\
h e_{21}(u)  & h^2 e_{22} (u)
\end{array}
\right)(x_1,hx_2).
$$
Applying Korn's inequality to $v$ in $\Omega_1$ yields
$$\int_{\Omega_1}|\nabla v-R|^2\,dx \leq C \int_{\Omega_1} |e(v)|^2 \,dx,$$
for some skew symmetric matrix $R$, and where $C>0$ is the Korn constant in the unit cube.  In view of the above computations and using that $h \in (0,1)$, we deduce that 
$$ |e(v)(x_1,x_2)|^2\leq |e(u)(x_1,hx_2)|^2,$$
while using that $R_{11}=R_{22}=0$ and $R_{12}=-R_{21}$ we get
\begin{eqnarray*}
&&|\nabla v(x_1,x_2)-R|^2\\
&&\hspace{2cm} =\big(|\partial_1 u_1|^2+ |h\partial_1 u_2+R_{12}|^2+|h \partial_2 u_1-R_{12}|^2+h^4|\partial_2 u_2|^2\big) (x_1,hx_2) \notag \\
&&\hspace{2cm}= \big(|\partial_1 u_1|^2+h^2 |\partial_1 u_2+h^{-1}R_{12}|^2+h^2 |\partial_2 u_1-h^{-1}R_{12}|^2+h^4|\partial_2 u_2|^2\big)(x_1,hx_2)  \notag \\
&&\hspace{2cm}\geq h^4|\nabla u(x_1,hx_2)-R_h|^2,\notag
\end{eqnarray*}
where 
\renewcommand*{\arraystretch}{1.5}
$$R_h:=h^{-1}R=
\begin{pmatrix}
0  & h^{-1}R_{12}   \\
-h^{-1} R_{12}  & 0
\end{pmatrix}
$$
is still a skew symmetric matrix. We then obtain that
$$h^4\int_{\Omega_1}|\nabla u(x_1,hx_2)-R_h|^2\,dx \leq C \int_{\Omega_1} |e(u)(x_1,hx_2)|^2 \,dx.$$
Finally, using the change of variables $(y_1,y_2)=(x_1,hx_2)$ we get 
$$h^4\int_{\Omega_h}|\nabla u-R_h|^2\,dy \leq C \int_{\Omega_h} |e(u)|^2\,dy,$$
which completes the proof of the Lemma.
\end{proof}
 
The next lemma is a standard flatness estimate on curves coming from Pythagoras Theorem.
 
 \begin{lemma}\label{pythag} 
 Let $\gamma:[0,1]\to \R^2$ be a curve with endpoints $z=\gamma(0)$ and $z'=\gamma(1)$, with image $\Gamma:=\gamma([0,1])$. Then
  \begin{equation}
 {\rm dist}(y,[z,z'])^2 \leq \frac{\mathcal{H}^1(\Gamma) \big(\mathcal{H}^1(\Gamma)-|z'-z|\big)}{2}  \quad \text{for all } y\in \Gamma. \label{estimgeom}
 \end{equation} 
 \end{lemma}
 
\begin{proof}
Let $\bar y$ be a maximizer of the function $y \in \Gamma \mapsto {\rm dist}(y,[z,z'])$, i.e., $\bar y$ is the most distant point in $\Gamma$ to the segment $[z,z']$, and define $d=:{\rm dist}(\bar y,[z,z'])$. Let us consider the point $y' \in \R^2$ making $(z,z',y')$ an isosceles triangle with same height $d$  (See Figure \ref{fig5}). Denoting by $a:=|z-z'|/2 $ and $L := |y'-z|$, according to Pythagoras Theorem, we have 
 $$d^2 = L^2-a^2=(L -a) (L +a).$$
	
\begin{figure}[!ht]
\begin{center}
\scalebox{1} 
{
\begin{pspicture}(0,-2.2388477)(10.72291,2.2138476)
\definecolor{color1685}{rgb}{0.4,0.4,0.4}
\definecolor{color1028}{rgb}{0.6,0.6,0.6}
\definecolor{color1608}{rgb}{0.4,0.4,1.0}
\psline[linewidth=0.03cm](2.2810156,0.65884763)(1.3210156,-1.6011523)
\psline[linewidth=0.03cm,linecolor=color1028](1.3210156,-1.6211524)(4.4210157,1.2388476)
\psdots[dotsize=0.16](1.3010156,-1.6011523)
\psdots[dotsize=0.16](8.381016,0.35884765)
\usefont{T1}{ptm}{m}{n}
\rput(1.3724707,-2.0161524){$z=\gamma(0)$}
\usefont{T1}{ptm}{m}{n}
\rput(0.7624707,-0.5161523){\color{color1608}$\Gamma$}
\psline[linewidth=0.03cm](1.3410156,-1.5811523)(8.381016,0.35884765)
\psline[linewidth=0.03cm](2.2810156,0.63884765)(8.361015,0.33884767)
\usefont{T1}{ptm}{m}{n}
\rput(4.352471,-0.016152345){\color{color1685}$d$}
\usefont{T1}{ptm}{m}{n}
\rput(2.0224707,0.90384763){$\bar y$}
\psline[linewidth=0.03cm](0.8610156,0.25884765)(7.7610154,2.1988478)
\psline[linewidth=0.03cm,linecolor=color1028](4.4410157,1.2388476)(8.401015,0.35884765)
\psline[linewidth=0.03cm,linecolor=color1028,arrowsize=0.05291667cm 2.0,arrowlength=1.4,arrowinset=0.4]{<->}(4.4410157,1.1788477)(4.9810157,-0.54115236)
\usefont{T1}{ptm}{m}{n}
\rput(3.9624708,1.4438477){$y'$}
\usefont{T1}{ptm}{m}{n}
\rput(6.702471,-0.61615235){\color{color1685}$a$}
\psline[linewidth=0.03cm,linecolor=color1028,arrowsize=0.05291667cm 2.0,arrowlength=1.4,arrowinset=0.4]{<->}(5.0610156,-0.7011523)(8.441015,0.25884765)
\usefont{T1}{ptm}{m}{n}
\rput(9.502471,-0.016152345){$z'=\gamma(1)$}
\psdots[dotsize=0.2](4.4210157,1.2588477)
\psdots[dotsize=0.16](2.2610157,0.63884765)
\psline[linewidth=0.03cm,linecolor=color1028,arrowsize=0.05291667cm 2.0,arrowlength=1.4,arrowinset=0.4]{<->}(4.5210156,1.3588476)(8.321015,0.47884765)
\usefont{T1}{ptm}{m}{n}
\rput(6.0724707,1.2638477){\color{color1685}$L$}
\pscustom[linewidth=0.04,linecolor=color1608]
{
\newpath
\moveto(1.2810156,-1.6211524)
\lineto(1.2710156,-1.5711523)
\curveto(1.2660156,-1.5461524)(1.2510157,-1.5011524)(1.2410157,-1.4811523)
\curveto(1.2310157,-1.4611523)(1.2160156,-1.4211524)(1.2110156,-1.4011524)
\curveto(1.2060156,-1.3811524)(1.1960156,-1.3361523)(1.1910156,-1.3111523)
\curveto(1.1860156,-1.2861524)(1.1810156,-1.2361523)(1.1810156,-1.2111523)
\curveto(1.1810156,-1.1861523)(1.1760156,-1.1411524)(1.1710156,-1.1211524)
\curveto(1.1660156,-1.1011523)(1.1560156,-1.0611523)(1.1510156,-1.0411524)
\curveto(1.1460156,-1.0211524)(1.1410156,-0.97615236)(1.1410156,-0.9511523)
\curveto(1.1410156,-0.92615235)(1.1410156,-0.87615234)(1.1410156,-0.85115236)
\curveto(1.1410156,-0.8261523)(1.1410156,-0.7761524)(1.1410156,-0.75115234)
\curveto(1.1410156,-0.72615236)(1.1410156,-0.67615235)(1.1410156,-0.6511524)
\curveto(1.1410156,-0.62615234)(1.1410156,-0.5761523)(1.1410156,-0.55115235)
\curveto(1.1410156,-0.5261524)(1.1510156,-0.47615233)(1.1610156,-0.45115235)
\curveto(1.1710156,-0.42615235)(1.1910156,-0.37615234)(1.2010156,-0.35115233)
\curveto(1.2110156,-0.32615235)(1.2310157,-0.27615234)(1.2410157,-0.25115234)
\curveto(1.2510157,-0.22615235)(1.2760156,-0.18115234)(1.2910156,-0.16115235)
\curveto(1.3060156,-0.14115234)(1.3360156,-0.101152346)(1.3510156,-0.08115234)
\curveto(1.3660157,-0.061152343)(1.3910156,-0.026152344)(1.4010156,-0.011152344)
\curveto(1.4110156,0.0038476563)(1.4360156,0.038847655)(1.4510156,0.058847655)
\curveto(1.4660156,0.078847654)(1.5010157,0.11884765)(1.5210156,0.13884765)
\curveto(1.5410156,0.15884766)(1.5810156,0.19384766)(1.6010156,0.20884766)
\curveto(1.6210157,0.22384766)(1.6560156,0.25384766)(1.6710156,0.26884764)
\curveto(1.6860156,0.28384766)(1.7210156,0.30384767)(1.7410157,0.30884767)
\curveto(1.7610157,0.31384766)(1.7960156,0.33384764)(1.8110156,0.34884766)
\curveto(1.8260156,0.36384764)(1.8560157,0.38884765)(1.8710157,0.39884767)
\curveto(1.8860157,0.40884766)(1.9210156,0.43384767)(1.9410156,0.44884765)
\curveto(1.9610156,0.46384767)(2.0010157,0.48384765)(2.0210156,0.48884764)
\curveto(2.0410156,0.49384767)(2.0810156,0.50884765)(2.1010156,0.51884764)
\curveto(2.1210155,0.52884763)(2.1610155,0.5488477)(2.1810157,0.55884767)
\curveto(2.2010157,0.56884766)(2.2460155,0.57884765)(2.2710156,0.57884765)
\curveto(2.2960157,0.57884765)(2.3460157,0.57384765)(2.3710155,0.56884766)
\curveto(2.3960156,0.56384766)(2.4410157,0.5538477)(2.4610157,0.5488477)
\curveto(2.4810157,0.5438477)(2.5210156,0.52384764)(2.5410156,0.50884765)
\curveto(2.5610156,0.49384767)(2.6010156,0.46884766)(2.6210155,0.45884764)
\curveto(2.6410155,0.44884765)(2.6760156,0.42884767)(2.6910157,0.41884765)
\curveto(2.7060156,0.40884766)(2.7360156,0.38384765)(2.7510157,0.36884767)
\curveto(2.7660155,0.35384765)(2.8010156,0.32884765)(2.8210156,0.31884766)
\curveto(2.8410156,0.30884767)(2.8810155,0.28884766)(2.9010155,0.27884766)
\curveto(2.9210157,0.26884764)(2.9560156,0.24884766)(2.9710157,0.23884766)
\curveto(2.9860156,0.22884765)(3.0160155,0.20384766)(3.0310156,0.18884766)
\curveto(3.0460157,0.17384766)(3.0760157,0.14884765)(3.0910156,0.13884765)
\curveto(3.1060157,0.12884766)(3.1360157,0.10384765)(3.1510155,0.08884766)
\curveto(3.1660156,0.07384766)(3.1960156,0.043847658)(3.2110157,0.028847657)
\curveto(3.2260156,0.013847657)(3.2560155,-0.016152345)(3.2710156,-0.031152343)
\curveto(3.2860155,-0.046152342)(3.3260157,-0.071152344)(3.3510156,-0.08115234)
\curveto(3.3760157,-0.09115234)(3.4110155,-0.116152346)(3.4210157,-0.13115235)
\curveto(3.4310157,-0.14615235)(3.4660156,-0.17115234)(3.4910157,-0.18115234)
\curveto(3.5160155,-0.19115235)(3.5610156,-0.21115234)(3.5810156,-0.22115235)
\curveto(3.6010156,-0.23115234)(3.6410155,-0.24615234)(3.6610155,-0.25115234)
\curveto(3.6810157,-0.25615233)(3.7210157,-0.27115235)(3.7410157,-0.28115234)
\curveto(3.7610157,-0.29115236)(3.8060157,-0.30615234)(3.8310156,-0.31115234)
\curveto(3.8560157,-0.31615233)(3.9010155,-0.33115235)(3.9210157,-0.34115234)
\curveto(3.9410157,-0.35115233)(3.9810157,-0.36615235)(4.0010157,-0.37115234)
\curveto(4.0210156,-0.37615234)(4.0660157,-0.38615236)(4.091016,-0.39115235)
\curveto(4.1160154,-0.39615235)(4.1610155,-0.40615234)(4.1810155,-0.41115233)
\curveto(4.2010155,-0.41615236)(4.2460155,-0.42615235)(4.2710156,-0.43115234)
\curveto(4.2960157,-0.43615234)(4.3460155,-0.44615233)(4.3710155,-0.45115235)
\curveto(4.3960156,-0.45615235)(4.4510155,-0.46115234)(4.4810157,-0.46115234)
\curveto(4.5110154,-0.46115234)(4.5660157,-0.46615234)(4.591016,-0.47115234)
\curveto(4.6160154,-0.47615233)(4.6660156,-0.47115234)(4.6910157,-0.46115234)
\curveto(4.716016,-0.45115235)(4.7610154,-0.43115234)(4.7810154,-0.42115235)
\curveto(4.801016,-0.41115233)(4.841016,-0.38115233)(4.861016,-0.36115235)
\curveto(4.881016,-0.34115234)(4.9210157,-0.31115234)(4.9410157,-0.30115235)
\curveto(4.9610157,-0.29115236)(4.9910154,-0.26115236)(5.0010157,-0.24115235)
\curveto(5.0110154,-0.22115235)(5.0410156,-0.17615235)(5.0610156,-0.15115234)
\curveto(5.0810156,-0.12615234)(5.1060157,-0.07615235)(5.111016,-0.051152345)
\curveto(5.1160154,-0.026152344)(5.131016,0.018847656)(5.1410155,0.038847655)
\curveto(5.1510158,0.058847655)(5.1610155,0.10384765)(5.1610155,0.12884766)
\curveto(5.1610155,0.15384765)(5.1660156,0.19884765)(5.1710157,0.21884766)
\curveto(5.176016,0.23884766)(5.1860156,0.28384766)(5.1910157,0.30884767)
\curveto(5.196016,0.33384764)(5.2010155,0.38384765)(5.2010155,0.40884766)
\curveto(5.2010155,0.43384767)(5.2010155,0.48384765)(5.2010155,0.50884765)
\curveto(5.2010155,0.53384763)(5.1910157,0.57384765)(5.1810155,0.58884764)
\curveto(5.1710157,0.6038477)(5.1460156,0.63884765)(5.131016,0.65884763)
\curveto(5.1160154,0.6788477)(5.0810156,0.70384765)(5.0610156,0.70884764)
\curveto(5.0410156,0.71384764)(4.9960155,0.71384764)(4.9710155,0.70884764)
\curveto(4.946016,0.70384765)(4.9110155,0.68384767)(4.9010158,0.6688477)
\curveto(4.8910155,0.65384763)(4.881016,0.6138477)(4.881016,0.58884764)
\curveto(4.881016,0.56384766)(4.8960156,0.51384765)(4.9110155,0.48884764)
\curveto(4.926016,0.46384767)(4.9510155,0.41884765)(4.9610157,0.39884767)
\curveto(4.9710155,0.37884766)(5.0010157,0.33884767)(5.0210156,0.31884766)
\curveto(5.0410156,0.29884765)(5.0810156,0.25384766)(5.1010156,0.22884765)
\curveto(5.1210155,0.20384766)(5.1610155,0.16884765)(5.1810155,0.15884766)
\curveto(5.2010155,0.14884765)(5.2460155,0.13384765)(5.2710156,0.12884766)
\curveto(5.2960157,0.123847656)(5.3460155,0.11884765)(5.3710155,0.11884765)
\curveto(5.3960156,0.11884765)(5.4510155,0.11884765)(5.4810157,0.11884765)
\curveto(5.5110154,0.11884765)(5.5760155,0.11884765)(5.611016,0.11884765)
\curveto(5.6460156,0.11884765)(5.7010155,0.123847656)(5.7210155,0.12884766)
\curveto(5.7410154,0.13384765)(5.7760158,0.14884765)(5.7910156,0.15884766)
\curveto(5.8060155,0.16884765)(5.8460155,0.18384765)(5.8710155,0.18884766)
\curveto(5.8960156,0.19384766)(5.9510155,0.20384766)(5.9810157,0.20884766)
\curveto(6.0110154,0.21384765)(6.0660157,0.22384766)(6.091016,0.22884765)
\curveto(6.1160154,0.23384766)(6.1610155,0.24384765)(6.1810155,0.24884766)
\curveto(6.2010155,0.25384766)(6.2410154,0.26884764)(6.2610154,0.27884766)
\curveto(6.2810154,0.28884766)(6.3260155,0.30884767)(6.3510156,0.31884766)
\curveto(6.3760157,0.32884765)(6.4210157,0.34884766)(6.4410157,0.35884765)
\curveto(6.4610157,0.36884767)(6.506016,0.38884765)(6.5310154,0.39884767)
\curveto(6.5560155,0.40884766)(6.6060157,0.42884767)(6.631016,0.43884766)
\curveto(6.6560154,0.44884765)(6.7060156,0.46884766)(6.7310157,0.47884765)
\curveto(6.756016,0.48884764)(6.8060155,0.50884765)(6.8310156,0.51884764)
\curveto(6.8560157,0.52884763)(6.9060154,0.5488477)(6.9310155,0.55884767)
\curveto(6.9560156,0.56884766)(7.0010157,0.58384764)(7.0210156,0.58884764)
\curveto(7.0410156,0.59384763)(7.0860157,0.6088477)(7.111016,0.61884767)
\curveto(7.1360154,0.62884766)(7.1860156,0.63884765)(7.2110157,0.63884765)
\curveto(7.236016,0.63884765)(7.2910156,0.64384764)(7.321016,0.64884764)
\curveto(7.3510156,0.65384763)(7.4110155,0.6638477)(7.4410157,0.6688477)
\curveto(7.4710155,0.6738477)(7.5260158,0.6788477)(7.551016,0.6788477)
\curveto(7.5760155,0.6788477)(7.6260157,0.6788477)(7.6510158,0.6788477)
\curveto(7.676016,0.6788477)(7.7260156,0.6788477)(7.7510157,0.6788477)
\curveto(7.7760158,0.6788477)(7.8260155,0.68384767)(7.8510156,0.68884766)
\curveto(7.8760157,0.69384766)(7.926016,0.69384766)(7.9510155,0.68884766)
\curveto(7.9760156,0.68384767)(8.031015,0.6738477)(8.061016,0.6688477)
\curveto(8.091016,0.6638477)(8.146015,0.65384763)(8.171016,0.64884764)
\curveto(8.196015,0.64384764)(8.241015,0.62884766)(8.261016,0.61884767)
\curveto(8.281015,0.6088477)(8.326015,0.59384763)(8.351016,0.58884764)
\curveto(8.376016,0.58384764)(8.416016,0.56884766)(8.431016,0.55884767)
\curveto(8.446015,0.5488477)(8.471016,0.51884764)(8.481015,0.49884766)
\curveto(8.491015,0.47884765)(8.496016,0.43884766)(8.491015,0.41884765)
\curveto(8.486015,0.39884767)(8.466016,0.36884767)(8.451015,0.35884765)
\curveto(8.436016,0.34884766)(8.401015,0.33884767)(8.341016,0.33884767)
}
\end{pspicture} 
}
\end{center}
\caption{{The height estimate from Pythagoras Theorem.}}	
\label{fig5}
\end{figure}
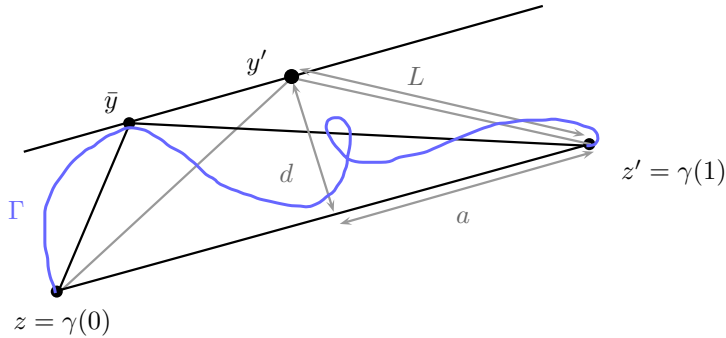

Thus $\mathcal{H}^1(\Gamma)\geq |z-\bar y| + |\bar y - z'|\geq 2L$ and $\mathcal H^1(\Gamma)\geq |z-z'|$ so that
\[d^2\leq \frac{1}{4}\left(\mathcal{H}^1(\Gamma )-|z-z'|\right) \left(\mathcal{H}^1(\Gamma )+|z-z'|\right)\leq  \frac{\mathcal{H}^1(\Gamma )\big(\mathcal{H}^1(\Gamma )-|z-z'|\big)}{2} ,\]
which proves \eqref{estimgeom}.
 \end{proof}

We conclude the appendix with the following  standard Lemma (see for instance \cite[proof of Corollary 33.50]{d}, \cite[proof of Theorem 5.5.]{dlm}, \cite[Section 10]{ddt} for a non-exhaustive list of similar results). Unfortunately, we could not find a precise reference   in the exact following elementary form, thus we provide an independent and complete proof for the reader's convenience.
 
 \begin{lemma}\label{c1estimates} 
 Let $K\subset \R^2$ be a closed set containing the origin and satisfying the following propery: there exist constants $C>0$, $r_0>0$ and $\alpha >0$ such that
$$\beta_K(x,r)\leq Cr^\alpha \quad \quad \text{for all } x\in K\cap B(0,1) \text{ and all } r\leq r_0.$$
Then there exists $a\in (0,1)$ (only depending on $C$, $r_0$, and $\alpha$) such that  $K\cap B(0,a)$ is a $10^{-2}$-Lipschitz graph, as well as a $\mathcal{C}^{1,\alpha}$ regular curve.
 \end{lemma}
	 
\begin{proof}  For every $x\in K \cap B(0,1)$ and $0<r\leq r_0$, we denote as usual by $L(x,r)$ an affine line which approximates $K\cap B(x,r)$, i.e. such that 
\begin{equation}
\max \Big\{ \sup_{z\in K\cap \overline{B}(x,r)}{\rm dist }(z,L(x,r)), \sup_{z\in L(x,r)\cap \overline{B}(x,r)}{\rm dist }(z,K) \Big\}  \leq \beta(x,r)r \leq Cr^{1+\alpha}. \label{betabeta}
\end{equation}
In addition, we denote by $\tau(x,r) \in \mathbb{S}^1/\{\pm 1\}$ a non-oriented unit vector which is tangent to $L(x,r)$ and defined modulo $\pm1$.   We use in $\mathbb{S}^1/\{\pm 1\}$ the complete distance defined, for all $\tau_1,\tau_2 \in \mathbb{S}^1/\{\pm 1\}$, by
$$d_S(\tau_1,\tau_2):=\min(|\tau_1-\tau_2|,|\tau_1+\tau_2|).$$	  
	  
\medskip

\noindent \emph{Step 1. Existence of tangents.}   For all $k\in \mathbb{N}$ we denote by $r_k:=2^{-k}r_0$.  We claim that the sequence $(\tau(x,r_k))_{k \in \mathbb N}$ is a Cauchy sequence in $(\mathbb{S}^1/\{\pm 1\},d_S)$. For that purpose, we show that for all $k\geq 0$, and all $x\in K\cap B(0,1)$ we have
$$d_S\big(\tau(x,r_{k+1}),\tau(x,r_k)\big)\leq 9Cr_k^\alpha.$$
Indeed, let $z:= x+ \tau(x,r_{k+1})r_{k+1}\in L(x,r_{k+1})   \cap \overline B(x,r_{k+1})$. Because of \eqref{betabeta}, there exists $y \in K\cap \overline{B}(x,r_{k+1})$ such that $|z-y|\leq Cr_{k+1}^{1+\alpha}$ and in particular, 
\begin{equation}
r_{k+1}-Cr_{k+1}^{1+\alpha}\leq |y-x|\leq r_{k+1}. \label{estimnorm}
\end{equation}
Then, if we denote by $v:=\frac{y-x}{|y-x|}$, we have that  
\begin{eqnarray}
d_S\big(v,\tau(x,r_{k+1})\big)\leq |v-\tau(x,r_{k+1})|&=&\left|\frac{y-x}{|y-x|} - \frac{z-x}{r_{k+1}}\right| \notag \\
&\leq &\left|\frac{y-x}{|y-x|} - \frac{y-x}{r_{k+1}}\right|
 +\frac{1}{r_{k+1}}|z-y| \notag \\
&\leq & \frac{|r_{k+1}-|y-x||}{r_{k+1}} + Cr_{k+1}^\alpha \notag \\
&\leq & 2Cr_{k+1}^\alpha, \label{estimDS}
\end{eqnarray}
where we used \eqref{estimnorm} to get the last inequality.  Similarly, since   $y\in \overline B(x, r_{k+1}) \cap K \subset \overline B(x, r_{k}) \cap K$, there exists $z'\in L(x,r_{k})   \cap \overline B(x,r_k)$ such that $|y-z'|\leq Cr_{k}^{1+\alpha}$. By \eqref{estimnorm} again we can estimate 
$$ |z'-x|\leq |y-x|+|z'-y|\leq r_{k+1}+Cr_{k}^{1+\alpha}$$
and 
$$|z'-x|\geq  |y-x|-|z'-y|\geq r_{k+1}-Cr_{k+1}^{1+\alpha}- Cr_{k}^{1+\alpha}\geq r_{k+1} - 2Cr_{k}^{1+\alpha},$$
thus a computation similar to the one of \eqref{estimDS} leads to 
\begin{eqnarray}
d_S\big(v,\tau(x,r_{k})\big)\leq \left|v-\frac{z'-x}{|z'-x|}\right|&=&\left|\frac{y-x}{|y-x|} - \frac{z'-x}{|z'-x|}\right| \notag \\
&\leq & \left|\frac{y-x}{|y-x|} - \frac{y-x}{r_{k+1}}\right|+\left|\frac{y-x}{r_{k+1}} - \frac{z'-x}{r_{k+1}}\right|+\left|\frac{z'-x}{|z'-x|} - \frac{z'-x}{r_{k+1}}\right|   \notag.\\
&\leq & Cr_{k+1}^\alpha + C\frac{r_k^{1+\alpha}}{r_{k+1}} + 2C\frac{r_k^{1+\alpha}}{r_{k+1}}\leq 7Cr_k^\alpha. \label{estimDS2}
\end{eqnarray}
Gathering both the above inequalities, we  obtain
$$d_S(\tau(x,r_{k}),\tau(x,r_{k+1}))\leq d_S(\tau(x,r_{k}),v)+d_S(v,\tau(x,r_{k+1})) \leq 9Cr_k^\alpha=9Cr_0^{\alpha}2^{-k\alpha},$$
as claimed. It follows that for all $k,l\geq k_0$, 
	  
$$d_S\big(\tau(x,r_{k}),\tau(x,r_{l})\big) \leq \sum_{i=k_0}^{+\infty} 9Cr_0^{\alpha}2^{-i\alpha}= 2^{-k_0\alpha}\left(\frac{9Cr_0^{\alpha}}{1-2^{-\alpha}}\right).$$
Since the latter can be made arbitrarily small provided $k_0$ is large enough, we deduce that $\tau(x,r_k)$ is a Cauchy sequence in $\mathbb{S}^1/\{\pm 1\}$, and therefore, it converges to some vector denoted by $\tau(x)$. In particular, letting $l\to +\infty$, we get the following estimate for all $k\geq 0$	 
$$ d_S\big(\tau(x,r_{k}),\tau(x)\big) \leq C' r_k^\alpha,$$
where 
$$C':=\frac{9C}{1-2^{-\alpha}}.$$	
Moreover, it can be easily seen through the distance estimate \eqref{betabeta}, that $T_x:=x+\mathbb{R}\tau(x)$ is a  tangent line for the set $K$ at the point $x$.

\medskip

\noindent \emph{Step 2. H\"older estimate for tangents.} We now prove that the mapping $x\mapsto \tau(x)$ is H\"older continuous.  Let $x$ and $y$ be two different points of  $K \cap B(0,1)$ and let $\rho:=|y-x|$. Assume first that $\rho\leq r_0/4$ and let $k\in \mathbb{N}$ be such that 
$$r_{k+2}\leq \rho\leq r_{k+1}.$$
We have that
\begin{eqnarray}
d_S\big(\tau(x),\tau(y)\big)&\leq &d_S\big(\tau(x),\tau(x,r_k)\big)+d_S\big(\tau(x,r_{k}),\tau(y,r_{k})\big)+d_S\big(\tau(y,r_{k}),\tau(y)\big)\notag \\
&\leq & 2C'r_k^\alpha+ d_S\big(\tau(x,r_{k}),\tau(y,r_{k})\big). \label{etape11}
\end{eqnarray}
In order to estimate $d_S\big(\tau(x,r_{k}),\tau(y,r_{k})\big)$, we notice that  $y\in \overline B(x,r_k) \cap K$, thus there exists { $z \in L(x,r_{k}) \cap \overline B(x,r_k)$} such that $|y-z|\leq Cr_{k}^{1+\alpha}$.   Let us set $v:=\frac{y-x}{|y-x|}$, so that a computation similar to that of \eqref{estimDS} or \eqref{estimDS2} leads to
$$d_S(v,\tau(x,r_{k}))\leq 8Cr_{k}^\alpha,$$
and inverting the roles of $x$ and $y$
$$d_S(v,\tau(y,r_{k}))\leq 8Cr_{k}^\alpha.$$
Turning back to \eqref{etape11}, we deduce that 
\begin{equation}
d_S\big(\tau(x),\tau(y)\big)\leq 2C'r_k^\alpha +16Cr_k^\alpha \leq  16(C'+C)2^{2\alpha} r_{k+2}^\alpha \leq 4^{\alpha+2}(C'+C) |x-y|^\alpha.\label{holderestim0}
\end{equation}

In the case when $\rho\geq r_0/4$, we can simply estimate 
$$d_S\big(\tau(x),\tau(y)\big) \leq 2 \leq 2\frac{4^\alpha}{r_0^\alpha}|x-y|^\alpha,$$
which finally yields, for general $x,y \in K\cap B(0,1)$,
\begin{eqnarray} 
d_S\big(\tau(x),\tau(y)\big)\leq C'' |x-y|^\alpha, \label{holderestim}
\end{eqnarray}
with $C'':=\max\big(4^{\alpha+1} r_0^{-\alpha} ,4^{\alpha+2}(C'+C)\big)$.

In other words, we have proved that $K$ admits a tangent everywhere on $B(0,1)$ and that tangent lines behave nicely. We will prove now that $K\cap \overline B(0,a)$ is a  curve for $a$ small enough. Actually, a convenient way to prove this is to show the  stronger property that $K\cap \overline B(0,a)$ is a  Lipschitz graph for some $a\in(0,1)$ small enough.

\medskip

\noindent \emph{Step 3. $K\cap \overline B(0,a)$ is a Lipschitz graph.} We first show that for $a>0$ small enough (to be fixed later), the set $K\cap \overline B(0,a)$ is a  graph above the line $\R\tau(0)$, that we assume for simplicity to be oriented by $e_1:=\tau(0)$. Notice that for all $x\in K\cap \overline B(0,a)$,
\begin{equation}
d_S(\tau(x),e_1)\leq  C''a^{\alpha}, \label{Reiff}
\end{equation}
which means that for $a$ small, all the tangents are oriented almost horizontally in $K\cap \overline B(0,a)$.

We assume by contradiction that there exist two distinct points $x,y \in K\cap \overline B(0,a)$ such that $x_1=y_1$. Let $a\leq r_0/10$,
$\rho:=10|x-y|=10|x_2-y_2|\leq 20a$,  and let $k \in \mathbb N$ be such that 
$$r_{k+1}\leq \rho \leq r_k.$$
We denote by $\gamma_k\in [0,\pi/2]$ the angle between $e_1$ and $\tau(x,r_k)$. Since $$d_S(e_1,\tau(x,r_k))\leq C''a^\alpha+C'r_k^\alpha\leq C''a^\alpha+C'(40a)^\alpha,$$
for $a$ small enough it is not restrictive to assume $\gamma_k\in [0,\pi/4]$. We deduce that
$${\rm dist}(y,T_x)\leq \frac{{\rm dist}(y,L(x,r_k))}{\cos \gamma_k}\leq\sqrt{2}\,{\rm dist}(y,L(x,r_k))\leq \sqrt{2}Cr_{k}^{1+\alpha}\leq \sqrt{2}\, C (40a)^\alpha r_k.$$
Similarly, if $\gamma\in [0,\pi/2]$ stands for the angle between $e_1$ and $\tau(x)$, we have for $a$ small enough and for a universal constant $C'''>0$
$$|x_2-y_2|=|x-y|=\frac{{\rm dist}(y, T_x)}{\cos\gamma}\leq	2\, C (40a)^\alpha r_k \leq a^{\alpha}C'''|x_2-y_2|,$$
which is a contradiction for $a$ small enough (depending on $C'''$). { Therefore, $K\cap \overline B(0,a)$ must be a graph above the segment $\overline B(0,a) \cap \tau(0) \R$ identified to $[-a,a]$.} Now to prove that the graph is $10^{-3}$-Lipschitz for $a$ small enough, we can reproduce the same argument but  for $x,y \in K\cap \overline B(0,a)$ satisfying  now, by contradiction, $|x_2-y_2|> 10^{-3}|x_1-y_1|$.  
   
\medskip

\noindent \emph{Step 4. Conclusion.} We have proved that $K\cap \overline B(0,a)$ is the $10^{-3}$-Lipschitz graph of some function $f$ on $[-a,a]$. Moreover, the tangent line to the graph of $f$ at the point $(t,f(t))$, which exists for a.e.  $t\in [-a,a]$, coincides with  the tangent line $x+\R \tau(x)$ to $K$ at the point $x=(t,f(t))$. Since the map $x\mapsto \tau(x)$ is $\alpha$-H\"older continuous, it follows that the map $t\mapsto f'(t)$ coincides a.e. on $[-a,a]$ with an $\alpha$-H\"older continuous function. A smoothing argument then implies that $f\in \mathcal{C}^{1,\alpha}([-a,a])$, and $K\cap \overline B(0,a)$ is a $\mathcal{C}^{1,\alpha}$ curve.
\end{proof} 

\section*{Acknowledgements}
F.I. has been a recipient of scholarships from the Fondation Sciences Math\'ematiques de Paris, Emergence Sorbonne Universit\'es and from the S\'ephora-Berrebi Foundation. She gratefully acknowledges their support.

\end{document}